\documentclass[11pt]{amsart}
\usepackage{a4wide,xcolor,eucal,enumerate,mathrsfs}
\usepackage[normalem]{ulem}
\usepackage{amsmath,amssymb,amsthm}
\usepackage{epsfig}
\newcommand\N{\mathbb{N}}
\newcommand\R{\mathbb{R}}

\newcommand\Q{\mathbb{Q}}

\def\L{{\mathcal L}}
\def\A{{\mathcal A}}

\newcommand{\forget}[1]{}

\def\Om{{\Omega}}  
\def\om2{{\Om\times\Om}}
\def\M{\mathcal{M}}

\def\B{{\mathcal B}}

\def\supp{\mathrm{supp}\,}

\def\div{\mathrm{div}\,}

\def\Lip{\mathrm{Lip}}
\def\lip{\mathrm{lip}}

\newcommand{\res}{\llcorner} 

\newcommand{\MM}{\mathbb M}
\usepackage{color}

\newcommand{\ld}{[\![}
\newcommand{\rd}{]\!]}

\newcommand{\nor}[1]{\left\| #1 \right\|}

\newtheorem{theorem}{Theorem}[section]

\newtheorem{lemma}[theorem]{Lemma}
\newtheorem{proposition}[theorem]{Proposition}
\newtheorem{corollary}[theorem]{Corollary}

\theoremstyle{remark}
\newtheorem{remark}[theorem]{Remark}

\numberwithin{equation}{section}

\newcommand{\1}{\mathbf{1}}

\newcommand{\cur}[1]{\left\{ #1 \right\}}

\begin{document}

\title[Three superposition principles] {Three superposition principles: currents, continuity equations and curves of measures}

\author{Eugene Stepanov}%
\address{
St.Petersburg Branch
of the Steklov Mathematical Institute of the Russian Academy of Sciences,
Fontanka 27,
191023 St.Petersburg,
Russia
\and
Department of Mathematical Physics, Faculty of Mathematics and Mechanics,
St. Petersburg State University, Universitetskij pr.~28, Old Peterhof,
198504 St.Petersburg, Russia
\and ITMO University
}
\email{stepanov.eugene@gmail.com}
\thanks{The work of the first author was financed by GNAMPA, by RFBR grant \#14-01-0053,
  by the project
2010A2TFX2 ``Calcolo delle variazioni''
 of the
Italian Ministry of Research and by St.Petersburg State
University grants \#6.38.670.2013 and \#6.38.223.2014 and by the Russian government grant NSh-1771.2014.1. The second author has been partially supported by PRIN10-11 grant from MIUR
for the project Calculus of Variations and is a member of GNAMPA}

\author{Dario Trevisan}%
\address{ Universit\`a degli Studi di Pisa,
 Dipartimento di Matematica,
 Largo Bruno Pontecorvo 5,
 56127 Pisa, Italy}
\email{dario.trevisan@unipi.it}

\begin{abstract}
We establish a general superposition principle for curves of measures solving a continuity equation on metric spaces
without any smooth structure nor underlying measure, representing them as marginals of measures concentrated on the solutions of the associated ODE defined by some algebra of observables. We relate this result with decomposition of acyclic normal currents in metric spaces.

As an application, a slightly extended version of a probabilistic representation for absolutely continuous curves in Kantorovich-Wasserstein spaces, originally due to S.\ Lisini, is provided in the metric framework. This gives a hierarchy of implications between \emph{superposition principles} for curves of measures and for metric currents.
\end{abstract}


\maketitle

\section{Introduction}

Recently, a lot of results appeared in the literature regarding the structure of curves of measures in metric spaces, all referred to as ``superposition principles''. In particular, this includes
\begin{itemize}
  \item[(A)] the results relating the family of measures solving the first order PDE called \emph{continuity equation},
  with the measures on the space of solutions to the associated characteristic ODE;
  \item[(B)] the results giving the structure of curves of positive Borel measures
  absolutely continuous with respect to some Kantorovich-Wasserstein distance, in terms of
  measures over some space of rectifiable curves, and, finally,
  \item[(C)] the results on decomposition (or, more appropriately, ``disintegration'') of general normal currents (in the sense of De Rham or in the sense of De Giorgi-Ambrosio-Kirchheim) in simpler
  ones associated with rectifiable curves.
\end{itemize}

The results of type~(A) provide a link between Euler (continuity equation) and the Lagrange (characteristic ODE) descriptions of the flow of measures and can be traced back to  theory of ``generalized curves'' by L.C.\ Young~\cite{Young-37}.
For equations in Euclidean spaces such a result has been  proven in~\cite[\S 8]{Ambrosio-Gigli-Savare-05} and
further implemented in many different contexts, in particular with respect to DiPerna-Lions uniqueness theory (see~\cite{ambrosio-crippa-edinburgh} for a recent overview), in the theory of stochastic processes, e.g.\ in the theory of c\`adl\`ag martingale problems in abstract spaces (Echeverria theorem~\cite[theorem 4.9.17]{ethier-kurtz}) and the diffusion counterpart of DiPerna-Lions theory~\cite{figalli-sdes}. In~\cite{AmbrTrevis14-conteq}
this result has been extended to general metric measure spaces and formulated in the language of $\Gamma$-calculus. It is the primary goal of the present paper to extend this result to the most general context of metric spaces without any specifically distinguished reference measure.

The results of type~(B) have been proven in~\cite[\S 8]{Ambrosio-Gigli-Savare-05} for curves of measures in a Euclidean space and further by S.~Lisini~\cite{lisini07} for general metric spaces. They found a lot of applications in optimal transportation theory~\cite{Ambrosio-Gigli-Savare-05} and in many related theories, e.g.\ in the theory of Sobolev and BV spaces over metric measure spaces (see, e.g.~\cite{GigliHan15-conteq,AmbrDiMar14,AmbrDiMarSav15,AmbrGigliMonRaj15,AmbrGigliSav15,AmbrGigliSav14,AmbrGigliSav13,AmbrColombDiMar12-reflSobol1q}). In the present paper we extend the original result of S.~Lisini showing additionally the relationship between curves of measures and continuity equations for general metric space setting,
and show that in fact this result is a corollary of the superposition principle for solutions of continuity equations (i.e.~(A)).

The results of type~(C) for De Rham currents have been proven first by S.~Smirnov~\cite{Smirnov94} (later several different proofs have been given for partial results of this kind, see e.g.\ \cite{Santambr-dacmos14} and references therein, and also \cite{GabrMin-econ14} for an interesting discrete analogue), and for general metric current in~\cite{PaoSte11-acycl, PaoSte12-cycl}. In this paper we show in fact that
the result on representation of acyclic metric currents provides~(A) for general metric spaces. This establishes the hierarchy of
superposition principles $(C)\Rightarrow (A)\Rightarrow (B)$.

\subsection{Description of results.} Our first main result is Theorem~\ref{thm:sp} giving the general superposition principle
of type~(A). It says (in a sense to be made precise) that for every metric space $E$, a time-dependent family of vector fields $\{V_t\}_{ t \in (0,1)}$ on $E$ and a narrowly continuous  family of measures $\{\mu_t\}_{t \in [0,1]}$ solving the continuity equation
\begin{equation} \label{eq:pde-introduction}
\partial_t \mu_t + \div( V_t  \mu_t) = 0 \quad \text{in $E\times (0,1)$,}
\end{equation}
there exists a Borel probability measure $\eta$ on the space $C([0,1]; E)$ concentrated on the solutions to the ODE
\begin{equation}\label{eq:ode-introduction}
\dot{\theta} (t)= V_t (\theta) \quad \text{for a.e.\ $t \in  (0,1)$,}
\end{equation}
such that $(e_t)_{\#} \eta = \mu_t$, for every $t \in [0,1]$, where $e_t\colon \theta\in C([0,1];E)\mapsto \theta(t)\in E$
is the evaluation map. Of course, the notions of solutions as well as some objects appearing in the above equations have to be properly understood, which leads to a number of difficulties that we explain below.

\subsubsection*{How to understand the weak solution of continuity equation (what is the right class of test functions)}
One clearly has to look for solutions of~\eqref{eq:pde-introduction} in the appropriately weak sense of duality with some set of test functions over $E\times [0,1]$.
The reasonable choice of the latter would be smooth functions, if $E$ possesses a smooth structure (e.g.\ is a Euclidean space), while in a generic metric space the usual substitute of smooth functions are Lipschitz ones.
But then there is a problem of interpreting~\eqref{eq:pde-introduction} already in the Euclidean setting, moreover, even
in case $E=\R$.
Namely, formally writing the weak formulation of~\eqref{eq:pde-introduction}, we have
\[
 \int_0^1\, dt \int_E \left((V_t\pi(\cdot,t))(x)+\partial_t \pi(x,t)\right)\,d\mu_t(x)=0
\]
(the meaning of $V_t\pi$ will be explained in a moment) for every test function $\pi$ appropriately vanishing at $t=0$ and $t=1$. But if $\pi$ is just Lipschitz continuous, then $\partial_t \pi(x,t)$
is not well-defined for $\mu_t\otimes dt$-a.e.\ $(x,t)\in E\times \R$ unless $\mu_t$ satisfies some extra conditions
(e.g.\ are all absolutely continuous with respect to the same measure; note that this is exactly the case considered in~\cite{AmbrTrevis14-conteq}). This is the case of even the simplest
example of the flow representing the evolution of just a single particle in $E=\R$, in which case $\mu_t$ are Dirac measures.
 The usual valid alternative which we choose here is using only test functions which belong to the tensor product of some algebra $\A$ of functions of only space variables and, say, bounded Lipschitz (or just smooth) functions of time.

\subsubsection*{What is the action of a vector field in absence of smooth structure}
 One has to understand also the ``action $V_t \pi$ of a vector field $V_t$ over a test function $\pi$'', once $E$ is a generic metric space which may not possess any smooth structure.
To define it,  we rely on a suitable notion of derivation, i.e., we define $V_t$ as an operator acting on  $\A$, satisfying the Leibniz rule
\[
V_t(fg) = f V_t(g) + g V_t(f)\quad  \text{for $f$, $g \in \A$,}
\]
together with some further mild continuity condition. 
This notion is strongly related to the notion of measurable derivations introduced by N.~Weaver in~\cite{Weaver_deriv00}  but is in fact \emph{strictly weaker} than the latter (see Remark~\ref{rem:not-weaver}); another close analogue is the class of derivations
introduced in~\cite{dimarino-14}.
It includes as particular case derivations in metric measure spaces considered in~\cite{AmbrTrevis14-conteq}, but is more flexible,
since as opposed to the latter paper $\mu_t$ here is not assumed to be absolutely continuous with respect to some reference measure.

\subsubsection*{How to understand the ODE}
At last, note that also the ODE~\eqref{eq:ode-introduction} formally does not make sense
 unless $E$ possesses a smooth structure and absolutely continuous curves in $E$ are differentiable almost everywhere.
 To understand properly this equation, we use the ideology described in~\cite{JetNestruev}. Namely, we consider the functions from $\A$ as observables
 and interpret~\eqref{eq:ode-introduction} as
 \begin{equation}\label{eq:ode-observables}
 \frac{d}{dt} \pi (\theta (t)) = (V_t \pi) (\theta(t)), \quad \text{a.e.\ $t \in  (0,1)$,}
 \end{equation}
for every $\pi\in \A$ and
for $\eta$-a.e.\ $\theta$ (the set of curves $\theta$ such that~\eqref{eq:ode-observables} holds depending possibly on $\pi$).
If $E$ is, say, a Banach space with Radon-Nikodym property (or a metric space bi-Lipschitz embedded in the latter, in particular, a smooth finite-dimensional manifold),
then we show that in fact~\eqref{eq:ode-introduction} holds as it is written, with $V_t$ replaced by a pointwise-defined vector field $\hat{V}_t$.


\subsubsection*{Basic machinery}
To prove Theorem~\ref{thm:sp}, we introduce the metric $1$-current $T'$ defined over the space-time $E\times [0,1]$
formally by the expression
\begin{equation}\label{eq:introduction-current}
    T' (f\,d\pi) :=
  \int_0^1\, dt \int_E f(x,t) \left((V_t\pi(\cdot,t))(x)+\partial_t \pi(x,t)\right)\,d\mu_t(x).
\end{equation}
One certainly has to prove that the above formula makes sense for every Lipschitz function $\pi$ of space and time; this is nontrivial since we a priori defined $V_t$ only on some subalgebra $\A$ of Lipschitz functions of only a space variable, and, as mentioned, for a generic Lipschitz function $\pi$ even the derivative $\partial_t \pi$ is not generally speaking defined $\mu_t\otimes dt$-a.e. The abstract Lemma~\ref{lm_extnormalcurr1} on extension of bilinear functionals to currents will be instrumental here. Note that once the current $T'$ is defined, thanks to the abstract connection between metric currents and Weaver derivations~\cite[Appendix A and B]{PaoSte14-flow}, we may read it alternatively as an extension of the original vector fields $\{V_t\}_{t \in (0,1)}$ to Lipschitz functions, but such an extension in general depends as well on the solution $\{\mu_t\}_{t \in [0,1]}$ (the ``dynamics'') we are dealing with.

Once $T'$ is defined, we show that it is acyclic,
 and then apply the Smirnov-type decomposition for acyclic normal $1$-currents in metric spaces obtained in~\cite{PaoSte11-acycl}, i.e.\ the superposition principle of type~(C), representing
this current by a measure over Lipschitz space-time curves.
We further show  that
since the ``vector field'' defining the current $T'$ has ``positive component along the time direction'',
or, in other words, the current ``is directed forward in time''
(this intuitive concept can be in fact expressed in intrinsic terms, see Remark~\ref{rem_forward1})
then the space-time curves
representing $T'$ can be collectively reparameterized so as to have the unit velocity in the time direction.
Note that, although it might seem counterintuitive, the choice of the distance in space-time will play an important role here.
One gets in this way the measure $\eta$ over absolutely continuous curves in $E$ that does the job.
We further show, similarly to the proof of proposition~5.1 from~\cite{PaoSte14-flow}, that if $E$ is a Banach space
and one chooses any representation
of $T'$ without cancelation of mass, then
instead of the ODE~\eqref{eq:ode-observables} or~\eqref{eq:ode-introduction} we have only the weaker relationship for metric velocities
 \begin{equation}\label{eq:ode-metric-velocity}
 |\dot{\theta}| (t) = |V_t|(\theta(t)), \quad \text{a.e.\ $t \in  (0,1)$,}
 \end{equation}
for $\eta$-a.e.\ $\theta$, where $|V_t|$ is some naturally defined ``modulus'' of the vector field $V_t$.
If, moreover, the norm in $E$ is strictly convex, then we show~\eqref{eq:ode-observables}
(or even~\eqref{eq:ode-introduction}, if $E$ has Radon-Nikodym property) in Theorem~\ref{th_fl2T'repr2}.
This requires a the study of the equality case in the Jensen inequality for vector integrals
which we interpret as a weak$^*$ integrals, see Appendix~\ref{sec_wkstInt} below.
Note now that the claim of
the Theorem~\ref{thm:sp} is independent of the equivalent choice of the distance in $E$ (hence of the mass of $T'$).
Thus the result for a generic metric space $E$
follows then from embedding this space  in $\ell^\infty$ and renorming the latter space by an equivalent dual strictly convex norm.
We think it also worth mentioning here that the relationship between the continuity equation and currents in space-time is even deeper: in fact, every such current ``directed forward in time'' defines in a natural way
a flow $\{\mu_t\}$ of measures and the family of velocity vector fields satisfying~\eqref{eq:pde-introduction} (see Remark~\ref{rem:from-curves-to-equation}).

\subsubsection*{Representation of curves of measures by measures over curves}
As an application of Theorem~\ref{thm:sp} we get
the result the structure of  curves of probability measures over a metric space $E$
absolutely continuous with respect to some Kantorovich-Wasserstein distance $W_p$ with $p \in (1,\infty]$, i.e.\ superposition principle of class~(B). We show in fact (Theorem~\ref{th_Lisini1}) that any such curve in fact satisfies the continuity equation for some
explicitly constructed family of vector fields $\{V_t\}_{t\in [0,1]}$, and hence in particular satisfies~\eqref{eq:ode-observables} (or~\eqref{eq:ode-introduction} if it makes sense) for some measure $\eta$ over absolutely continuous curves, and, moreover,
\[
\int_{C([0,1];E)} |\dot \theta| ^p (t) d\eta(\theta) = | \dot{ \mu } _t | _{W_p}^p, \quad \text{a.e.\ $t \in (0,1)$.}
\]
This extends slightly the original results of S.~Lisini from~\cite{lisini07,Lis14} (in fact, in the latter papers the family of vector fields has been
found only for the case when $E$ is a Banach space with Radon-Nikodym property, which is a particular case of our general result).
Our proof is in a sense ``dual'' to that of Lisini, and in fact follows the strategy of the proof of the analogous result~\cite[theorem 8.2.1]{Ambrosio-Gigli-Savare-05}, for Euclidean spaces.

%

\subsubsection*{Further applications}
We also investigate some refinements of results related to flows of measures on metric spaces, in the framework introduced in~\cite{PaoSte14-flow}, where a definition of measures $\mu$ that may be ``transported'' by some current $T$
(in space $E$ only rather than in space-time) is given. This notion is strongly related to existence of a solution of the continuity equation, with initial measure $\mu_0 = \mu$, driven by a vector field associated to that of the current $T$. Here we slightly strengthen~\cite[theorem~6.1]{PaoSte14-flow} relating it to the superposition principle, whenever the underlying space $E$ is a Banach space with Radon-Nikodym property and strictly convex norm, by proving that any measure on curves provided by the latter for the continuity equation satisfied by the naturally defined flow of $\mu$, in fact represents the same vector field of $T$, or, in other words, represents a normal current in space with the same underlying vector field as $T$ and with different mass measure (still absolutely continuous with respect to that of $T$).


Let us stress that our proof technique is based on the more recent approach of~\cite{PaoSte11-acycl}, ultimately relying on a reduction to polygonal currents, and not on Smirnov's original proof. In fact, in view of Smirnov's original argument, it seems reasonable that it should allow one to deduce a superposition results for currents from that for continuity equations, reversing our main implication (of course, one should first establish a general result for continuity equations by other means).



\section{Basic notation}\label{sec:notation}

Throughout the paper, we let $E$ be a metric space, endowed with a distance $d\colon E \times E \to [0,\infty)$, and let $C_b(E)$ denote the space of real-valued,
 bounded continuous functions on $E$, endowed with the uniform norm $\nor{\cdot}_\infty$.
 We write $B_r(x) := \{ y \in E \colon d(x,y)< r\}$ for the open ball centered at $x \in E$, with radius $r>0$.

For any set $D\subset E$ and any function $\pi$ defined on $E$, we let $\bar D$ be the closure of $D$, $\1_D$ be its characteristic function, $\pi \res D$ be the restriction of $f$ to $D$.


We denote by $\Lip(E)$ (respectively, $\Lip_b(E)$, $\Lip_1(E)$) the set of real-valued Lipschitz functions on $E$ (respectively, bounded Lipschitz functions, $1$-Lipschitz functions on $E$), and for $\pi\colon E \to \R$, $D \subseteq E$, we let $\Lip(\pi \res D) \in [0,\infty]$ denote the Lipschitz constant of $\pi\res D$. 
 For $x \in E$, we define the \emph{asymptotic Lipschitz constant} of $\pi$ at $x\in E$ as
\[
(\lip\,\pi \res D)(x) := \inf_{\varepsilon>0} \Lip\,(\pi\res( D \cap  B_\varepsilon(x) ) )= \lim_{\varepsilon\to 0^+} \Lip\,(\pi\res (D \cap B_\varepsilon(x) ))
\]
and notice that $x \mapsto \lip\,\pi(x):= (\lip\,\pi\res E)(x)$ is upper semicontinuous. Moreover, one has $\lip( f +g)(x) \le \lip(f)(x) + \lip(g)(x)$ and $\lip(\lambda f )(x) = |\lambda| \lip(f)(x)$. Such a notion is inherently related to that of \emph{local slope} at $x \in E$,
\[
|\nabla\pi|(x):= \limsup_{y\to x} \frac{|\pi(x)-\pi(y)|}{d(x,y)}
\]
(unless $x$ is an isolated point, in which case we set $|\nabla \pi|(x):=0$). Namely, one has that
\[
(\lip\,\pi)(x)\geq |\nabla\pi|^*(x),
\]
where  $ |\nabla\pi|^*$ stands for the upper semicontinuous envelope of $|\nabla\pi|$, i.e.\ the minimum upper semicontinuous function not exceeding $|\nabla\pi|$; the above inequality is an equality if $E$ is a length space~\cite[proposition~3.4]{AmbrColombDiMar12-reflSobol1q}.


It is well-known~\cite{Weaver_deriv00} that $\Lip_b(E)$ equipped with the norm $\pi\mapsto \|\pi\|_\infty+\Lip\,u$ (or any equivalent one)
is dual to a Banach space called Aarens-Eels space, which is separable whenever so is $E$. This duality induces
the weak$^*$ topology over $\Lip_b(E)$. Convergence of a sequence $\{\pi_\nu\}\subset \Lip_b(E)$ with respect to the latter
(i.e.\  weak$^*$ convergence) to some $\pi\in \Lip_b(E)$ is equivalent to pointwise convergence of $\pi_\nu$ to $\pi$
(i.e.\ $\lim_\nu \pi_\nu(x)=\pi(x)$ for all $x\in E$)
with the sequence
of Lipschitz constants $\Lip\,\pi_\nu$ bounded.

Clearly, if $\lim_\nu \pi_\nu(x)=\pi(x)$ for all $x\in E$ with $\{\pi_\nu\}\in \Lip(E)$
(in particular, if $\{\pi_\nu\}\subset \Lip_b(E)$, $\pi\in \Lip_b(E)$ and $\pi_\nu\rightharpoonup \pi$ in weak$^*$ sense of $\Lip_b(E)$ as $\nu\to \infty$)
then
\begin{equation}\label{eq_lsclip1}
\Lip \, (\pi\res D) \leq \liminf_\nu \Lip\, (\pi_\nu\res D),
\end{equation}
for every $D\subset E$,
and the inequality above clearly may be strict.
If both $\lim_\nu \pi_\nu(x)=\pi(x)$ for all $x\in E$ and $\lim_\nu \Lip\, \pi_\nu=\Lip\, \pi$, then we say that $\pi_\nu$ Lip-converge to $\pi$
(an overview of some known density results for Lip-convergence can be found in~\cite{GarrJaraRang10,GarrJaraRang13}).

We denote by $PC^1([0,1])$ the set of piecewise continuously differentiable functions on $[0,1]$.
If $\A$ and $\B$ are linear spaces functions over $E$ and over $[0,1]$ respectively, then
we write  $\A\otimes\B$ for the space of functions
$\pi(x,t) = \sum_{k=1}^m\phi_k(t) g_k(x)$, with $g_k\in \A$ and $\phi_k\in \B$.
Unless otherwise stated, we always endow $E\times [0,1]$ with the distance
$d_2$, where
\[
d_p((x,s), (y,t)):= \left(d(x,y)^p +|t-s|^p\right)^{1/p}, \quad \text{for $x$, $y \in E$, $s$, $t \in [0,1]$, $p\in [1,+\infty)$.}
\]
For $\pi\in \Lip(E\times[0,1])$ we introduce the notation
$(\lip_E\,\pi)(x,t) := (\lip\,\pi \res ( E \times \{t\}) )(x)$,
for $x\in E$, $t\in [0,1]$.

All the curves in a metric space $E$ in the sequel will be considered continuous and parameterized over $[0,1]$, the set of all such curves
being denoted by $C([0,1];E)$.
A curve $\theta\colon [0,1]\to E$ is called \emph{absolutely continuous} if there exists a nonnegative integrable function $g\colon (0,1)\to \R$ such that
$d(\theta(s), \theta(t) ) \le \int_s^t g(r)\, dr$ for every $\{s,t\}\subset [0,1]$, and we let $|\dot{\theta}|(t)$ stand for its metric derivative
at $t\in [0,1]$, which is defined for a.e.\ $t\in [0,1]$, see e.g.~\cite{AmbrTilli00}.


Whenever $E$ is a Banach space endowed with the norm $\|\cdot\|$, we denote by $E'$ its continuous dual
(the respective dual norm being denoted by $\|\cdot\|^\prime$)
and by $\langle \cdot, \cdot\rangle$ the respective duality pairing.
Following~\cite{Goodman71_qdiff}, we call a function $\varphi\colon E\to \R$ is called \emph{quasi-differentiable} at  $x\in E$, if there exists a (unique) functional $d\varphi(x)\in E'$ such that
\[
\lim_{\varepsilon \to 0 } \frac{ \varphi(\theta(t+\varepsilon))- \varphi(\theta(t)) }{\varepsilon } = \langle\dot{\theta}(t),d\varphi(x)\rangle,
\]
for every $\theta \in C([0,1];E)$ and $t \in (0,1)$ such that $x=\theta(t)$ and $\theta$ is differentiable at $t$ with derivative $\dot{\theta}(t)$,
 and simply \emph{quasi-differentiable},  if it is quasi-differentiable at every $x\in E$.
 This notion of differentiability is slightly stronger than G\^{a}teaux one but strictly weaker than Fr\'{e}chet one.
 We let $Q^1(E)$ be the space of quasi-differentiable functions $\varphi\colon E\to \R$ such that
$\|d\phi(\cdot)\|'$ is uniformly bounded
and the map $(x,y)\mapsto \langle y, d\varphi(x) \rangle$ is continuous on $E \times E$.
Let us remark that $Q^1(E)\subset \Lip(E)$, since for $\varphi \in Q^1(E)$, and $\{x,y\}\subset E$, we may let $\theta(t) := (1-t)x + ty$, and by the fundamental theorem of calculus, we obtain the bound
\[
|\varphi(y)-\varphi(x)|  \leq  \int_0^1 \|y-x\|\cdot \|d\varphi((1-t)x+ty)\|'\, dt \leq \|y-x\| \sup_{z \in E} \|du(z)\|'  .
\]
We also note that for $\pi\in Q^1(E)$, one has $\|d\pi(x)\|^\prime\leq \lip\,\pi(x)$, since for every open set $D \subseteq E$, one has $\Lip (\pi\res D)=\sup_{z\in D} \|d\pi(z)\|^\prime$.

Unless otherwise stated, all the measures considered in the sequel are Borel, positive and finite.
For a Borel measure $\mu$ on $E$ and $D\subset E$ Borel, we let $\mu\res D \colon A \mapsto \mu(A \cap D)$ stand for the restriction of $\mu$ to $D$.
If $\mu$ and $\nu$ are measures over $\Omega$ with $\mu$ is absolutely continuous with respect to $\nu$, we write $\mu\ll\nu$
and $\mu = \frac{d\mu}{d\nu} \nu$ where $\frac{d\mu}{d\nu}$ stands for the  respective Radon-Nikodym derivative.
We also need 
Lebesgue-Bochner spaces $L^p(\Omega,\mu; X)$ 
with $X$ a separable Banach space
 of $\mu$-a.e.\ equivalence classes of $X$-valued maps on $E$ that are Bochner integrable on $\Omega$ 
 with power $p\in [1,+\infty)$  or, for $p=+\infty$, $\mu$-essentially bounded. We omit the reference to $X$ when $X = \R$. We use the standard notation $\ell^p$ for the space of sequences summable with $p$-th power, if  $p\in [1,+\infty)$, or the space of bounded sequences, if $p=+\infty$.
 If $E$ is a space of sequences, i.e.\ $E\subset\R^{\infty}$, then a function $\pi\colon E\to\R$ is called \emph{cylindrical},
 if it depends only on finite number of coordinates. The set of such functions is  denoted $\mathrm{Cyl}(E)$.

For metric currents we use the notation from~\cite{PaoSte11-acycl}, which coincides with that of~\cite{AmbrKirch00} except for the notation on the mass measure
of the current $T$ which we denote by $\mu_T$.

\section{The superposition principle for continuity equations}\label{sec:general-sp}

In this section, we formulate and prove our main result concerning solutions to continuity equations on metric spaces.

\subsection{Statement of main result}
\label{sec:statement}

For a metric space $(E,d)$, we consider an algebra $\A\subset \Lip_b(E)$ (i.e., a family stable with respect to pointwise linear combinations and products) with $\1 \in \A$ and we define the continuity equation in duality with respect to $\A$.
Namely, suppose we are given a curve $\{\mu_t\}_{t \in [0,1]}$ of finite nonnegative measures on $E$ such that $t \mapsto \mu_t(A)$ is Borel for every $A \subset E$ Borel. We introduce a (time-dependent) family of vector fields  $\{V_t\}_{t \in (0,1)}$ as a linear operator $V\colon \A\to L^1(E \times [0,1] , \mu_t \otimes dt)$, with  $(V_t \pi) (x) := (V\pi)(t,x)$, satisfying
the Leibniz rule
\begin{equation}\label{eq_VtLeib1}
V_t (fg)= f V_t g + g V_t f, \quad \text{for a.e.\ $t\in [0,1]$ and every $f$, $g \in \A$,}
\end{equation}
and the inequality
\begin{equation}\label{eq_VtLip1}
|(V_t\pi) (x)|\leq \overline{ V}(x,t)\,  \lip\, \pi(x), \quad \text{for a.e.\ $t\in [0,1]$, $\mu_t$-a.e.\ $x\in E$ and every $\pi \in \A$,}
\end{equation}
for some (fixed) Borel function $\overline{ V}\colon E\times[0,1]\to [0,\infty]$ with
\begin{equation}\label{eq:V-bar-integrability}
\int_0^1\int_E \overline{ V} (x,t)\,d\mu_t(x) dt < +\infty.
\end{equation}

\begin{remark}\label{rem:not-weaver}
Note that the above conditions do not  imply that $V_t$ is a $\mu_t$-derivation (or $\mu_t$-measurable vector field) in the sense of Weaver \cite{Weaver_deriv00}, even if $\overline{V}$ is bounded. Indeed, we may let $E:=[0,1]$, $\mu_t:=\delta_t$ for $t\in [0,1]$, $\A:= C^1_b(E)$ and $V_t \pi (s):= \frac{d}{ds}\pi(s)$ (independent of $t$). Then, $V$ satisfies the above conditions with $\overline{V}=1$, but it is not a $\mu_t$-vector field: in fact, no nonzero
$\mu_t$-vector fields exist for such a choice of $\mu_t$.
\end{remark}

With the above notation we say that $\{\mu_t\}_{t \in [0,1]}$ satisfies the continuity equation~\eqref{eq:pde-introduction} if, for every $\pi \in \A$, one has
\begin{equation}\label{eq_fl2cont1}
    \frac{d}{dt}\int_E \pi\,d\mu_t = \int_E V_t\pi \,d\mu_t,
\end{equation}
in the sense of distributions on the interval $(0,1)$.

Our main result will require some density assumptions on $\A$ that we list below.
Namely, we will assume that either
\begin{itemize}
\item[($\A_1$)]
there is 
a constant $C>0$ such that
for every $\pi \in \Lip_b(E\times [0,1])$
there is a sequence of uniformly bounded functions $\{\tilde\pi_m\}\subset \Lip(E\times [0,1])$ with uniformly bounded Lipschitz constants
such that
$\tilde\pi_m(\cdot,t)\in \A$ for all $t\in [0,1]$, $m \in \N$,
$\lim_m \tilde\pi_m=\pi$ pointwise in $E \times [0,1]$,
and
\begin{equation}\label{eq_Liplocconv1b}
\limsup_m \lip\, \tilde\pi_m(x,t)  \leq C\lip\, \pi(x,t)\quad\text{for all $(x,t)\in E\times [0,1]$},
\end{equation}
\end{itemize}
or, alternatively,
\begin{itemize}
\item[($\A_2$)]
there is 
a constant $C>0$ such that
for every $\pi \in \Lip_b(E\times [0,1])$ there is a sequence of functions
$\{\tilde\pi_m\}\subset \Lip(E\times [0,1])$ as in~($\A_1$) but satisfying
\begin{equation}\label{eq_Liplocconv1a}
\limsup_m \Lip\, \tilde \pi_m\leq C \Lip\, \pi.
\end{equation}
instead of~\eqref{eq_Liplocconv1b}, and
\item[($B$)] for every open $U\subset E\times [0,1]$ and $f\in \Lip_b(E\times [0,1])$ with $\supp\, f\subset U$
there is a sequence of uniformly bounded functions
$\{f_m\}\subset \Lip_b(E\times [0,1])$ with uniformly bounded Lipschitz constants
such that for all $m\in \N$ one has $f_m(\cdot,t)\in \A$ for all $t\in [0,1]$
and $\supp\, f_m \subset U$, while $\lim_m f_m=f$
pointwise.
\end{itemize}

Note that all the above assumptions are stable with respect to any equivalent choice of the distance in $E\times [0,1]$.
We will see in the sequel that in fact, they are stronger than what one really needs for
the results to hold, and hence can be further weakened by replacing pointwise convergence with weaker notions
of convergence involving the measures $\mu_t$, see Remark~\ref{rem:a1-a2_1} below.
Anyway, the assumptions~($\A_1$),~($\A_2$) and~($B$) or their respective weaker versions will not be used together. Namely, we will
assume either~($\A_1$) or alternatively trade the local condition~\eqref{eq_Liplocconv1b} for the global one~\eqref{eq_Liplocconv1a}, hence assuming~($\A_2$) and compensate then the lack of local information by assuming additionally~($B$). Assumptions~($\A_2$) and~($B$) may be easier to verify in the finite-dimensional setting, as the following Remark shows.

\begin{remark}\label{rem:a1-a2_0a}
 An important example when~($\A_2$) and~($B$) hold 
 is that of $E$ being a smooth connected Riemannian manifold, with $\A = C^\infty_c(E)$, or a compact connected
 Finsler manifold with $\A = C^1_b(E)$ (continuously differentiable functions with bounded derivatives),  see e.g.~\cite{GarrJaraRang10,GarrJaraRang13}.
\end{remark}

On the other hand, in some infinite-dimensional settings it might be difficult to have~($B$), hence~($\A_1$) might be easier to verify. For instance (and we will use this in the sequel in section~\ref{sec_Lisini1}), if $E$ is a separable subspace of
$\ell^\infty$, then by Lemma~\ref{lm_mcurr_densc0lip} an appropriate algebra of smooth cylindrical (i.e.\ depending only on a finite number of coordinates) functions satisfies~($\A_1$). The following trivial fact however is worth mentioning.

\begin{remark}\label{rem:a1-a2_0}
Clearly, $\A:= \Lip_b(E)$ satisfies all the assumptions~($\A_1$),~($\A_2$) and~($B$).
\end{remark}

Note that Lemma~\ref{lm_approxtens1} says that whenever $\A$ satisfies~($\A_1$) (resp.~($\A_2$)), then automatically a better approximation result
holds, namely, every function in $\Lip_b(E\times [0,1])$ can be approximated in the sense of pointwise convergence with uniformly bounded Lipschitz
constants by uniformly bounded functions from $\A\otimes PC^1([0,1])$, such that,~\eqref{eq_Liplocconv1b}
(resp.~\eqref{eq_Liplocconv1a}) holds for the approximating sequence (maybe with a different, but still universal, constant $C>0$). The same is true about condition~($B$) by Remark~\ref{rem_approxtens2}, namely, if $\A$ satisfies this condition, then
every function in $\Lip_b(E\times [0,1])$ can be approximated in the same sense
by uniformly bounded functions from $\A\otimes PC^1([0,1])$ supported in the neighborhood of the support of the limit function.

In the rest of this section we address the proof of the following result.

\begin{theorem}[Superposition principle for continuity equations]\label{thm:sp} Let $\{\mu_t\}$ and $\{V_t\}$, where $t \in [0,1]$, be a narrowly continuous curve of finite positive Borel measures and a family of time-dependent vector fields  as above, such that~\eqref{eq_VtLeib1}, \eqref{eq_VtLip1}, \eqref{eq:V-bar-integrability} and \eqref{eq_fl2cont1} hold.  If either assumption~($\A_1$) or the assumptions~($\A_2$) and~($B$) hold for $\A$, then there exists some finite positive Borel measure $\eta$ over $C([0,1];E)$ concentrated on absolutely continuous curves $\theta\colon[0,1] \to E$ such that $(e_t)_\# \eta = \mu_t$, for every $t \in [0,1]$, and one has
\[\frac{d}{dt} \pi(\theta (t))= (V_t\pi)(\theta(t)) \quad
\text{for $\eta$-a.e.\ $\theta$ and a.e.\ $t \in (0,1)$}
\]
for every $\pi \in \A$, the exceptional set of pairs $(\theta,t)$ for which the above equation does not hold depending possibly on $\pi$.
\end{theorem}

Let us point out that, in the  above theorem, narrow continuity of $t \mapsto \mu_t$ is a necessary condition for the validity of a probabilistic representation. On the other hand, it is not essential, since some additional separability assumption would imply existence of a unique narrowly continuous representative, arguing as in~\cite[lemma~8.1.2]{Ambrosio-Gigli-Savare-05}.

\begin{remark}\label{rem:a1-a2_1}
As it will be clear from the proof, the assumptions~($\A_1$),~($\A_2$) and~($B$) can be weakened and replaced by
the following assumptions~($\A_1'$),~($\A_2'$) and~($B'$) respectively.
\begin{itemize}
\item[($\A_1'$)]
There is 
a constant $C>0$ such that
for every $\pi \in \Lip_b(E\times [0,1])$
there is a sequence $\{\tilde\pi_m\}\subset \A\otimes \Lip_b([0,1])$ such that
$\lim_m \tilde\pi_m=\pi$ in $L^1(E \times [0,1],\mu_t\otimes dt)$,
$\lim_m \tilde\pi_m(\cdot,t)=\pi(\cdot,t)$ in $L^1(E,\mu_t)$ for every $t\in [0,1]$,
and
\begin{equation}\label{eq_Liplocconv1bb0}
\limsup_m \lip\, \tilde\pi_m(x,t)  \leq C\lip\, \pi(x,t)\quad\text{for all $(x,t)\in E\times [0,1]$},
\end{equation}
where the limit is in the weak$^*$ sense of $L^\infty(E\times [0,1], \mu_t\otimes dt)$.
\item[($\A_2'$)]
There is 
a constant $C>0$ such that
for every $\pi \in \Lip_b(E\times [0,1])$ there is a sequence of functions
$\{\tilde\pi_m\}$ as in~($\A_1'$) but satisfying
\begin{equation}\label{eq_Liplocconv1aa0}
\limsup_m \Lip\, \tilde \pi_m\leq C \Lip\, \pi.
\end{equation}
instead of~\eqref{eq_Liplocconv1bb0}.
\item[($B'$)] For every open $U\subset E\times [0,1]$ and $f\in \Lip_b(E\times [0,1])$ with $\supp\, f\subset U$
there is a sequence
$\{f_m\}\subset \A\otimes \Lip([0,1])$ with
 $\lim_m f_m=f$ in $L^1(E\times [0,1],\mu_t\otimes dt)$ and $\supp\, f_m\subset U$.
\end{itemize}
All these assumptions will only be needed in Proposition~\ref{prop_fl2T'curr1} on extension of the functional $T'$ defined
by~\eqref{eq_fl2defT'} below to a metric current.
\end{remark}

\subsection{Extension of the vector field to tensor products in space-time}
For an arbitrary $\pi\in \A\otimes \Lip_b([0,1])$, $\pi = \sum_{k=1}^m \phi_k(t)  g_k(x)$, with $g_k \in \A$,
$\phi_k\in \Lip_b(E\times [0,1])$,
 we define
\[ V_t\pi(\cdot,t) (x):= \sum_{k=1}^m \phi_k(t) (V_t g_k)(x).\]
Clearly the right hand side is in $L^1(E\times [0,1], \mu_t \otimes dt)$, since
\[ \left| \sum_{k=1}^m \phi_k(t) (V_t g_k)(x) \right|\le \overline{V}(x,t) \sum_{k=1}^m |\phi_k(t)| \, \lip \, g_k (x)\]
by~\eqref{eq_VtLip1}. The next lemma improves this inequality and entails in particular the correctness of such a definition.

\begin{lemma}\label{lem:extension-tensor-like-space-time}
For $\pi \in \A \otimes \mathcal{B}^\infty([0,1])$, one has
\[|(V_t\pi(\cdot, t) (x)|\leq \overline{ V}(x,t)\,  \lip_E\, \pi(x,t), \quad \text{for a.e.\ $t\in [0,1]$, $\mu_t$-a.e.\ $x\in E$.}
\]
\end{lemma}

\begin{proof}
Let $\{\lambda_k\}_{k=1}^m \subset \Q$, and consider $\sum_{k=1}^m \lambda_k g_k \in \A$. Then, by \eqref{eq_VtLip1}, one has
\[ V_t\left( \sum_{k=1}^m \lambda_k g_k\right) (x) \le\overline{ V}(x,t)\,  \lip \left( \sum_{k=1}^m \lambda_k g_k\right)(x),  \quad \text{for a.e.\ $t\in [0,1]$, $\mu_t$-a.e.\ $x\in E$,}\]
where we may assume that the negligible set of $(x,t)$ above does not depend on $\{\lambda_k\}_{k=1}^m \subset \Q$. Thus for $(x,t)$ outside such set, we may consider sequences $\{\lambda_k^n\}_{n} \subset \Q$ with $\lim_n\lambda_k^n =\phi_k(t)$  for $k \in \{1, \ldots, m\}$, so that
\begin{equation*}
\begin{split} (V_t\pi(\cdot,t))(x) &= \lim_{n\to \infty} V_t\left( \sum_{k=1}^m \lambda_k^n g_k\right) (x)\\
& \le  \overline{ V}(x,t)\, \limsup_{n \to \infty}  \lip \left( \sum_{k=1}^m \lambda_k^n g_k\right)(x) \\
&\le\overline{ V}(x,t)\left( \lip \left( \sum_{k=1}^m \phi_k(t) g_k\right)(x)  +  \limsup_{n \to \infty}  \lip \left( \sum_{k=1}^m (\lambda_k^n-\phi_k(t))  g_k\right)(x)\right) \\
&\le \overline{ V}(x,t)\,  \lip_E\, \pi(x,t) +  \overline{ V}(x,t)\, \limsup_{n \to \infty} \sum_{k=1}^m |\lambda_k^n-\phi_k(t)|  \lip \left(  g_k\right)(x) \\
& \le  \overline{ V}(x,t)\,  \lip_E\, \pi(x,t),
\end{split}
\end{equation*}
concluding the proof.
\end{proof}

\subsection{Extension to a space-time current}

From now on until the end of Section~\ref{sec:general-sp} we assume that $\{\mu_t\}$ and $\{V_t\}$, where $t \in [0,1]$, be a narrowly continuous curve of finite positive Borel measures and a family of time-dependent vector fields respectively, satisfying the relationships~\eqref{eq_VtLeib1},~\eqref{eq_VtLip1},~\eqref{eq:V-bar-integrability} and~\eqref{eq_fl2cont1}.
For every $f\,d\pi\in D^1(E\times[0,1])$ with $\pi \in \A\otimes PC^1([0,1])$, let
\begin{equation}\label{eq_fl2defT'}
    T' (f\,d\pi):=
  \int_0^1 \int_E f(x,s) \left( \left(V_s\pi(\cdot,s)\right)(x)+\partial_t \pi(x,s)\right)\,d\mu_s(x) \,ds.
\end{equation}
The first step consists in extending such bilinear functional to a metric current on $E \times [0,1]$. In fact, the following assertions regarding $T'$ hold true.

\begin{proposition}\label{prop_fl2T'curr1}
Let $E\times [0,1]$ be equipped with any distance equivalent to $d_1$ and either assumption~($\A_1'$) (in particular, the stronger assumption~($\A_1$)) or the assumption~($\A_2'$) together with~($B'$) (in particular, the stronger assumptions~($\A_2$) and~($B$)) hold for $\A$.
Then the  functional $T'$ can be (uniquely) extended
to a metric current (still denoted by $T'$)
$T'\in \M_1(E\times [0,1])$. Furthermore, the extended current $T'$ is normal, with
\[
\partial T' = \mu_1\otimes \delta_1- \mu_0\otimes \delta_0,
\]
and, for $t \in [0,1]$, the subcurrent $T'_t:= T\res \1_{[0,t]}$, defined by
\[     T'_t (f\,d\pi) =
  \int_0^t  \int_E f(x,s) \left( (V_s\pi(\cdot,s))(x) + \partial_t \pi(x,s)\right) \,d\mu_s(x)ds\]
is also normal, with $\partial T'_t = \mu_t\otimes \delta_t- \mu_0\otimes \delta_0$ and
$\mu_{T'_t} \leq
  C(1+\overline{ V})\,d\mu_s(x)\, ds$ for some $C>0$.
\end{proposition}

\begin{proof}
We may assume without loss of generality $E\times [0,1]$ to be endowed with the distance $d_1$.

{\sc Locality and mass estimate}.
Let 
$f\in L^1(E\times [0,1], ( 1+ \overline V)\mu_t\otimes\, dt)$,
 $\pi\in \A\otimes PC^1([0,1])$. Then, for all $t\in [0,1]$, one has
\begin{equation*}\label{eq_massestTt'1a}
\begin{aligned}
|T'_t (f\,d\pi)| & \leq
  \int_{E\times (0,t)} |f(x,s)| \left( \overline V(x,s)\lip_E\pi(x,s) +
\left|\partial_t \pi(x,s)\right|\right)\,d\mu_s(x)\,ds\\
&\leq
\int_{E\times (0,t)} |f(x,s)| \left( 1+ \overline V(x,s)\right)\lip\,\pi(x,s)\, d\mu_s(x)\,ds.
\end{aligned}
\end{equation*}

{\sc Boundary mass estimate and Leibniz rule}. 
For $\pi(x,t):= \phi(t) g(x)$ with $g\in \A$ and $\phi\in PC^1([0,1])$, the continuity equation~\eqref{eq_fl2cont1} and  continuity of the map $s \mapsto \int_E g d\mu_s$ give
\begin{equation}\label{eq_boundTt'1}
\begin{aligned}
T'_t (1\,d\pi) & =
  \int_0^t \dot{\phi}(s) \,ds \int_E g\,d\mu_s +
  \int_0^t \phi(s) \,ds \int_E  V_s g\,d\mu_s(x)\\
  &=\phi(t)\int_E g\,d\mu_s- \phi(0)\int_E g\,d\mu_0 \\
  &\qquad\qquad +
  \left\langle \phi,  - \frac{d\,}{ds}\int_E g\,d\mu_s + \int_E V_s g\,d\mu_s \right\rangle\\
  &= \phi(t)\int_E g\,d\mu_s- \phi(0)\int_E g\,d\mu_0 =   \left\langle \pi,  \mu_t\otimes \delta_t- \mu_0\otimes \delta_0 \right\rangle.
\end{aligned}
\end{equation}
By linearity, the identity extends to $\pi\in \A\otimes PC^1([0,1])$. Finally, using Leibniz rule~\eqref{eq_VtLeib1} for $V_t$
and for $\partial_t$, we get that
\[
T'_t(f\,d\pi)+T'_t(\pi\,df)=T'_t(1\,d(f\,\pi))
\]
for every $(f,\pi)\in (\A\otimes PC^1([0,1]))\times (\A\otimes PC^1([0,1]))$.

We invoke now Lemma~\ref{lm_extnormalcurr1} below (with $E\times [0,t]$,
$\A\otimes PC^1([0,1])$,
$T_t'$, 
$\left( 1+ \overline V\right)\mu_s\otimes ds$
and $\mu_t\otimes \delta_t+ \mu_0\otimes \delta_0$ instead of
$E$,
$\A$,
$T$, $\mu$ and $\nu$
respectively), recalling that~($\A_1'$) implies assumption~(i) of this Lemma with these data in view of Lemma~\ref{lm_approxtens1},
while~($\A_2'$) together with~($B'$) similarly implies assumption~(ii) in view of Lemma~\ref{lm_approxtens1}
and Remark~\ref{rem_approxtens2}.
We conclude therefore that $T'_t$ may be uniquely extended by continuity to a current (denoted in the same way).
In particular,
for $t \in [0,1]$ one has $T'_t \leq T'_1$.
Finally,~\eqref{eq_boundTt'1}
implies
\[
\partial T'_t = \mu_t\otimes \delta_t- \mu_0\otimes \delta_0,
\]
that is, $T'_t$ are normal currents with boundary measures as claimed.
\end{proof}

The above proof uses the following result on extension of bilinear functionals to metric currents.

\begin{lemma}\label{lm_extnormalcurr1}
Let $E$ be a complete metric space, $\mu$ and $\nu$ be finite positive Borel measures on $E$ and
$\A\subset \Lip_b(E)$ be an algebra with $\1 \in \A$.
Let
$T\colon
\A \times\A\to \R$ be a bilinear functional for which we retain the notation of currents,
i.e.\ we write $T(f\,d\pi)$ for $T(f,\pi)$. Suppose that
\begin{eqnarray}\label{eq_Text_leib1}
T(f\,d\pi)+ T(\pi\,df)& = & T(1\, d(f\pi)),\\
\label{eq_Text_mass1}
  |T(f\,d\pi)| & \leq & \int_E |f| \, \lip\,\pi\, d\mu,\\
\label{eq_Text_bdmass1}
    |T(\1\,d\pi)| &\leq & \int_E |\pi|\, d\nu
\end{eqnarray}
for every $(f,\pi)\in
\A \times\A$. Assume also that
\begin{itemize}
\item[(i)] either there is a constant $C$ such that, for every  $\pi \in \Lip_b(E)$,  there is a sequence $\{\pi_n\} \subset \A$ of uniformly bounded functions with $\lim_n \pi_n=\pi$ in $L^1(E,\mu+\nu)$ and
\begin{equation}
\label{eq_Liplocconv2a}
\lim_n \lip\,\pi_n\leq C\lip\,\pi
\end{equation}
where the limit in~\eqref{eq_Liplocconv2a} is in the weak$^*$ sense of $L^\infty(E, \mu)$.
\item[(ii)] or there is a constant $C$ such that, for every  $\pi \in \Lip_b(E)$,  there is a sequence $\{\pi_n\} \subset \A$
of uniformly bounded functions with $\lim_n \pi_n=\pi$ in $L^1(E,\mu+\nu)$ and
\begin{equation}\label{eq_Liplocconv2}
\lim_n \Lip\,\pi_n\leq C\Lip\,\pi
\end{equation}
and for every open $U\subset E$ and $f\in \Lip_b(E)$ with $\supp\, f\subset U$ there is a sequence
$\{f_n\}\subset \A$ with $\supp\, f_n\subset U$ such that $\lim_n f_n=f$ in the sense of $L^1(E,\mu)$.
\end{itemize}
Then $T$ may be uniquely extended by continuity to
a normal current (still denoted by $T$) with $\mu_T \leq C\mu$ and $\mu_{\partial T} \le \nu$.
\end{lemma}

Note that~\eqref{eq_Liplocconv2} is weaker than Lip-convergence of $\pi_n$, and strictly stronger than pointwise convergence with bounded Lipschitz constants. The validity of~(i) or~(ii) imply the density of $\A$ both in $L^1(E,\mu)$ and in $L^1(E,\nu)$. It worth also emphasizing that it is for \eqref{eq_Text_bdmass1} to make sense that we require $\1 \in \A$, but it would be unnecessary if $T$ were initially defined over $L^1(E,\mu) \times \A$.

\begin{remark}\label{rem_equivalent_Text_mass1}
Condition~\eqref{eq_Text_mass1}, under the validity of~(ii), is equivalent to the fact that
\begin{equation}\label{eq_mutdpi1}
 |T(f\,d\pi)| \leq \Lip\,(\pi\res U) \int_U |f|\, d\mu \quad \text{for every $U \subset E$ open}
\end{equation}
whenever $f \in \A$, with $\supp f \subset U$.
In fact, if~\eqref{eq_mutdpi1} is valid for some function, then for fixed $\pi \in \A$, the functional $\A \ni f \mapsto T(f\, d\pi)$ is induced by a (unique) measure $\mu_{T\res d\pi} \le \Lip\, (\pi)\,  \mu$, because of \eqref{eq_mutdpi1} with $U=E$ and the density of $\A$ in $L^1(E,\mu)$. For each open $U\subset E$, by \eqref{eq_mutdpi1} and the density of functions in $\A$ with support in $U$ in $L^1(U,\mu)$ entail that the Radon-Nikodym derivative $\frac{d\mu_{T\res d\pi} }{d\mu}$ is $\mu$-a.e.\ smaller than $\Lip \, \pi\res U$ over $U$. Taking a countable family of balls $(U_j)_{j}$ such that, for every $x \in E$, $\varepsilon >0$, there are $U_j$, $U_k$ with $U_j \subset B_\varepsilon(x) \subset U_k$, it follows that $\frac{d\mu_{T\res d\pi} }{d\mu} \le \lip \, \pi$, $\mu$-a.e.\ on $E$. The converse implication is obvious.
\end{remark}

\begin{proof}
Given $(f_1,\pi_1)$, $(f_2,\pi_2) \subset\A\times \A$, the assumptions on $T$ entail
\begin{equation}\label{eq_Text_mass2}
\begin{aligned}
    |T(f_1\,d\pi_1) & -T(f_2\,d\pi_2)|
    \leq  |T(1\,d(f_1\pi_1-f_2\pi_2)|  + |T(\pi_1\,df_1)-T(\pi_2\,df_1)| +|T((f_1-f_2)\,d\pi_2)|\\
    &\leq  \int_E |f_1|\, |\pi_1-\pi_2|\,d\nu  +  \int_E \lip \,f_1|\pi_1-\pi_2|\,d\mu
    + \int_E\lip\, \pi_2 |f_1-f_2|\,d\mu
       \\ &\leq  \int_E |f_1|\, |\pi_1-\pi_2|\,d\nu  +  \Lip \,f_1\int_E |\pi_1-\pi_2|\,d\mu
    + \Lip\, \pi_2 \int_E |f_1-f_2|\,d\mu.
\end{aligned}
\end{equation}
Given $(f,\pi) \in \Lip_b (E)\times \Lip_b(E)$, and $\{(f_n,\pi_n)\}\subset\A\times \A$, $f_n$ and $\pi_n$ being uniformly bounded
with $\lim_n f_n=f$ and $\lim_n \pi_n=\pi$ in $L^1(E,\mu+\nu)$ and with $\lip\, f_n$ and $\lip\, \pi_n$ uniformly bounded in $L^\infty(E,\mu)$, from the above inequality above we get that $\{T(f_n\,d\pi_n)\}$ is convergent. Thus, by either~\eqref{eq_Liplocconv2a} or~\eqref{eq_Liplocconv2}, we obtain that $T$ can be uniquely extended by continuity to $\Lip_b(E)\times \Lip_b (E)$, for which we retain the same notation. The identity~\eqref{eq_Text_leib1} is still valid for $(f,\pi) \in \Lip_b(E)\times\Lip_b(E)$ because $\lip (f_n \pi_n)$ is also uniformly bounded in $L^\infty(E,\mu)$. Inequality~\eqref{eq_Text_bdmass1} also clearly remains valid for $\pi \in \Lip_b(E)$.  The validity of~\eqref{eq_Text_mass1}  with $C\mu$ instead of $\mu$
and $\Lip_b(E)$ instead of $\A$ follows in case~(i) from
\[
| T (f\, d\pi )| = | \lim_n T (f\, d\pi_n )| \le \lim_n  \int_E \lip\,\pi_n |f|\, d\mu \le C  \int_E \lip\, \pi |f|\, d\mu,
\]
for $(f,\pi) \in \Lip_b(E)\times \Lip_b(E)$.

In case~$\A$ satisfies~(ii), to show~\eqref{eq_Text_mass1} for $(f,\pi) \in \Lip_b(E) \times \Lip_b(E)$, we first
prove the following locality property: if $f \in \Lip_b(E)$, $\pi \in \Lip_b(E)$ are such that $\pi$ is constant on some open set $U \supset \supp f$, then $T(f \, d\pi) = 0$.
To this aim, suppose first that $f\in \A$ and $\pi\in \Lip_b(E)$. Since $T(f\,d \1)=0$ by~\eqref{eq_Text_mass1}, and $T$ is linear in each argument,
it is sufficient to assume that $\pi = 0$ on $U$, thus $f \pi = 0$.
Let $\{\pi_n\} \subset \A$ be a sequence of uniformly bounded functions with $\lim_n \pi_n=\pi$ in $L^1(E,\mu+\nu)$, so that
in particular $\lim_n \pi_n(x)=\pi(x)$ for $\mu$-a.e.\ $x\in U$.
One has then
\begin{align*}
|T(f\,d\pi)| &= |T(\pi\,df) | =\lim_n |T(\pi_n\,df) |  \le \lim_n \int_E \lip\, f \, |\pi_n|\,  d \mu =
 \lim_n\int _{U^c} \lip\, f \, |\pi^n|\,  d \mu =0.
\end{align*}
For a generic $f\in \Lip_b(E)$, we approximate $f$ by $f_n\in \A$ as in~(ii), i.e.\ in the sense of $L^1(E,\mu)$ and so that $\supp\, f_n \subset U$. Since $T(f\,d\pi)=\lim_n T(f_n\,d\pi)$ whenever $\lim_n f_n= f$ in $L^1(E,\mu)$ by~\eqref{eq_Text_mass2}, and $\lim_n T(f_n\,d\pi)=0$ as just proven, then $T(f\,d\pi)=0$.
This proves the claimed locality.
We also have for every $\pi\in \Lip_b(E)$ and the sequence $\{\pi_n\}$ as in~(ii) the inequality
\begin{align*}
| T (f\, d\pi )|  = | \lim_n T (f\, d\pi_n )| & \le  \lim_n  \int_E \lip\, \pi_n \, |f|\,  d \mu \\
& \leq
 \lim_n \Lip\,\pi_n \int_E |f|\,  d \mu \leq C \Lip\,\pi \int_E |f|\,  d \mu.
\end{align*}
Thus
at this point, we note that if we equip $E$ with the new distance $\tilde d(x,y):=d(x,y)\wedge 1$ (denoting the resulting metric space $\tilde E$ for further reference), so that $\Lip(\tilde E)=\Lip_b(\tilde E)=\Lip_b(E)$, we already got that the extended functional $T$ is a normal current over $\tilde E$ (with mass measure $\mu_T\leq C\mu$ and the boundary mass measure
$\mu_{\partial T}\leq \nu$). In particular, it satisfies~\eqref{eq_mutdpi1} and hence, by Remark~\ref{rem_equivalent_Text_mass1}, also~\eqref{eq_Text_mass1} (of course, with all Lipschitz constants, both global and local, relative to $\tilde d$). But since local Lipschitz constant $\lip\,\pi$ is the same when calculated with respect to $\tilde d$ or to the original distance $d$, we have that~\eqref{eq_Text_mass1} holds in $E$ with the original distance.

Summing up, we have shown up to now that in both cases~(i) and~(ii) $T$ is uniquely extendable to a bilinear functional over $\Lip_b(E)\times\Lip_b(E)$ satisfying the claimed properties with $C\mu$ instead of $\mu$. To conclude the proof, it remains to extend it to $\Lip_b(E)\times \Lip(E)$ and verify that the extension is a normal current.
To this aim, for an arbitrary $\pi \in \Lip (E)$, setting $\pi^k:= (-k)\wedge \pi \vee k \in \Lip_b(E)$, we get for $m < n$ the estimate
\[ \lip (\pi^m-\pi^n) \leq \1_{\{ m\leq |\pi|\leq n \} } \, \lip\,  \pi\]
so that $\lip\, (\pi^m-\pi^n) \to 0$ pointwise as $(m,n) \to \infty$ and is uniformly bounded by $\Lip \, \pi$. This gives
\begin{align*}
    |T(f\,d\pi^m)-T(f\,d\pi^n)| \leq  C\int_E |f|\, \lip (\pi^m-\pi^n)\,d\mu
\end{align*}
so that that the sequence $\{T(f\,d\pi^k)\}$ is convergent, and we denote its limit by
$T(f\,d\pi)$, defining the extension of $T$
to $\Lip_b(E)\times \Lip (E)$, still denoted by $T$.
One also has that $\lim_k \lip\, \pi^k(x)=\,\lip\, \pi(x)$, and therefore,
\[
|T(f\,d\pi)|=|\lim_k T(f\,d\pi^k)| \leq \lim_k C\int_E |f|\, \lip \pi^k\,d\mu = C\int_E |f|\, \lip \pi\,d\mu,
\]
which gives both the locality property of currents and the mass estimate
\[
|T(f\,d\pi)|\leq C\Lip\,\pi \int_E |f| \,d\mu.
\]
Since also~\eqref{eq_Text_bdmass1} passes to the limit, then the extended $T$ is a normal current as claimed.
\end{proof}

\subsection{Extension of the vector field}\label{subsec_extV1}

Once we know that $T'$ is a current, we may define the ``space-time velocity'' determining the movement of each particle in the flow.  Namely by~\cite[proposition~A.7]{PaoSte14-flow} there is a linear map
$W\colon \Lip(E\times [0,1])\to L^1(E\times [0,1],\mu_t\otimes dt)$
satisfying the Leibniz rule such that
\begin{align} \label{eq:defin-W}
    T'(f\,d\pi)&=\int_0^1\,dt \int_E f(x,t) (W\pi)(x,t)\, d\mu_t(x),\\
    \MM(T') & = \int_0^1\,dt \int_E |W|(x,t) d\mu_t(x),\label{eq:defin-W-abs}
\end{align}
where $|W|$ is defined as the Radon-Nikodym derivative
\[
|W|(x,t):= \frac{d\mu_{T'}}{d(\mu_t\otimes dt)} \le C (1+\overline V) 
\]
for some $C>0$,
so that the map $\pi \mapsto W(\pi)/|W|$ (with the convention $0/0:=0$) is well-defined as a $\mu_{T'}=|W|\,d\mu_t\otimes dt$-measurable vector field in the sense of Weaver.

To keep track of the $E$-component of $W$, mainly for aesthetic purpose we may
extend the vector field $V_t$ defined initially only over the algebra $\A$
to a vector field $\tilde V_t$ defined over some larger class of functions (to be more precise, one has to extend the vector fields over space-time $E\times [0,1]$ and not just pointwise for a.e. instant of time) by the formula
\begin{equation}\label{eq:defin-tilde-V}
\tilde V \pi:= W\pi-\partial_t \pi.
\end{equation}
Note however that $(\tilde V\pi)(x,t)$ cannot a priori be defined for $\mu_t\otimes dt$-a.e. $(x,t)$
for an arbitrary $\pi\in \Lip(E\times [0,1])$, because $\partial_t \pi$ may not exist $\mu_t\otimes dt$-a.e.\ (an example is again that of Remark~\ref{rem:not-weaver}). This means that $\tilde V$ cannot be defined over
the whole $\Lip_b(E\times [0,1])$ but only on some proper subclasses; a reasonable (but of course not the only possible) choice is to define it over
$\Lip(E)\otimes \Lip([0,1])$ or $\Lip(E)\otimes C^1([0,1])$.
In this way
$\tilde V$ satisfies the Leibniz rule and one can write
\[    T' (f\,d\pi)=
  \int_0^1\, dt \int_E f(x,t) \left(\tilde V_t\pi(\cdot,t)(x)+\partial_t \pi(x,t)\right)\,d\mu_t(x),
\]
where $\tilde V_t(x):= \tilde V(x,t)$. Note that $\tilde V_t$ may now be considered as defined on the algebra
$\Lip_b(E)$ for a.e.\ $t\in [0,1]$.
In what follows, in Theorem~\ref{th_fl2T'repr2} we define its ``modulus''  $|\tilde V_t|$
by setting $|\tilde V_t|(x) :=(|W|(x,t)^2-1)^{1/2}$, i.e.\ so that $|W|= (1+|\tilde V_t|^2)^{1/2}$.
A useful alternative 
characterization of
$|\tilde V_t|$ is provided by Proposition~\ref{prop_modVelest1a}.

\subsection{Acyclicity and direction along time}\label{ssec_acycl1}

Our next step is to study some properties of the current $T'$, in particular with respect to the ``time'' component $t \in [0,1]$. To state them, let us first notice a general decomposition for currents $S\in \M_1(E\times [0,1])$. If we let
$\tau\colon E\times[0,1]\to \R$,  $\tau(t):=t$, then one has
\[
|S(f\,d\tau)|\leq \int_{E\times [0,1]} |f|\, d\mu_S
\]
for every $f\in L^1(E\times [0,1], \mu_S)$,
since $\Lip\, \tau=1$, thus, for some $a_S\in L^1(E\times [0,1], \mu_S)$, $|a_S(x,t)|\leq 1$ for $\mu_S$-a.e.\ $(x,t)$, we have
\[
S(f\,d\tau)=\int_{E\times [0,1]} f a_S\, d\mu_S.
\]
Let now
\[
S^{\leq}:= S\res \{a_S\leq 0\}, \qquad S^{>}:= S\res \{a_S> 0\}.
\]

\begin{remark}\label{rem_leqcommres1}
Since $a_{S\res B} = a_S\cdot \1_B$, then
$S^{\leq}\res B = (S\res B)^\leq$ and analogously  $S^{>}\res B= (S\res B)^>$
for all Borel sets $B \subset E\times [0,1]$.
\end{remark}

In this section, we show the following result which is valid for generic currents over space-time
such that the underlying vector field has a strictly positive component in time direction; this condition is
clearly satisfied also by $T'$ defined by \eqref{eq_fl2defT'}.

\begin{proposition}\label{prop_fl2T'Sleq0}
Let $T\in \M_1(E\times [0,1])$
satisfy $T^\leq =0$ and let $S\in \M_1(E\times [0,1])$ be such that $S\leq T$. Then $S^\leq =0$.
\end{proposition}

\begin{remark}\label{rem_forward1}
In the language of derivations, the condition $T^\leq =0$ means that $(W\tau)(x,t)>0$ for $\mu$-a.e.\ $(x,t)\in E\times [0,1]$, where $W\colon \Lip(E\times [0,1])\to L^1(E\times [0,1], \mu)$ and the nonnegative Borel measure $\mu$ over $E\times [0,1]$ are such that
\[
T(f\,d\pi)= \int_{E\times [0,1]} f W\pi \, d\mu\quad \mbox{ for all $f\,d\pi\in D^1(E\times [0,1])$}.
\]
Intuitively, this is an intrinsic way of expressing the notion of a vector field (or a current) ``directed forward in time''.
\end{remark}

\begin{remark}\label{rm_T'acycl1}
Of course, the conditions of the above proposition are satisfied for $T'$ (with $\mu:=\mu_t\otimes dt$), because
in this case, recalling that $V_t\tau=0$ by~\eqref{eq_VtLip1}, we get
\[
T'(f\,d\tau)= \int_{E\times [0,1]} f W\tau \, d\mu_t\otimes dt = \int_{E\times [0,1]} f\, d\mu_t\otimes dt
\]
for all $f\in L^1(E\times [0,1], \mu_t\otimes dt)$, and hence $W\tau=1$ $\mu_t\otimes dt$-a.e.
Thus the conclusion holds for $T'$ and also for $T'_t$,
since $S\leq T'_t$ implies $S\leq T'$ in view of
$T'_t\leq T'$.
\end{remark}

\begin{remark}\label{rem:product-cycles}
The intuition behind Proposition~\ref{prop_fl2T'Sleq0} becomes clear, if we think that $S^\leq$ should correspond to the subcurrent represented as a superposition of curves which go backward in time (i.e., whose $t$-component is nonincreasing), while $T'_t$ by construction should be a superposition of curves whose  $t$-component is strictly increasing. Unfortunately, this intuition does not take into account that the notion of subcurrent depends on the choice of the distance and, if we endow the product space $E \times[0, 1]$ with a distance different from $d_2$, we may obtain subcurrents of $T'$ that are cycles. As an explicit example, let $E := \R^2$, $\mu_t$ be a uniform distribution on the circle $\mathbb{S}^1$, which we may regard as a solution of the continuity equation with respect to a constant unit vector field $V := \partial_\theta$. If we endow the space $E\times [0,1]$ with the distance $d_1$,
then the cycle given by the ``cylindrical'' current  $\ld \mathbb{S}^1\rd$ seen as a one-dimensional current
over $E \times [0,1]$ (i.e.\ corresponding to the vector field $(V,0)$) is a subcurrent of $T'$.
\end{remark}

In view of Remark~\ref{rem:product-cycles}  our proof must take into account the fact that we endow $E \times[0,1]$ with the distance $d_2$. This is done in two steps: first, we argue on finite dimensional approximations of $E$, induced by the isometric embedding of $E$ into $\ell^\infty$, and then pass to the limit.

The following lemmata will be useful.

\begin{lemma}\label{lm_PnSleq1}
Let $(E,d)$ and  $(F,\rho)$ be metric spaces, let $Q\in \M_1(E \times [0,1])$,
$\phi \colon E \to F$ be Lipschitz and define $\Phi \colon E \times [0,1] \to F \times [0,1]$ by $\Phi(x,t) := (\phi(x), t)$. Then, one has
\[(\Phi_{\#}(Q^\leq))^\leq = \Phi_{\#}(Q^\leq).
\]
\end{lemma}

\begin{proof}
In fact
\begin{align*}
\int_{F\times [0,1]} f a_{\Phi_{\#}Q}\, d\mu_{\Phi_{\#}Q} & = \Phi_{\#}Q (f\,d\tau) = Q(f\circ \Phi\, d\tau\circ\Phi)\\
& = Q(f\circ \Phi\, d\tau)\qquad\mbox{ because $\tau\circ \Phi=\tau$}\\
&= \int_{E\times [0,1]} f\circ \Phi\, a_Q\, d\mu_Q = \int_{F\times [0,1]} f \, d \Phi_{\#} (a_Q\mu_Q)
\end{align*}
for every $f\in C_b(F\times [0,1])$,
which means
\begin{equation}\label{eq_PnSleq2}
a_{\Phi_{\#}Q}\mu_{\Phi_{\#}Q} = \Phi_{\#} (a_Q\mu_Q).
\end{equation}
Since~\eqref{eq_PnSleq2} is valid for every $Q\in \M_1(E\times [0,1])$, then we may apply it
with $Q^\leq$ instead of $Q$, thus getting
\[
a_{\Phi_{\#}(Q^\leq) }\mu_{\Phi_{\#}(Q^\leq)} = \Phi_{\#} (a_{Q^\leq}\mu_{Q^\leq}) \leq 0,
\]
because $a_{Q^\leq}(x)\leq 0$ for $\mu_Q$-a.e.\ $x$ and hence for $\mu_{Q^\leq}$-a.e.\ $x$.
This means
\[
a_{\Phi_{\#}(Q^\leq) }(x)\leq 0
\]
 for  $\mu_{\Phi_{\#}(Q^\leq)}$-a.e. $x$, thus showing the claim.
\end{proof}

\begin{lemma}\label{lm_muPnapprox1}
Let $C\subset \ell^\infty$ be the closure of some $\sigma$-compact set.
There is a sequence of finite rank linear projection operators $P_n\colon \ell^\infty\to \ell^\infty$
of unit norm for all $n\in\N$ and $\lim_n\|P_n x-x\|=0$
for all $x\in C$.
Therefore,
for every finite Borel measure $\mu$ over $\ell^\infty$ with $\supp\mu\subset C$ one has
$P_{n\#}\mu\rightharpoonup \mu$ in the narrow sense of measures as $n\to \infty$.
\end{lemma}

\begin{proof}
One has that $\bigcup_{\nu=1}^\infty K_\nu$ is dense in $C$,
where $\{K_\nu\}$ is an increasing sequence of compact subsets of $\ell^\infty$.
By lemma~5.7 from~\cite{PaoSte11-acycl} for every $\nu\in\N$ there is a
a finite rank linear projection operator $P_n$ with norm $| \| P_n \| | \leq 1$ such that $\|P_n x-x\| <1/n$
for all $x\in K_n$. Thus for every $x\in \cup_{\nu=1}^\infty K_\nu$ there is a $\nu=\nu(x)\in \N$ such that
$\|P_n x-x\| <1/n$ for all $n\geq \nu$.
Finally, if $x\in \overline{\cup_{\nu=1}^\infty K_\nu}$, then taking a sequence
$\{x_k\}\subset \cup_{\nu=1}^\infty K_\nu$ with $\lim_k x_k=x$, we get
\begin{align*}
    \|P_n x -x\| & \leq \|P_n x -P_n x_k \| + \|P_n x_k -x_k\| +\| x_k-x\|\leq  \|P_n x_k -x_k\| + 2\| x_k-x\|.
\end{align*}
For an arbitrary $\varepsilon>0$ let $k=k(\varepsilon)\in \N$ be such that
$\| x_k-x\|\leq \varepsilon/3$, and let then $\nu=\nu(k)\in \N$ such that $\|P_n x_k -x_k\|\leq \varepsilon/3$
for all $n\geq \nu$. Thus one has $\|P_n x -x\|\leq \varepsilon$ for all $n\geq \nu$,
which concludes the proof.
\end{proof}

\begin{proof}[Proof of Proposition~\ref{prop_fl2T'Sleq0}]
We may assume without loss of generality that $(W\tau)(x,t)=1$ for $\mu$-a.e. $(x,t)\in E\times [0,1]$
(otherwise, substitute the map $W$ with $\pi\in \Lip(E)\mapsto W\pi/W\tau$ and $\mu$ by $W\tau\,\mu$).
As usual, we define
$|W|$ as the Radon-Nikodym derivative
$|W|(x,t):= \frac{d\mu_{T}}{d\mu}$, and observe that
\[
\int_{E\times [0,1]} f\, d\mu = \int_{E\times [0,1]} f W\tau \, d\mu = T(f\,d\tau) \leq \int_{E\times [0,1]} f\, d\mu_T
= \int_{E\times [0,1]} f|W|\, d\mu
\]
for every $f\in \L^\infty(E\times [0,1],\mu)$, $f\geq 0$, and thus $\mu$-a.e.\ one has $|W|\geq 1$.

Up to embedding $E$ isometrically into $\ell^\infty$, we may assume without loss of generality that $E=\ell^\infty$. First, we notice that it is sufficient to show that
\begin{equation}\label{eq_resAlam1}
\MM((S\res \{|W|\leq\lambda  \})^\leq) =0, \quad \text{ for all $\lambda \geq 0$. }
\end{equation}
Indeed, by Remark~\ref{rem_leqcommres1}, we have $(S\res \{|W|\leq\lambda  \})^\leq = S^\leq \res \{|W|\leq\lambda  \}$
for every $\lambda \geq 0$, and hence
\[
\MM(S^\leq)=\sup_{\lambda > 0} \MM(S^\leq \res \{|W|\leq\lambda  \})
=\sup_{\lambda\geq 0} \MM((S\res \{|W|\leq\lambda  \})^\leq) =0.
\]

Lemma~\ref{lm_muPnapprox1} provides a sequence of linear operators $P_n\colon \ell^\infty\to \ell^\infty$ of unit norm and finite dimensional range $E_n\subset \ell^\infty$,
satisfying $P_n x\to x$ as $n\to \infty$, for every $x\in \supp T'$, which we naturally extend to operators on $\ell^\infty \times [0,1]$ by $P_n(x,t) := (P_n(x), t)$
(denoted by the same symbol). We notice that
\[
\MM(P_{n\#}(S\res \{|W|\leq\lambda  \})^\leq)\nearrow \MM(S\res \{|W|\leq\lambda  \})^\leq)
\]
as $n\to +\infty$, so \eqref{eq_resAlam1} follows once we show
\begin{equation}\label{eq_resAlam2}
\limsup_n \MM(D_n) =0, \mbox{ where } D_n :=P_{n\#}((S\res \{|W|\leq\lambda  \})^\leq)
\end{equation}

Similarly, let 
 $T_n:=P_{n\#}(T'\res \{|W|\leq\lambda  \})$.
 Clearly,
\begin{equation}\label{eq_SleqT'1}
(S\res \{|W|\leq\lambda  \})^\leq \leq S\res \{|W|\leq\lambda  \} \leq T'\res \{|W|\leq\lambda  \},
\end{equation}
and thus, although
 $D_n$ is not necessarily a subcurrent of $T_n$, still for $n \to \infty$
 we have
\begin{equation}\label{eq_quasisubS1}
\MM (D_n) +\MM(T_n-D_n)  = \MM(T_n) + o(1),
\end{equation}
because $\MM(P_{n\#}Q)\nearrow \MM(Q)$ for every $Q\in \M_1(E\times [0,1])$, $Q\leq T$.
Finally, we introduce the measure
\[
\mu_n := P_{n\#} (\mu \res \{|W|\leq\lambda \}).
\]
We claim now that we can represent
\begin{equation}\label{claim_T}  T_n= (V_n \circ P_n, 1)\wedge \mu_n, \, \text{ with $V_n\colon E_n\times [0,1]\to E_n$ Borel, $\| V_n\|\leq \lambda$ $\mu_n$-a.e., and}
\end{equation}
\begin{equation}\label{claim-D} D_n= (\alpha_n \circ P_n, \beta_n)\wedge \mu_n, \, \text{with $\alpha_n \colon E_n\times [0,1]\to E_n$ and $\beta_n\colon E_n\times [0,1]\to (-\infty, 0]$,}
\end{equation}
i.e., $\beta_n(x,t) \le 0$ for $\mu_n$-a.e. $(x,t)$
(note that throughout this proof we denote by $\|\cdot\|$ both the norm in $\ell^\infty$ and the induced norms in all its subspaces $E_n$).
Indeed, we have
\begin{equation}\label{eq_fmuTn1}
\begin{aligned}
\mu_{T_n} & \leq P_{n\#} (\mu_{T\res \{|W|\leq\lambda \}}) \leq
P_{n\#} \left(\1_{\{|W|\leq\lambda \}} |W| \mu \right)
\leq \lambda  P_{n\#} \left(\1_{\{|W|\leq\lambda \}} \mu \right) = \lambda \mu_n.
\end{aligned}
\end{equation}
Since $E_n$ (hence $E_n\times \R$) is finite-dimensional, by~\cite[lemma~A.3]{PaoSte14-flow} one has the representation
$T_n= (V_n, \sigma_n)\wedge \mu_n$ (where $T_n$ is seen as a current over $E_n\times [0,1]$)
for some $V_n\colon E_n\times [0,1]\to E_n$ and
 $\sigma_n\colon E_n\times [0,1]\to \R$ Borel maps.
Hence in particular,
\[
T_n(f\,d\tau)= \int_{E_n\times [0,1]} f\sigma_n\, d\mu_n
\]
for every $f\in C_b(E_n\times [0,1])$. On the other hand $\mu$-a.e.\ one has $W(\tau\circ P_n)=W\tau=1$,
and thus
\begin{align*}
T_n (f\,d\tau) & = (P_{n\#}(T\res \{|W|\leq\lambda \})) (f\,d\tau)
= \int_{E\times [0,1]} f \circ P_n W(\tau\circ P_n)\1_{\{|W|\leq\lambda \}}\, d\mu\\
& = \int_{E\times [0,1]} f \circ P_n \1_{\{|W|\leq\lambda \}}\, d\mu =
\int_{E_n\times [0,1]} f\, d\mu_n,
\end{align*}
so that
\[
\int_{E_n\times [0,1]} f\sigma_n\, d\mu_n =\int_{E_n\times [0,1]} f\, d\mu_n
\]
for every $f\in C_b(E_n\times [0,1])$,
i.e.\ $\sigma_n(x,t)=1$ for $\mu_n$-a.e.\ $(x,t)\in E_n\times [0,1]$. Hence we may write $T_n= (V_n, 1)\wedge \mu_n$ and \cite[lemma~A.1]{PaoSte14-flow} gives
\[
\mu_{T_n} = \sqrt{1+\|V_n\|^2}\mu_n \le \lambda \mu_n,
\]
where the latter inequality follows from~\eqref{eq_fmuTn1}, concluding the proof of \eqref{claim_T}.

To prove~\eqref{claim-D}, we notice that~\eqref{eq_SleqT'1}
gives the following inequality for mass measures
\[
\mu_{D_n} \leq P_{n\#}\mu_{(S\res \{|W|\leq\lambda  \})^\leq} \leq P_{n\#}\mu_{T'\res \{|W|\leq\lambda  \}}\leq
\lambda  \mu_n,
\]
where the last inequality of the above chain is prove in~\eqref{eq_fmuTn1}.
As before, since $E_n$ is finite-dimensional, by~\cite[lemma~A.3]{PaoSte14-flow} one has the representation
$D_n= (\alpha_n, \beta_n)\wedge \mu_n$ for some $\alpha_n\colon E_n\times [0,1]\to E_n$ and
 $\beta_n\colon E_n\times [0,1]\to \R$ Borel maps. In particular, for every $f\in C_b(E_n\times [0,1])$, one has
\[
 \int_{E_n\times [0,1]} f \beta_n d\mu_n = D_n(f\,d\tau)= \int_{E_n\times [0,1]} f a_{D_n}\, d\mu_{D_n}
\]
with $a_{D_n}\leq 0$, $\mu_{D_n}$-a.e.\ in $E_n\times [0,1]$, because $D_n^\leq = D_n$ in view of Lemma~\ref{lm_PnSleq1} (applied with $F:=E$, $\Phi:=P_n$).
Thus $\beta_n\mu_n = a_{D_n}\mu_{D_n}$ coincide, which gives $\beta_n(x,t) \le 0$ for $\mu_n$-a.e.\  $(x,t)\in E_n\times [0,1]$, proving~\eqref{claim-D}.

Now, plugging~\eqref{claim_T} and~\eqref{claim-D} into~\eqref{eq_quasisubS1}, we get
\begin{align*}
\int_{E_n\times [0,1]} (\beta_n^2 +\|\alpha_n\|^2)^{1/2}\,d\mu_n & +
\int_{E_n\times [0,1]} ((1-\beta_n)^2 +\|V_n-\alpha_n\|^2)^{1/2}\,d\mu_n \\
& \qquad \le
\int_{E_n\times [0,1]} (1 +\|V_n\|^2)^{1/2}\,d\mu_n + o(1)
\end{align*}
as $n\to \infty$. 
Since
\begin{align*}
((1-\beta_n)^2 +\|V_n-\alpha_n\|^2)^{1/2} & \geq ((1-\beta_n)^2 +(\|V_n\|-\|\alpha_n\|)^2)^{1/2}\\
&\geq  (1 +\|V_n\|^2)^{1/2} - \frac{\beta_n +\|\alpha_n\|\cdot\|V_n\|}{(1 +\|V_n\|^2)^{1/2}}
\end{align*}
(note that it is here that the choice of the distance $d_2$ is important),
we obtain
\begin{align*}
\int_{E_n\times [0,1]} \left( (\beta_n^2 +\|\alpha_n\|^2)^{1/2}
- \frac{\beta_n +\|\alpha_n\|\cdot\|V_n\|}{(1 +\|V_n\|^2)^{1/2}}
\right)
\,d\mu_n \le  o(1)
\end{align*}
as $n\to \infty$.
Denoting by $\gamma_n$ the angle between the vectors
$(\|V_n\|,1)\in \R^2$ and $(\|\alpha_n\|, \beta_n)\in \R^2$, it is immediate to observe that
\begin{align*}
(\beta_n^2 +\|\alpha_n\|^2)^{1/2}  &
- \frac{\beta_n +\|\alpha_n\|\cdot\|V_n\|}{(1 +\|V_n\|^2)^{1/2}} \\
& \qquad = \frac{(\beta_n^2 +\|\alpha_n\|^2)^{1/2}(1 +\|V_n\|^2)^{1/2} - (\|V_n\|,1)\cdot(\|\alpha_n\|, \beta_n)}{(1 +\|V_n\|^2)^{1/2}}\\
&\qquad
= (\beta_n^2 +\|\alpha_n\|^2)^{1/2} (1-\cos\gamma_n),
\end{align*}
and thus
\begin{align*}
\int_{E_n\times [0,1]} (\beta_n^2 +\|\alpha_n\|^2)^{1/2} (1-\cos\gamma_n)
\,d\mu_n \leq o(1)
\end{align*}
as $n\to \infty$.
But
\[
\cos \gamma_n \leq \frac{\|V_n\|}{(1 +\|V_n\|^2)^{1/2}} \leq \frac{\sqrt{\lambda^2-1}}{\lambda},
\]
(the first inequality being valid because $\beta_n \le 0$, the second one because $\mu_n$-a.e. one has $\|V_n\| \leq \lambda$)
Therefore the inequality
\begin{align*}
\frac{1}{\lambda } \MM(D_n) & =
\frac{1}{\lambda }\int_{E_n\times [0,1]} (\beta_n^2 +\|\alpha_n\|^2)^{1/2}\,d\mu_n\\
& \le \frac{\lambda  - \sqrt{\lambda^2-1}}{\lambda }\int_{E_n\times [0,1]} (\beta_n^2 +\|\alpha_n\|^2)^{1/2}
\,d\mu_n
\leq o(1).
\end{align*}
holds, and, for $\lambda>1$,
as $n \to \infty$, we deduce $\MM(D_n)\to 0$,  which shows~\eqref{eq_resAlam2}, hence concluding the proof.
\end{proof}

\begin{corollary}\label{co_fl2T'acycl}
Let $T\in \M_1(E\times [0,1])$ be  such that for $T^\leq =0$.
Then $T$ is acyclic. In particular,
$T'$ is acyclic (and hence so are $T_t'$ for all $t\in [0,1]$).
\end{corollary}

\begin{remark}\label{rm__fl2T'acycl_dist1}
The above corollary relies on Proposition~\ref{prop_fl2T'Sleq0} and hence on the particular choice of the distance
in $E\times [0,1]$. The claim becomes however independent of the latter if one substitutes the requirement $T^\leq =0$
by the more general one $S^\leq =0$ for every $S\leq T$.
\end{remark}

\begin{proof}
If $C\in \M_1(E\times [0,1])$, $C\leq T$ and $\partial C=0$, then  $C^\leq =0$ by Proposition~\ref{prop_fl2T'Sleq0},
so that $C= C^>$.
Thus
\[
0 =\partial C(\tau)= C(1\,d\tau)= C^> (1\, d\tau)=\int_{\{a_C>0\}} a_C d\mu_C,
\]
which means that $\mu_C(\{a_C>0\})=0$. On the other hand,
$\mu_C(\{a_C\leq 0\})=0$, because $C^\leq =0$, and therefore $\mu_C=0$, i.e.\ $C=0$ showing acyclicity of $T$.
The acyclicity of $T'$ follows now by
Remark~\ref{rm_T'acycl1}. All $T'_t$ are also acyclic as subcurrents of the acyclic current $T'$.
\end{proof}

\subsection{Decomposition in curves}

We show now the following statement regarding currents in space-time, which is in particular valid for the current
$T'$ in view of Proposition~\ref{prop_fl2T'Sleq0}.

\begin{proposition}\label{prop_fl2T'repr1}
Let $T\in \M_1(E\times [0,1])$ be a normal current such that $T^\leq =0$.
Then there is a finite Borel measure $\eta$ on $C([0,1]; E)$ concentrated over absolutely continuous curves,
such that for $T_t:= T\res\mathbf{1}_{[0,t)}$, $f\,d\pi\in D^1(E\times [0,1])$, one has
\begin{eqnarray}
\label{eq_fl2reprT't1}
    T_t (f\,d\pi)=
  \int_{C([0,1]; E)}\, d\eta(\theta)\int_0^t f(\theta(s),s) \frac{d\,}{ds}\pi(\theta(s),s)\,ds ,\\
\label{eq:mass-T'}
\MM(T_t)=  \int_{C([0,1]; E)}\, d\eta(\theta)\int_0^t \sqrt{1+|\dot{\theta}|^2(s)}\, ds.
\end{eqnarray}
If in particular $T=T'$ (hence $T_t=T_t'$), then also $e_{t\#}\eta=\mu_t$ for every  $t \in [0,1]$.
\end{proposition}

\begin{proof}
The current $T$ is acyclic by Corollary~\ref{co_fl2T'acycl}, and hence
by the representation theorem for acyclic one-dimensional normal metric currents~\cite[theorem~5.1]{PaoSte11-acycl}, there is a measure $\eta'$ over $C([0,1]; E\times[0,1])$ concentrated over injective Lipschitz curves $\theta'\in \Lip([0,1];E\times [0,1])$
such that
\begin{equation}\label{eq_fl2reprT't2_0}
\begin{aligned}
    T (f\,d\pi) &=
  \int_{C([0,1]; E\times[0,1])}\, d\eta'(\theta')\int_0^1 f(\theta'(s)) \frac{d\,}{ds}\pi(\theta'(s))\,ds,\\
    \MM(T) &=
  \int_{C([0,1]; E\times[0,1])}\, \ell(\theta') d\eta'(\theta'),
\end{aligned}
\end{equation}
and $e_{s\#}\eta'=\mu_s\otimes\delta_s$ for $s=0$ and $s=1$.

Next, we obtain $\eta$ by a suitable reparametrization of the curves selected by $\eta'$. We claim first that for $\eta'\otimes \mathcal{L}^1$ a.e.\ $(\theta', t)\in C([0,1]; E\times[0,1])\times [0,1]$
one has $\theta'(t)=(\theta(t), \sigma(t))$, with $\dot{\sigma}(t) >0$.
In fact, suppose the contrary, i.e.\ the existence of an $\alpha <0$ and of a  Borel set $\Delta \subset C([0,1]; E\times[0,1])\times [0,1]$
such that $(\eta'\otimes \mathcal{L}^1) (\Delta)>0$  and for every
$(\theta', t)\in \Delta$
one has $\theta'(t)=(\theta(t), \sigma(t))$ with $\dot{\sigma}(t) \leq \alpha$. The current $S\in \M_1(E\times [0,1])$ defined by
\[
  S(f\,d\pi) :=
  \int_{\Delta} f(\pi(\theta'(t))) \frac{d\,}{dt}\pi(\theta'(t)) \, d\eta'(\theta')\otimes dt
\]
satisfies $a_S \leq 0$, and hence $S^\leq =S$. On the other hand, $S\leq T$, and hence $S=0$ by Proposition~\ref{prop_fl2T'Sleq0}.
But $\MM(S)\geq |\alpha| (\eta'\otimes \mathcal{L}^1)(\Delta)>0$ which is a contradiction.

Now, since the representation~\eqref{eq_fl2reprT't2_0} is independent on reparametrization of the curves,
in view of the established claim,
we may assume without loss of generality that $\eta'$ is concentrated over curves
of the form $\theta'(t):=(\theta(t),t)$, with $\theta\in C([0,1];E)$.
Let now $r_t\colon \theta\mapsto \theta\res [0,t]$, and define $\eta_t :=r_{t\#}\eta$.
Then, thanks to the reparametrization, one has $T_{\eta_t}=T_t$, for $t \in [0,1]$, and in particular
\begin{equation}\label{eq_fl2reprT't2}
\begin{aligned}
 T_t (f\,d\pi) &=
  \int_{C([0,1]; E\times[0,1])}\, d\eta'(\theta')\int_0^t f(\theta(s), s) \frac{d\,}{ds}\pi(\theta(s), s)\, ds\\
 \MM(T_t) &=
 \int_{C([0,1]; E\times[0,1])}\ell(\theta'\res [0,t))\, d\eta'(\theta')
  = \int_{C([0,1]; E\times[0,1])}\, d\eta'(\theta')\int_0^t \sqrt{1+|\dot{\theta}|^2(s)}\, ds.
\end{aligned}
\end{equation}
Now,~\eqref{eq_fl2reprT't1} and~\eqref{eq:mass-T'} follow from~\eqref{eq_fl2reprT't2}, if we consider the map
$\hat{P}_E(\theta,\sigma)(t):= \theta(t)$ for every $t\in [0,1]$, $(\theta, \sigma)\in C([0,1];E \times [0,1])$, and define $\eta:=\hat{P}_{E\#}\eta'$.

If $T=T'$ (hence $T_t=T_t'$), then
\[
\mu_t\otimes\delta_t - \mu_0\otimes\delta_0=\partial T'_t= e_{1\#}\eta_t'- e_{0\#}\eta_t' = e_{t\#}\eta'- e_{0\#}\eta' = e_{t\#}\eta'-\mu_0\otimes\delta_0,
\]
give $\mu_t\otimes\delta_t = e_{t\#}\eta'= e_{t\#}\eta\otimes\delta_t$, and thus $e_{t\#}\eta=\mu_t$ as claimed.
\end{proof}

\begin{remark}\label{rem:from-curves-to-equation}
From the above Proposition~\ref{prop_fl2T'repr1}, we deduce that in fact \emph{every} normal current in space-time $T \in \M_1(E\times [0,1])$ satisfying $T^\le = 0$ defines naturally a family of Borel measures $\{\mu_t\}_{t \in [0,1]}$ and a family of maps $\{V_t\}_{t \in [0,1]}$ satisfying~\eqref{eq_VtLeib1}, \eqref{eq_VtLip1} and~\eqref{eq:V-bar-integrability} such that the continuity equation~\eqref{eq_fl2cont1} holds in duality with $\A := \Lip_b(E)$. Indeed, with the notation of the proof above, we let $\mu := (\beta_0)_\#( \eta' \otimes \mathcal{L}^1)$, where $\beta_0( \theta, t) := (\theta(t), t)$, and $W$ be such that $T (f\, d\pi) = \int_{E\times [0,1]} f W \pi d\mu$. Let $\mu_t$ be measures on $E$ such that $(e_t)_\# \eta = \mu_t$, and define $(V_t \pi)(x) :=( W \pi )(t,x)$, for $\pi \in \Lip_b(E)$. Then, for $g \in C^1_0(0,1)$, $\pi \in \Lip_b(E)$, we have
\begin{align*}
- \int_0^1 g'(t)  dt \int_E \pi d\mu_t & = - \int_{C([0,1];E)} d\eta(\theta) \int_0^1 g'(t)  \pi(\theta(t)) dt \\
& = \int_{C([0,1];E )} d\eta(\theta) \int_0^1 g(t)  \frac{d}{dt} \pi(\theta(t)) dt \\
& = T (g \, d\pi) \quad \text{by \eqref{eq_fl2reprT't2}}\\
& = \int_{E \times [0,1]} g \, W\pi d\mu = \int_{C([0,1];E )} d\eta(\theta) \int_0^1 g(t)  (W \pi)(\theta(t), t) dt\\
&= \int_0^1 g(t)  dt \int_E V_t\pi d\mu_t,
\end{align*}
so that the continuity equation~\eqref{eq_fl2cont1} holds.
\end{remark}

\subsection{Characteristic ODE in observables}

In this section we conclude the proof of Theorem~\ref{thm:sp}, showing the following results.

\begin{theorem}\label{th_fl2T'repr2}
Let $\{\mu_t\}$ and $\{V_t\}$, where $t \in [0,1]$, be a narrowly continuous curve of finite positive Borel measures and a family of time-dependent vector fields respectively, satisfying~\eqref{eq_VtLeib1},~\eqref{eq_VtLip1},~\eqref{eq:V-bar-integrability} and~\eqref{eq_fl2cont1}.
If the algebra $\A$ is such that the functional $T'$  defined by~\eqref{eq_fl2defT'} is extendable to a metric current
over $E\times [0,1]$, then the following assertions hold.

\begin{itemize}
\item[(i)] There is a Borel measure $\eta$ over $C([0,1]; E)$ concentrated over absolutely continuous curves
such that $e_{t\#}\eta=\mu_t$, and
for
$\nu_x^t$ conditional probability measures over $e_t^{-1}(x) = \{\theta: \theta(t) = x\} \subset C([0,1]; E)$
defined by disintegration $\eta=\mu_t\otimes \nu_x^t$
one has that
$|\dot{\theta}|(t)$ is constant for $\nu_x^t$-a.e.\ $\theta$, $\mu_t$-a.e.\ $x\in E$ and a.e.\ $t\in [0,1]$, and
\begin{equation}\label{eq_fl2repreta1mod1}
|\dot{\theta}|(t)= |\tilde V_t|(\theta(t)) \quad \text{for $\eta$-a.e.\ $\theta$ and
a.e.\ $t\in [0,1]$,}
\end{equation}
where $|\tilde V_t|(x) :=(|W|(x,t)^2-1)^{1/2}$ (i.e.\ $|W|= (1+|\tilde V_t|^2)^{1/2}$).
The above assertion holds in fact for every $\eta$ representing $T'$ without cancellation of mass and satisfying $e_{t\#}\eta=\mu_t$ for every $t \in [0,1]$.
\item[(ii)] Further, there is a measure $\eta$ as above satisfying in addition to~\eqref{eq_fl2repreta1mod1}
\begin{equation}\label{eq_fl2repreta1}
\frac{d\,}{dt} \pi(\theta(t))= (\tilde V_t\pi)(\theta(t))\qquad \text{for $\eta$-a.e.\ $\theta$ and
a.e.\ $t\in [0,1]$.}
\end{equation}
for every $\pi\in \Lip(E)$.
\end{itemize}

In particular, all the above assertions hold when $\A$ satisfies either~($\A_1$) or~($\A_2$) together with~($B$), i.e.\ under conditions of Theorem~\ref{thm:sp} (or, alternatively, the weaker conditions either~($\A_1'$) or~($\A_2'$) together with~($B'$)).
\end{theorem}


When $E$ can be bi-Lipschitz embedded into a Banach space with Radon-Nikodym property, we may provide a more precise description of the velocity of integral curves of~\eqref{eq_fl2repreta1} as follows.

\begin{theorem}\label{thm:radon-nikodym}
If under conditions of Theorem~\ref{th_fl2T'repr2} one has additionally the existence of a bi-Lipschitz embedding $j\colon E\to \hat{E}$ of $E$ into a Banach space $\hat{E}$ with Radon-Nikodym property, then
there is a Borel function $\hat V\colon E\times [0,1]\to \hat{E}$ such that
\[
\frac{d}{dt} j(\theta (t))= \hat V_t(\theta(t)) \quad \text{for $\eta$-a.e.\ $\theta$ and
a.e.\ $t\in [0,1]$,}
\]
where $\hat V_t(x):= \hat V(x,t)$.
\end{theorem}

\begin{remark}\label{rem_rn_null1}
Theorem~\ref{th_fl2T'repr2} affirms only that~\eqref{eq_fl2repreta1} holds for $\theta$ outside of some $\eta$-nullset of curves, possibly depending on the observable $\pi$.
On the contrary, in the situation of Theorem~\ref{thm:radon-nikodym} one has that~\eqref{eq_fl2repreta1}
holds for $\pi := \hat{\pi}\circ j$ with $\hat{\pi} \in Q^1(\hat{E})$ outside of some universal (i.e.\ independent of
of $\pi$) $\eta$-nullset of curves.
\end{remark}

Theorem~\ref{th_fl2T'repr2} implies the following corollary, particularly useful in case of autonomous
(i.e.\ time-independent) vector fields.

\begin{corollary}\label{co_currconteq0}
Under the conditions of Theorem~\ref{th_fl2T'repr2}, let $T$ be the functional on $D^1(E)$, defined by
\[
\begin{split}
T(f\,d\pi) &=  \int_0^1\,dt\int_E f(x)  (\tilde V_t\pi)(x)\,d\mu_t(x)
\end{split}
\]
Then, $T\in \M_1(E)$ is a normal current (in space only) with
\begin{equation}\label{ineq_massVt1}
\MM(T)\leq \int_0^1 \,dt \int_E |\tilde V_t|(x) \, d\mu_t(x).
\end{equation}
If $\eta$ is as in Theorem~\ref{th_fl2T'repr2}(ii), then $T=T_\eta$.
Moreover,
if $\tilde V_t/|\tilde V_t|$ is independent of $t$ in the sense that the exists a $\check{V}\colon \Lip_b(E)\to L^\infty(E,\mu_t)$
such that $(\tilde V_t\pi)(x,t)=(\check{V}\pi) (x,t)|\tilde V_t|(x,t)$ for all $\pi\in \Lip_b(E)$, the equality being intended
for $\mu_t\otimes dt$-a.e.\ $(x,t)\in E\times [0,1]$, then~\eqref{ineq_massVt1} becomes equality and $\eta$ represents $T$ without cancelation of mass.
\end{corollary}

\begin{proof}
Clearly $T = P_{E\#} T'$, where $P_E(x,t) := x$,
and hence $T\in \M_1(E)$ is a normal current, since so it $T'\in \M_1(E\times [0,1])$.
Then~\eqref{ineq_massVt1} is immediate from the definition of mass, since $|\tilde V_t\pi|\leq |\tilde V_t|\Lip\,\pi$
for all $\pi\in \Lip(E)$
as follows from Proposition~\ref{prop_modVelest1a}.
Further, for $\eta$ as in Theorem~\ref{th_fl2T'repr2}(ii) we get
\begin{align*}
T_\eta(f\,d\pi) &= \int_{C([0,1];E)} \ld\theta\rd(f\,d\pi)\, d\eta(\theta) =
\int_{C([0,1];E)} \, d\eta(\theta)\int_0^1 f(\theta(t)) \frac{d\,}{dt}\pi(\theta(t))\,dt \\
& = \int_{C([0,1];E)} \, d\eta(\theta)\int_0^1 f(\theta(t)) (\tilde V_t\pi)(\theta(t))\,dt \quad\text{by~\eqref{eq_fl2repreta1}}\\
&=\int_0^1 \, dt \int_{C([0,1];E)} f(\theta(t)) (\tilde V_t\pi)(\theta(t))\, d\eta(\theta)
 = \int_0^1\,dt\int_E f(x)  (\tilde V_t\pi)(x)\,d\mu_t(x),
\end{align*}
so that $T=T_\eta$.
If $\tilde V_t/|\tilde V_t|=:\check{V}$ is independent of $t$, the claimed equality in~\eqref{ineq_massVt1}
follows applying Lemma~\ref{lm_massestV1}, with
$\mu:=\int_0^1 |\tilde V_t|\mu_t\, dt$ and $\check{V}$ instead of $V$. Finally, in this case
\begin{align*}
\MM(T)&\leq \int_{C([0,1];E)} \ell(\theta)\, d\eta(\theta) =
\int_{C([0,1];E)} \, d\eta(\theta)\int_0^1 |\dot{\theta}|(t)) \,dt \\
& = \int_{C([0,1];E)} \, d\eta(\theta)\int_0^1 |\tilde V_t|(\theta(t))\,dt \quad\text{by~\eqref{eq_fl2repreta1mod1}}\\
&=\int_0^1\,dt\int_E |\tilde V_t|(x)\,d\mu_t(x) = \MM(T),
\end{align*}
so that $\eta$ represents $T$ without cancelation of mass as claimed.
\end{proof}

\begin{remark}\label{rem_currconteq1}
In general, inequality in \eqref{ineq_massVt1} can be strict, e.g.\
when $E:=\R$ and
\[
\begin{array}{ll}
V_t(x) =\left\{
  \begin{array}{ll}
     e_1, & t\in [0, 1/2]\\
     -e_1, & t\in (1/2, 1]
   \end{array}
\right.
&
\mu_t(x) =\left\{
  \begin{array}{ll}
     \delta_{t}, & t\in [0, 1/2]\\
     \delta_{1-t}, & t\in (1/2, 1]
   \end{array}
\right.
\end{array}
\]
one has $T=0$, while $\int_0^1 \,dt \int_E |V_t|(x) \, d\mu_t(x)=1$.
\end{remark}

The rest of the section contains the proofs of Theorems~\ref{th_fl2T'repr2} and~\ref{thm:radon-nikodym}.

\begin{proof}[Proof of Theorem~\ref{th_fl2T'repr2}.]
We prove both statements at once, in several steps.

{\em Step 1}.
Let $\eta$ be as in Proposition~\ref{prop_fl2T'repr1} (with $T:=T'$). Then
\begin{align*}
 \int_{E\times[0,1]}  f W\pi\,d\mu_t\,dt   & = T'(f\,d\pi) =
  \int_{C([0,1]; E)}\, d\eta(\theta)\int_0^1 f(\theta(t),t) \frac{d\,}{dt}\pi(\theta(t),t)\, dt\\
& =
\int_0^1 \,dt \int_E f(x,t) \,d\mu_t(x) \int_{e_t^{-1}(x)} \frac{d\,}{dt}\pi(\theta(t),t)\, d\nu_x^t(\theta).
\end{align*}
Since the above identities hold for every  bounded Borel $f: E\times [0,1] \to \R$, we deduce
\begin{equation}\label{eq:Wpi-disint1}
(W \pi)(x,t)= \int_{e_t^{-1}(x)}
\frac{d\,}{dt}\pi(\theta(t),t) \, d\nu_x^t(\theta).
\end{equation}
for $\mu_t\otimes dt$-a.e.\ $(x,t)\in E\times [0,1]$.
By Lemma~\ref{lm_massestV1}(iii) with $W$ instead of $V$, $E\times [0,1]$ instead of $E$, $T'$ instead of $T$ and $\mu_t\otimes dt$ instead of $\mu$, one has
\[
|W|(x,t) =
\sup_{\pi\in \mathcal{F}} (W\pi)(x,t) \quad \text{$\mu_t\otimes dt$-a.e.\ in $E \times [0,1]$,}
\]
whenever $\mathcal{F}\subset \Lip_1(E\times [0,1])$ is a countable set of functions with
$\mu_{T'}=\sup\,\{\mu_{T'\res d\pi}\colon \pi \in \mathcal{F}\}$.
This gives
\begin{equation}\label{eq_fl2pinuxt2a1}
\begin{aligned}
\MM(T') & =  \int_0^1\, dt \int_E |W|(x,t)\,d\mu_t(x)= \int_0^1\, dt \int_E \,d\mu_t(x) \sup_{\pi\in \mathcal{F}} (W\pi)(x,t) \\
&=
 \int_0^1\, dt \int_E \,d\mu_t(x) \sup_{\pi\in \mathcal{F}} \int_{e_t^{-1}(x)}
\frac{d\,}{dt}\pi(\theta(t),t) \, d\nu_x^t(\theta)\quad \text{by~\eqref{eq:Wpi-disint1}}\\
&\leq
 \int_0^1\, dt \int_E \,d\mu_t(x) \int_{e_t^{-1}(x)}
\sup_{\pi\in \mathcal{F}} \frac{d\,}{dt}\pi(\theta(t),t) \, d\nu_x^t(\theta)\\
& \leq
 \int_0^1\, dt \int_E \,d\mu_t(x) \int_{e_t^{-1}(x)}
 \sqrt{1+|\dot{\theta}|(t)^2} \, d\nu_x^t(\theta)\\
& =
\int_0^1\, dt \int_{C([0,1];E)} \,d\eta(\theta)
\sqrt{1+|\dot{\theta}|(t)^2}
= \MM(T')\quad \text{by  \eqref{eq:mass-T'}},
\end{aligned}
\end{equation}
so that in fact all the above inequalities are equalities. Thus
\begin{equation}\label{eq_fl2pinuxt3c}
\sup_{\pi\in \mathcal{F}}
\int_{e_t^{-1}(x)} \frac{d\,}{dt}\pi(\theta(t),t)\, d\nu_x^t(\theta)
 =
\int_{e_t^{-1}(x)} \sup_{\pi\in \mathcal{F}}
 \frac{d\,}{dt}\pi(\theta(t),t)\, d\nu_x^t(\theta)
=\int_{e_t^{-1}(x)}
\sqrt{1+|\dot{\theta}|(t)^2} \, d\nu_x^t(\theta).
\end{equation}

{\em Step 2}.
Suppose now that $E=X'$ is a Banach space dual to some separable Banach space $X$.
By Corollary~\ref{co_mcurr_densd1c} one may assume that $\mathcal{F}$ is a family of linear functionals of the
form $\pi(x)=\langle a, x \rangle$, for $a\in D$ where  $D$ is a countable dense subset of a unit ball $B_1(0)\subset X\times \R$.
By Lemma~\ref{lm_charODE_ban1}(A) with 
$X\times \R$ instead of $X$, $\R$ instead of $U$ and $E$ instead of $V$,
$\phi_{\#}\eta$ instead of $\eta$, $\phi_{\#}\nu_x^t$ instead of $\nu_x^t$, where
$(\phi(\theta))(t):= (\theta(t),t)$, one gets from~\eqref{eq_fl2pinuxt3c} for $\mu_t\otimes dt$-a.e.\ $(x,t)\in E\times [0,1]$
that $\sqrt{1+|\dot{\theta}|(t)^2}$
is constant
for $\phi_{\#}\nu_x^t$-a.e.\ $\theta'$ with $\theta'(t):=(\theta(t),t)$, hence
for $\nu_x^t$-a.e.\ $\theta$.
In particular, from~\eqref{eq_fl2pinuxt2a1} this gives
\[
\sqrt{1+|\dot{\theta}|(t)^2} = |W|(\theta(t),t)
\]
for $\eta$-a.e.\ $\theta\in C([0,1];E)$, which is~\eqref{eq_fl2repreta1mod1}.
Note that this is true whenever $\eta$ represents $T'$ without cancellation of mass
(and satisfies $e_{t\#}\eta=\mu_t$).

Further,~\eqref{eq_fl2pinuxt3c} becomes
\begin{equation}\label{eq_fl2pinuxt3d}
\sup_{\pi\in \mathcal{F}}
\int_{e_t^{-1}(x)} \frac{d\,}{dt}\pi(\theta(t),t)\, d\nu_x^t(\theta)
 =
\int_{e_t^{-1}(x)} \sup_{\pi\in \mathcal{F}}
 \frac{d\,}{dt}\pi(\theta(t),t)\, d\nu_x^t(\theta)
=
\sqrt{1+|\dot{\theta}|(t)^2},
\end{equation}
and therefore, if the norm in $E$ is strictly convex, then
by Lemma~\ref{lm_charODE_ban1}(B) applied again with 
$X\times \R$ instead of $X$ and
$\phi_{\#}\eta$ instead of $\eta$ (so that one has $\phi_{\#}\nu_x^t$ instead of $\nu_x^t$),
we get 
\begin{equation}\label{eq_fl2repreta2c}
\frac{d\,}{dt} \pi(\theta(t),t)= (W\pi)(\theta(t),t)
\end{equation}
for $\phi_{\#}\eta$-a.e.\ $\theta'$ with $\theta'(t):=(\theta(t),t)$ and
a.e.\ $t\in [0,1]$, hence
for $\eta$-a.e.\ $\theta$ and
a.e.\ $t\in [0,1]$.

{\em Step 3}.
To address the general case, we let $\hat{E}:=\ell^\infty$ 
and denote by
$j\colon E\to \ell^\infty$ the standard isometric embedding of $E$ into $\ell^\infty$.
Let $\eta$ be an arbitrary Borel measure over absolutely continuous
curves representing $T'$ without cancellation of mass and satisfying $e_{t\#}\eta=\mu_t$ for every $t \in [0,1]$. Then
$\psi_{\#}\eta$ with $\psi(\theta)(t):=j(\theta(t))$ represents in the same way $j'_{\#}T'$, where
$j'\colon E\times [0,1]\to \hat{E}\times [0,1]$, $j'(x,t):=(j(x),t)$, and $e_{t\#}\eta=j_{\#}\mu_t$,
so that, by Step~2
applied with $j'_{\#}T'$ instead of $T'$
we get
\begin{equation}\label{eq_metrvel1a}
|\dot{\hat \theta}|(t)= (|\hat W|^2(\hat \theta(t),t)-1)^{1/2} \quad \text{for $\psi_{\#}\eta$-a.e.\ $\hat\theta\in C([0,1];\hat{E})$ and
a.e.\ $t\in [0,1]$,}
\end{equation}
where $\hat W$ is defined by
\[
j'_{\#}T' (f\,d\pi)=\int_{\hat E\times [0,1]} f \hat W\pi\, d(j_{\#}\mu_t)\otimes dt.
\]
Since $j$ (hence also $j'$) is an isometry, then 
\begin{align*}
|\hat W|(j(x),t) & = |\hat W|(j'(x,t))= \frac{d\mu_{j'_{\#}T'}}{dj'_{\#}(\mu_t\otimes dt)} (j'(x,t))= \frac{dj'_{\#}\mu_{T'}}{dj'_{\#}(\mu_t\otimes dt)} \\
& = \frac{d\mu_{T'}}{d\mu_t\otimes dt} (x,t)= |W|(x,t),
\end{align*}
hence
one gets~\eqref{eq_fl2repreta1mod1}
holds for $\eta$ as claimed.

Further, we may equip $\ell^\infty$ with some strictly convex dual norm
equivalent to the standard one $\|\cdot\|_\infty$, 
so that
$j$ becomes bi-Lipschitz as a map into $\hat{E}:=\ell^\infty$.
Step~2 then gives the existence of some Borel measure $\hat \eta$  concentrated over absolutely continuous curves from $C([0,1];\hat{E})$ and
representing $j'_{\#}T'$ without cancellation of mass, with
$e_{t\#}\hat\eta=j_{\#}\mu_t$ (in fact, any measure with such properties will suit), such that
for every $\pi\in \Lip(\hat{E}\times [0,1])$ the equality
\begin{equation}\label{eq_fl2repreta3c}
\begin{aligned}
\frac{d\,}{dt} & \pi (\hat \theta(t),t) = (\hat W \pi)(\hat \theta(t),t) \quad \text{for $\hat\eta$-a.e.\ $\hat\theta\in C([0,1];\hat{E})$ and
a.e.\ $t\in [0,1]$}
\end{aligned}
\end{equation}
holds.
 Further, from the definition of $W$ and $\hat W$
one has
\begin{align*}
(\hat W \pi)(\hat x,t) & = (W (\pi\circ j'))(j^{-1}(\hat{x}),t) \quad \mbox{for $(j_{\#}\mu_t )\otimes dt$-a.e.\ $(\hat x,t)$}.
\end{align*}
Thus one gets~\eqref{eq_fl2repreta2c} with
$\eta:=\psi^{-1}_{\#}\hat\eta$ from~\eqref{eq_fl2repreta3c}.
At last,
the equality~\eqref{eq_fl2repreta1} follows now from~\eqref{eq_fl2repreta2c} with $\pi=\pi(x)$ (i.e.\ independent of $t$). \end{proof}

\begin{proof}[Proof of Theorem~\ref{thm:radon-nikodym}.] In this case, it is sufficient to modify the above proof
only in the last step, which we do without changing notation and letting $\hat{E}$ and $j$ be as in the statement being proven.
Since $\eta'$-a.e.\ curve $t\in[0,1]\mapsto (\theta(t),t)$ belongs to the support of $\mu_{T'}$, which is separable,
we may further restrict $\hat{E}$ to the image of the latter under the embedding $j$, so that $\hat{E}$ becomes separable, and, as such, may be endowed with the equivalent strictly convex norm (note that with this renorming $j$ remains bi-Lipschitz).
Now, since $\hat{E}$ has Radon-Nikodym property, by~\cite[proposition~5.1]{PaoSte14-flow} there is a Borel map
$\hat W\colon \hat{E}\times\R\to \hat{E} \times \R$ such that
\begin{align*}
    \hat T'(f\,d h)&=\int_0^1\,dt \int_{\hat{E}} f(\hat x,t) \langle \hat W(\hat x,t),  dh( \hat x,t)\rangle\, d\hat \mu_t(\hat x)
\end{align*}
whenever $h\in Q^1(\hat{E}\times \R)$ (in other words, with some abuse of the notation,
$(\hat Wh)(x,t)=\langle \hat W(\hat x,t),  d\hat \pi( \hat x,t)\rangle$),
$\langle\cdot,\cdot\rangle$ standing for the duality between $\hat{E}\times R$ and $\hat{E}'\times R$.  Moreover, we may write
\[
dh = \left(d_x h (\hat x,t), \partial_t h(\hat x,t)\right),
\]
so that the relationship
\[
(\hat V_t h)(\cdot,t)(\hat x)+\partial_t h(\hat x,t)= \langle \hat W(\hat x,t),  dh(\hat x,t)\rangle
\]
valid for $\hat\mu_t\otimes dt$-a.e.\ $(\hat x, t)$ provides the existence of a Borel function
$\hat V \colon \hat{E}\times [0,1]\to \R$ satisfying
$(\hat V_t h)(\cdot,t)(\hat x)= \langle \hat V(\hat x,t),  d_x h(\hat x,t)\rangle$
$\hat\mu_t\otimes dt$-a.e.,
now $\langle\cdot,\cdot\rangle$ standing for the duality between $\hat{E}$ and $\hat{E}'$.
Thus in particular by Lemma~\ref{lm_charODE_ban1}(C) for an arbitrary $h$ from 
the unit ball of $\hat{E}'$, from
\begin{align*}
\left\langle \frac{d\,}{dt} j(\theta(t)), h\right\rangle &=
\frac{d\,}{dt} \langle j(\theta(t)), h\rangle  = (\hat V_t h)(j(\theta(t)))= \langle \hat V_t (j(\theta(t))), h\rangle,
\end{align*}
which holds for $\eta\otimes dt$-a.e.\ $(\theta,t)\in C([0,1];E)\times [0,1]$ (independent of $h$),
we get the thesis for $\hat V_t\circ j$ instead of $\hat V_t$, so to conclude it is enough
to rename $\hat V_t$.
\end{proof}

Below the lemmata used in the above proofs are collected. 
%
%
%
The key auxiliary assertion for our construction is Lemma~\ref{lm_charODE_ban1} below
which is also of some independent interest. In fact, as a side corollary of this Lemma (and in particular of Lemma~\ref{lm_piauxcurrT1} used in the proof of the latter), we have that the properties of the measure $\eta$
we used in the proofs of our main results, namely, the relationship~\eqref{eq_fl2pinuxt3c}, and the fact that metric velocities $|\dot{\theta}|(t)$ of the curves are $\nu_x^t$-a.e.\ constant over $e_t^{-1}(x)$,  not only come from the fact that $\eta$ represents $T'$ without cancellation of mass, but in fact, are sufficient to define without cancellation of mass some normal $1$-current.

\begin{lemma}\label{lm_charODE_ban1}
Let $E$ be a Banach space.
Assume also that there is a measure $\eta$ over $C([0,1]; E)$ concentrated over absolutely continuous curves
satisfying 
$\eta=(e_{t\#}\eta)\otimes \nu_x^t$, $\nu_x^t$ being probability measures concentrated over
$e_t^{-1}(x)$ such that for $(e_{t\#}\eta)\otimes dt$-a.e. $(x,t)\in E\times [0,1]$
one has
\begin{equation}\label{eq_fl2pinuxt3c1}
\sup_{\pi\in \mathcal{F}}
\int_{e_t^{-1}(x)} \frac{d\,}{dt}\pi(\theta(t))\, d\nu_x^t(\theta)
 =\int_{e_t^{-1}(x)}|\dot{\theta}|(t)\, d\nu_x^t(\theta)
\end{equation}
for some countable family of functions $\mathcal{F}\subset \Lip_1(E)$.
The following statements hold.
\begin{itemize}
\item[(A)] If $E=X'$ is a dual of some separable Banach space $X$ and $\mathcal{F}$
is a family of linear functionals over $X'$ from $X$,
dense in the unit ball of $X$, and, moreover,
$E=U\times V$, with $U$ and $V$ Banach spaces with norms $\|\cdot\|_U$, $\|\cdot\|_V$ respectively, and the norm $\|\cdot\|$
in $E$ satisfies $\|(x_U,x_V)\|=\Phi(\|x_U\|_U,\|x_V\|_V)$ for some nondecreasing strictly convex function
$\Phi\colon \R^2\to \R^+$, then
for $e_{t\#}\eta$-a.e.\ $x\in E$ and a.e.\ $t\in [0,1]$ one has that
\begin{align*}
|\dot{\theta}_U|(t) & =\int_{e_t^{-1}(x)}|\dot{\sigma}_U|(t)\, d\nu_x^t(\sigma), \quad
|\dot{\theta}_V|(t)=\int_{e_t^{-1}(x)}|\dot{\sigma}_V|(t)\, d\nu_x^t(\sigma)
\end{align*}
for $\nu_x^t$-a.e.\ $\theta$, and in particular,
$|\dot{\theta}_U|(t)$, $|\dot{\theta}_V|(t)$ and hence $|\dot{\theta}|(t)$ are $\nu_x^t$-a.e.\ constant over $e_t^{-1}(x)$.
\item[(B)] If $E=X'$ is still a dual of some separable Banach space $X$ and $\mathcal{F}$
is a family of linear functionals over $X'$ from $X$,
dense in the unit ball of $X$, but now
$E$ is equipped with a strictly convex norm,
and $|\dot{\theta}|(t)$ is $\nu_x^t$-a.e.\ constant over $e_t^{-1}(x)$,
then
\begin{equation}\label{eq_fl2repreta1c}
\frac{d\,}{dt} \pi(\theta(t))= (W\pi)(\theta(t))
\quad \text{where }
(W\pi)(x,t):=
\int_{e_t^{-1}(x)}
\frac{d\,}{dt}\pi(\sigma(t)) \, d\nu_x^t(\sigma)
\end{equation}
for $\eta\otimes dt$-a.e.\ $(\theta,t)\in C([0,1];E)\times [0,1]$ (possibly depending on $\pi$)
and for every $\pi\in \Lip(E)$.
\item[(C)] Finally, if $E$ is a separable space with Radon-Nikodym property and $\mathcal{F}\subset Q^1(E)$, and
$|\dot{\theta}|(t)$ is $\nu_x^t$-a.e.\ constant over $e_t^{-1}(x)$, then~\eqref{eq_fl2repreta1c} holds for every $\pi\in Q^1(E)$. Moreover, in this case $\dot{\theta}(t)$ is $\nu_x^t$-a.e.\ constant over $e_t^{-1}(x)$, and the set of $(\theta,t)\in C([0,1];E)\times [0,1]$ of full measure $\eta\otimes dt$ such that~\eqref{eq_fl2repreta1c} holds may be chosen independent of $\pi\in Q^1(E)$.
\end{itemize}
\end{lemma}

\begin{proof}
Let
$E=X'$ be a dual of some separable Banach space $X$ and $\mathcal{F}$
be a family of linear functionals of the
form $\pi(x)=\langle a, x \rangle$
with $a\in D$ a countable dense subset of a unit ball $B_1(0)\subset X$.  Then
\begin{align*}
\sup_{\pi\in \mathcal{F}}  \int_{e_t^{-1}(x)} \frac{d\,}{dt}\pi(\theta(t))\, d\nu_x^t(\theta) & =
\sup_{a\in D} \int_{e_t^{-1}(x)} \left\langle
a, \dot{\theta}_w(t)\right\rangle\, d\nu_x^t(\theta) \\
&=
\sup_{a\in D} \left\langle a, \int_{e_t^{-1}(x)}
\dot{\theta}_w(t)\, d\nu_x^t(\theta) \right\rangle = \left\| \int_{e_t^{-1}(x)} \dot{\theta}_w(t)\, d\nu_x^t(\theta) \right\|^\prime,
\end{align*}
where $\dot{\theta}_w$ stands for the weak$^*$ derivative of $\theta$, the integrals
of $\dot{\theta}_w$ being intended in the weak$^*$ sense (see section~\ref{sec_wkstInt}).
Thus, from~\eqref{eq_fl2pinuxt3c1}, recalling that $|\dot{\theta}|(t)=\|\dot{\theta}_w(t)\|^\prime$ for
a.e.\ $t\in [0,1]$,
 one has
\begin{equation}\label{eq_normint1wk1}
\left\| \int_{e_t^{-1}(x)}\dot{\theta}_w(t)\, d\nu_x^t(\theta) \right\|^\prime =  \int_{e_t^{-1}(x)}\left\|\dot{\theta}_w(t)\right\|^\prime\, d\nu_x^t(\theta).
\end{equation}

In case~(A) we get by monotonicity and convexity of $\Phi$ the chain of estimates
\begin{align*}
\left\| \int_{e_t^{-1}(x)}\dot{\theta}_w(t)\, d\nu_x^t(\theta) \right\|^\prime & =
\Phi\left( \left\| \int_{e_t^{-1}(x)}(\dot{\theta}_w(t))_U\, d\nu_x^t(\theta) \right\|_U, \left\| \int_{e_t^{-1}(x)}(\dot{\theta}_w(t))_V\, d\nu_x^t(\theta) \right\|_V \right)\\
&\leq
\Phi\left(  \int_{e_t^{-1}(x)}\left\|(\dot{\theta}_w(t))_U\right\|_U\, d\nu_x^t(\theta) , \int_{e_t^{-1}(x)}\left\|(\dot{\theta}_w(t))_V\right\|_V\, d\nu_x^t(\theta)\right)\\
&\leq
\int_{e_t^{-1}(x)} \Phi\left( \left\|(\dot{\theta}_w(t))_U\right\|_U,
\left\|(\dot{\theta}_w(t))_V\right\|_V\right)\, d\nu_x^t(\theta)\\
&= \int_{e_t^{-1}(x)}\left\|\dot{\theta}_w(t)\right\|^\prime\, d\nu_x^t(\theta),
\end{align*}
where the second estimate follows from Jensen inequality.
In view of~\eqref{eq_normint1wk1} all the above inequalities are equalities, and minding strict convexity of $\Phi$, from the fact that Jensen inequality is an equality we get
\begin{equation}\label{eq_thetUV1}
\|(\dot{\theta}_w)_U\|_U(t)=\int_{e_t^{-1}(x)}\|(\dot{\sigma}_w(t))_U\|_U\, d\nu_x^t(\sigma), \quad
\|(\dot{\theta}_w)_V\|_V(t)=\int_{e_t^{-1}(x)}\|(\dot{\sigma}_w(t))_V\|_V\, d\nu_x^t(\sigma),
\end{equation}
and hence, minding that metric derivative of a curve is
a.e.\ equal to the norm of its weak derivative,
we get the claim~(A).

In case~(B), recalling the strict convexity of the norm in $E$, from~\eqref{eq_normint1wk1} with the help of Lemma~\ref{lm_revJensen_wkst1} (with $e_t^{-1}(x)$, $\nu_x^t$ and $\theta\mapsto \dot{\theta}_w(t)$ instead of $\Omega$,
$\mu$, and $f$ respectively), together
with the fact that $\|\dot{\theta}_w(t)\|^\prime = |\dot{\theta}|(t)$
is constant over $e_t^{-1}(x)$, we obtain
\[
\frac{d\,}{dt}\pi(\theta(t)) = \int_{e_t^{-1}(x)} \frac{d\,}{dt}\pi(\sigma(t))\, d\nu_x^t(\sigma)
\]
for $\nu_x^t$-a.e.\ $\theta$ and all
$\pi(x)= \langle a,x \rangle +b$, $a\in D$, $b\in \R$. 
This gives~\eqref{eq_fl2repreta1c}
for all such $\pi$.
For $\pi(x)=\|x-y\|^\prime$
for every $y\in E$,
recalling Lemma~\ref{lm_mcurr_densd1b}, we apply Lemmata~\ref{lm_piwedge1} and~\ref{lm_pilim1}
to obtain~\eqref{eq_fl2repreta1c} for such $\pi$ ad hence also for $\pi=a+\|x-y\|^\prime$, $a\in \R$.
Now recalling Lemma~\ref{lm_mcurr_densd1a}(i) (with $g_j(x,y):=\|x-y\|$) we apply again Lemmata~\ref{lm_piwedge1} and~\ref{lm_pilim1}
to obtain~\eqref{eq_fl2repreta1c} for a generic $\pi\in \Lip_1(E)$.

Finally, in case~(C) recalling that $\mathcal{F}\subset Q^1(E)\cap\Lip_1(E)$, one has
\begin{align*}
\sup_{\pi\in \mathcal{F}}
\int_{e_t^{-1}(x)} \frac{d\,}{dt}\pi(\theta(t))\, d\nu_x^t(\theta) & = \sup_{\pi\in \mathcal{F}} \int_{e_t^{-1}(x)} \left\langle
\dot{\theta}(t), d\pi(\theta(t))\right\rangle\, d\nu_x^t(\theta) \\
&=
\sup_{\pi\in \mathcal{F}} \left\langle \int_{e_t^{-1}(x)}
\dot{\theta}(t)\, d\nu_x^t(\theta), d\pi(x) \right\rangle \leq \left\| \int_{e_t^{-1}(x)} \dot{\theta}(t)\, d\nu_x^t(\theta) \right\|,
\end{align*}
where $\dot{\theta}$ stands for the (strong) derivative of $\theta$,
which exists because of the Radon-Nikodym property,
the integrals
of $\dot{\theta}$ being now intended in the sense of Bochner (it is to see that $\dot{\theta}$ are Bochner integrable that we use separability of $E$).
On the other hand, from~\eqref{eq_fl2pinuxt3c1}, recalling that $\|\dot{\theta}(t)\| = |\dot{\theta}|(t)$ for
a.e.\ $t\in [0,1]$, and that the latter is constant for $\nu_x^t$-a.e.\ $\theta$,
we get
\begin{align*}
\int_{e_t^{-1}(x)}\left\|\dot{\theta}(t)\right\|\, d\nu_x^t(\theta) &\geq
\left\| \int_{e_t^{-1}(x)}\dot{\theta}(t)\, d\nu_x^t(\theta) \right\| \\
& \geq
\sup_{\pi\in \mathcal{F}} \int_{e_t^{-1}(x)} \frac{d\,}{dt}\pi(\theta(t))\, d\nu_x^t(\theta)  =
\int_{e_t^{-1}(x)}\left\|\dot{\theta}(t)\right\|\, d\nu_x^t(\theta),
\end{align*}
so that the above inequalities are in fact equalities and therefore we get~\eqref{eq_normint1wk1} with
$\dot{\theta}$ instead of $\dot{\theta}_w$.
By Bochner integrability of $\dot{\theta}$ and by strict convexity of the norm in $E$ one has from Remark~\ref{rem_bochn_conv1} that
$\dot{\theta}(t)$ is constant for $\nu_x^t$-a.e.\ $\theta$, which also shows~\eqref{eq_fl2repreta1c} for all $\pi\in Q^1(E)$
(the latter equality becomes in fact tautological here).
\end{proof}

The lemmata below have been employed in the proof of Lemma~\ref{lm_charODE_ban1}.

\begin{lemma}\label{lm_piauxcurrT1}
Under
condition~\eqref{eq_fl2pinuxt3c1}
for an arbitrary $t\in (0,1)$ the functional $T_t$ over $D^1(E)$ defined by
\begin{align*}
T_t(f\,d\pi) &:= \int_0^t\,ds \int_{E} f(x) (W\pi)(x,s)\, d (e_{s\#}\eta)(x),
\end{align*}
where $W$ is as in~\eqref{eq_fl2repreta1c},
is a normal current, represented by the measure $r_{t\#}\eta$
(where  $r_t\colon \theta\in C([0,1];E)\mapsto \theta\res [0,t]$)
without cancelation of mass.
\end{lemma}

\begin{proof}
One has
\begin{align*}
T_t(f\,d\pi) &:= \int_0^t\,ds \int_{E} f(x) (W\pi)(x,s)\, d (e_{s\#}\eta)(x) \\
& = \int_0^t\,ds \int_{E} f(x) \, d (e_{s\#}\eta)(x)
\int_{e_s^{-1}(x)}
\frac{d\,}{ds}\pi(\sigma(s)) \, d\nu_x^s(\sigma)\\
&=  \int_0^t\,ds \int_{C([0,1]; E)}\, d\eta(\theta) f(\theta(s)) \frac{d\,}{ds}\pi(\theta(s)) =
\int_{C([0,1]; E)}\ld\theta\rd(f\,d\pi)\, d\eta(\theta),
\end{align*}
so that $T_t\in \M_1(E)$. Since $\eta$ is finite, the $T_t$ is normal.
Further,
\begin{align*}
    \MM(T_t) & = \int_0^t\,ds \int_{E} \, d (e_{s\#}\eta)(x)\sup_{\pi\in \mathcal{F}}
\int_{e_s^{-1}(x)}
\frac{d\,}{ds}\pi(\sigma(s)) \, d\nu_x^s(\sigma)\\
& = \int_0^t\,ds \int_{E}  \, d (e_{s\#}\eta)(x)
\int_{e_s^{-1}(x)} |\dot{\sigma}|(s) \, d\nu_x^s(\sigma)\\
& = \int_0^t\,ds \int_{C([0,1];E)}  |\dot{\theta}|(s) \, d \eta(\theta)= \int_{C([0,1];E)}  \ell(\theta) \, d (r_{t\#}\eta)(\theta),
\end{align*}
concluding the proof.
\end{proof}

\begin{lemma}\label{lm_piwedge1}
Under conditions of Lemma~\ref{lm_charODE_ban1}(B),
if~\eqref{eq_fl2repreta2c} holds for $\pi\in \{\pi_1,\ldots, \pi_m\} \subseteq \Lip(E)$,
then it holds also for $\pi:=\bigvee_{i=1}^m \pi_i$.
\end{lemma}

\begin{proof}
It suffices to give the proof in case $m=2$ (the generale case following by induction). Let us assume that, for $i \in \cur{1,2}$ one has
\begin{equation}\label{eq_Vtpi12a0}
\frac{d\,}{dt} \pi_i(\theta(t))= (W\pi_i)(\theta(t),t), \quad \text{for
$\eta$-a.e.\ $\theta$ and
a.e.\ $t\in [0,1]$.}
\end{equation}
 To prove that~\eqref{eq_fl2repreta2c} holds
for $\pi:=\pi_1\vee \pi_2$, 
we use the chain rule for currents $T_t$ defined by Lemma~\ref{lm_piauxcurrT1}
\begin{align*}
    T_t(f\,d\pi) & = T_t(f\1_{\pi_1\geq\pi_2}\,d\pi_1) + T_t(f\1_{\pi_1<\pi_2}\,d\pi_2)\\
                &=  \int_0^t\, ds \int_{E} f(x) \left(\1_{\pi_1\geq\pi_2} (x)(W\pi_1)(x,s)
                + \1_{\pi_1 <\pi_2} (x)(W\pi_2)(x,s)\right)\,d(e_{s\#}\eta)(x),
\end{align*}
for every bounded Borel function $f$ over $E$ so that $e_{t\#}\eta\otimes dt$-a.e.\ in $E\times [0,1]$
one has
\begin{equation}\label{eq_Vtpi12a}
W\pi=\1_{\pi_1\geq\pi_2} (W\pi_1)
                + \1_{\pi_1 <\pi_2} (W\pi_2). 
\end{equation}
Since clearly $\mu_{T_1}\ll \int_0^1\mu_s\, ds$, then~\eqref{eq_Vtpi12a} holds
also $\mu_{T_1}\otimes dt$-a.e.
Hence,  for $\eta\otimes dt$-a.e.\ $(\theta,t)$ one has
\begin{equation}\label{eq_Vtpi12b}
(W\pi)(\theta(t),t)=\1_{\pi_1\geq\pi_2}(\theta(t)) (W\pi_1)(\theta(t),t)
                + \1_{\pi_1 <\pi_2}(\theta(t)) (W\pi_2)(\theta(t),t). 
\end{equation}
In fact, denoting by $\beta_0\colon C([0,1];E\times [0,1])\to E$ the map
$\beta_0(\theta,s):=\theta(s)$, we get
$\mu_{T_1}=\beta_{0\#} ((\eta\otimes dt))$
 by~\cite[example~4.5]{PaoSte14-flow}, because $\eta$ represents $T_1$ without cancellation of mass by Lemma~\ref{lm_piauxcurrT1},
 and hence
$\mu_{T_1}\otimes dt=(\beta_{0},\mbox{id})_{\#} (\eta\otimes dt)$.
Thus the set of pairs $(\theta,t)$ such that~\eqref{eq_Vtpi12b} fails is a subset of
$(\beta_{0},\mbox{id})^{-1}(e)$, where
$e\subset E\times [0,1]$ is the set of pairs $(x,t)$ such that~\eqref{eq_Vtpi12a} fails, the latter having
$(\mu_{T_1}\otimes dt)(e)=0$, hence
showing that~\eqref{eq_Vtpi12b} holds  for
$\eta\otimes\, dt$-a.e.\ $(\theta,t)\in C([0,1];E)\times [0,1]$.

On the other hand, for every absolutely continuous curve $\theta$ (in particular for
$\eta$-a.e.\ $\theta\in C([0,1];E)$),  one has
\begin{equation*}
\frac{d\,}{dt} \pi(\theta(t)) =\1_{\pi_1\geq\pi_2}(\theta(t)) \frac{d\,}{dt} \pi_1(\theta(t))
                + \1_{\pi_1 <\pi_2}(\theta(t)) \frac{d\,}{dt} \pi_2(\theta(t)) \quad \text{for a.e.\ $t\in [0,1]$.}
\end{equation*}
and to conclude the proof it suffices to compare the latter equality with~\eqref{eq_Vtpi12b} using~\eqref{eq_Vtpi12a0}.
\end{proof}

\begin{lemma}\label{lm_pilim1}
Under conditions of Lemma~\ref{lm_charODE_ban1}(B),
if~\eqref{eq_fl2repreta2c} holds for $\pi\in \{\pi_k\}_{k\ge 1} \subseteq \Lip(E)$,
and $\pi_k(x)\to \pi_\infty(x)$ for all $x\in E$ as $k\to +\infty$, with
$\Lip\,\pi_k$ uniformly bounded, then
then it holds also for $\pi_\infty$.
\end{lemma}

\begin{proof}
For a generic $\pi\in \Lip(E)$, the real function $t \mapsto\pi(\theta(t))$ is absolutely continuous, thus we may equivalently rewrite~\eqref{eq_fl2repreta2c} in the integral form
\begin{equation}\label{eq_fl2repreta1int1}
\int_0^t (W \pi)(\theta(s),s)\,ds = \pi(\theta(t))- \pi(\theta(0)), \quad \text{for every $t \in [0,1]$.}
\end{equation}
Letting $\Delta$ stand for an arbitrary Borel subset of $C([0,1];E)$, and integrating the above
equality in $\eta\res\Delta$, we get
\begin{align*}
    \int_\Delta (\pi(\theta(t))- \pi(\theta(0)))\, d\eta(\theta) & =
    \int_0^t \, ds\int_\Delta  (W\pi)(\theta(s),s)d\eta(\theta) \\
    & = \int_0^t \, ds\int_{E} (W\pi)(x,s) \nu_x^s(\Delta)\,d\mu_s(x),
\end{align*}
the latter equality being obtained by disintegration.
Recalling that $T_t$
defined in Lemma~\ref{lm_piauxcurrT1} is a current,
we get that $W\pi_k\rightharpoonup W\pi_\infty$ in the weak$^*$ sense of $L^\infty(E\times [0,1],\mu_t\otimes dt)$,
hence, since
$0\leq \nu_x^s(\Delta)\leq 1$,
both sides of the above relationship are continuous with respect
to pointwise convergence of Lipschitz functions with uniformly bounded Lipschitz constants. Since $\Delta$ and $t \in [0,1]$ are arbitrary, then also~\eqref{eq_fl2repreta1int1} holds in the limit for $\pi_\infty$.
\end{proof}

\subsection{Estimate of the modulus of velocity}\label{sec_modVelest1}

In this section we provide the following useful alternative characterization of the ``modulus''  $|\tilde V_t|$ (defined in
Theorem~\ref{th_fl2T'repr2})
of the velocity field $\tilde V_t$.

\begin{proposition}\label{prop_modVelest1a}
Under the conditions of Theorem~\ref{th_fl2T'repr2}, letting
$|\tilde V|$ stand for the minimum function $\beta \in L^1(E\times[0,1],\mu_t\otimes dt)$ satisfying
the estimate
\begin{equation}\label{eq:defin-tilde-V-norm}
 |(\tilde V \pi)(x,t)|
\leq   \beta(x,t) \lip_{\tilde E}\pi(x,t)\quad \text{for all $\pi\in \Lip(E)$},
\end{equation}
we have that
$|\tilde V_t|(x)= |\tilde V|(x,t)$ for $\mu_t\otimes dt$-a.e.\ $(x,t)\in E\times [0,1]$.
\end{proposition}

\begin{proof}
Let $\beta \in L^1(E\times[0,1],\mu_t\otimes dt)$ be such that~\eqref{eq:defin-tilde-V-norm} holds. Then by Lemma~\ref{lem:extension-tensor-like-space-time} the same estimate holds also for all $\pi \in \Lip(E)\otimes\Lip([0,1])$
with $\lip\,\pi$ on the right hand side substituted by $\lip_E\,\pi$.
Thus for such $\pi$ one has
\begin{align*}
    |(W\pi)(x,t)| & = |(V_t\pi(\cdot, t))(x) +\partial_t\pi (x,t)|\leq \beta(x,t) \lip_E\,\pi (x,t) + |\partial_t\pi(x,t)|\\
    & \leq \left(1+\beta(x,t)^2\right)^{1/2}
\left(\lip_E\pi(x,t)^2 +
\left|\partial_t \pi(x,t)\right|^2\right)^{1/2}
\end{align*}
for $\mu_t\otimes dt$-a.e.\ $(x,t)$.
But it is a standard calculation to prove that
\[
\left(\lip_E^*\pi(x,t)^2 +
\left|\partial_t \pi(x,t)\right|^2\right)^{1/2} \leq \Lip\,(\pi\res B_\rho((x,t))
\]
for every  $(x,t)\in \tilde E\times [0,1]$, $\rho>0$,
where $\lip_E^*\,\pi(x,t) := \limsup_{(\delta, s) \to (0^+, t) } \Lip( \pi \res ( B_\delta (x) \times \{ s\})  )$,
and hence
\[
\left(\lip_E\pi(x,t)^2 +
\left|\partial_t \pi(x,t)\right|^2\right)^{1/2}\leq
\left(\lip_E^*\pi(x,t)^2 +
\left|\partial_t \pi(x,t)\right|^2\right)^{1/2}\leq \lip\,\pi (x,t),
\]
which therefore gives
\[
  |(W\pi)(x,t)| \leq \left(1+\beta(x,t)^2\right)^{1/2}
\lip\,\pi (x,t)
\]
for $\mu_t\otimes dt$-a.e.\ $(x,t)$.
Let $j\colon E\to \ell^\infty$ stand for an isometric embedding, and let $\tilde{\mathcal{F}}$ stand for a dense family in the unit
ball of $\ell^1\times \R$, and
\[
\mathcal{F}:=\{\pi\in \Lip_1(E\times [0,1]\colon \pi(x,t)=\tilde\pi(j(x),t)\mbox{ for all $(x,t)$ and for some $\pi\in \tilde{\mathcal{F}}$}\}.
\]
Since $\tilde{\mathcal{F}}$ by Lemma~\ref{lm_mcurr_linseq1a}
realizes the mass of every current in $\M_1(\ell^\infty\times [0,1])$, then so does
$\mathcal{F}$ for every current in $\M_1(E\times [0,1])$. By Lemma~\ref{lm_massestV1}(iii) we have thus
$|W|=\sup_{\pi\in \mathcal{F}}|W\pi|$. But since $\mathcal{F}\subset \Lip(E)\otimes\Lip([0,1])$,
we have for each $\pi\in \mathcal{F}$ the estimate
\[
|W\pi|\leq \sqrt{1+\beta^2}\, \lip\,\pi \leq  \sqrt{1+\beta^2}\, \Lip\,\pi\leq  \sqrt{1+\beta^2},
\]
and hence $\beta\geq \sqrt{|W|^2-1}=: |V_t|$
for $\mu_t\otimes dt$-a.e.\ $(x,t)$.
On the other hand, for every $\pi\in \Lip_1(E)$ and for $\mu_t\otimes dt$-a.e.\ $(x,t)$ we have
\begin{align*}
|(\tilde V \pi )(x,t)|  =
| (W\pi) (x,t)|
& = \left|\int_{e_t^{-1}(x)}
\frac{d\,}{dt}\tilde\pi(\theta(t))\, d\nu_x^t(\theta)\right|\\
& \leq \int_{e_t^{-1}(x)} |\dot{\theta}|(t)
 \, d\nu_x^t(\theta) \quad \text{because $\Lip\,\pi\leq 1$}\\
& = \int_{e_t^{-1}(x)} |V_t|(\theta(t))
 \, d\nu_x^t(\theta)\quad \text{by~\eqref{eq_fl2repreta1mod1}}\\
 & = |V_t|(x) \int_{e_t^{-1}(x)}
 \, d\nu_x^t(\theta) = |V_t|(x),
\end{align*}
concluding the proof.
\end{proof}

\section{Structure of absolutely curves of measures}\label{sec_Lisini1}

In this section 
we let $W_p$, $p \in [1,\infty]$, stand for the
the $p$-Kantorovich-Wasserstein distance in $\mathcal{P}_p(E)$, the latter standing for the set of Borel probability measures $\mu$ on $E$, having finite $p$-moment $\int_E d(x,x_0)^p d\mu <\infty$ for some $x_0 \in E$. Given any curve $\{\mu_t\}_{t\in [0,1]} \subseteq \mathcal{P}_p(E)$ that is also absolutely continuous with respect to $W_p$, we denote by $|\dot{\mu}_t|_p$
its metric derivative with respect to $W_p$ at $t\in [0,1]$.

\begin{theorem}\label{th_Lisini1}
Let $1<p \leq \infty$ and $\{\mu_t\}_{t\in [0,1]} \subseteq  \mathcal{P}_p(E)$ be an absolutely continuous curve  with respect to $W_p$, such that $|\dot{\mu}|_p \in L^p(0,1)$. Then there is a family of linear operators $\{V_t\}_{t\in [0,1]}$ on $\Lip(E)$  satisfying~\eqref{eq_VtLeib1} and~\eqref{eq_VtLip1}
with $\A:= \Lip_b(E)$, the families
$\{\mu_t\}_{t\in [0,1]}$, $\{V_t\}_{t\in [0,1]}$ solving
the continuity equation~\eqref{eq_fl2cont1}.

Moreover, there is a measure $\eta$ over $C([0,1]; E)$ concentrated over absolutely continuous curves for which $(e_t)_\# \eta = \mu_t$ for every $t \in [0,1]$ and the ODE in observables~\eqref{eq_fl2repreta1mod1} hold. Every such measure $\eta$ satisfies
\begin{equation}\label{eq_lis_metrvel1}
|\dot{\mu}_t|_p^p = \int_{C([0,1]; E)} |\dot \theta|^p(t) d\eta(\theta) \quad \text{for a.e.\ $t\in [0,1]$.}
\end{equation}
\end{theorem}

We remark that the original result by S.~Lisini for generic metric spaces gives only the existence of a measure
$\eta$ over $C([0,1]; E)$ concentrated over absolutely continuous curves for which $(e_t)_\# \eta = \mu_t$
for every $t \in [0,1]$, and~\eqref{eq_lis_metrvel1} holds. The above Theorem~\ref{th_Lisini1} therefore
gives a more precise information on the structure of curves in $\mathcal{P}_p$ absolutely continuous with respect to the Kantorovich-Wasserstein distance.

\begin{proof}
Since the statement is invariant with respect to
isometric
embeddings of the space $E$, we may assume $E$ be a separable closed linear subspace of $\ell^\infty$,  endowed with the inherited norm
$\|\cdot\|_\infty$. We let  $\A := \mathrm{Cyl}(E)\cap C^1(E)\cap \Lip_b(E)$ and notice that for  $\pi \in \A\otimes PC^1([0,1])$ the differential $d_x \pi$ is well defined and takes values in $\ell^1$ (actually only a finite number of components are non-zero).

For $\pi\in \A\otimes PC^1(\R)$ we define
\begin{equation}\label{eq_lisdefL1}
L(\pi):= - \int_0^1 \,dt \int_E \partial_t\pi (x,t)\,d\mu_t(x) +
\int_E \pi(x,1)\,d\mu_1(x) - \int_E \pi(x,0)\,d\mu_0(x).
\end{equation}
Thus, by Lebesgue dominated convergence theorem one has
\begin{equation}\label{eq_Lpi1}
\begin{aligned}
L(\pi) & = - \lim_{h\to 0}\int_0^{1} \,dt \int_E \frac{\pi(x, t+h)-\pi(x,t)}{h} \,d\mu_t(x) \\
&\qquad\qquad\qquad +
\int_E \pi(x,1)\,d\mu_1(x) - \int_E \pi(x,0)\,d\mu_0(x)\\
&= - \lim_{h\to 0}\int_h^1 \,dt \int_E \frac{\pi(x,t)}{h} \,d(\mu_{t-h}-\mu_t)(x) \\
&\qquad\qquad\qquad
- \lim_{h\to 0} \frac{1}{h}\int_1^{1+h} \,dt \int_E \pi(x,t) \,d\mu_{t-h}(x)+ \int_E \pi(x,1)\,d\mu_1(x)\\
&\qquad\qquad\qquad
+ \lim_{h\to 0}\frac{1}{h}\int_0^{h} \,dt \int_E \pi(x,t) \,d\mu_{t}(x) - \int_E \pi(x,0)\,d\mu_0(x)
\\
&=  - \lim_{h\to 0}\int_h^1 \,dt \int_E \frac{\pi(x,t)}{h} \,d(\mu_{t-h}-\mu_t)(x),
\end{aligned}
\end{equation}
where the boundary terms cancel out because the function $t\mapsto \int_E \pi(x,t) \,d\mu_{t}(x)$ is continuous.

Denoting by $\gamma_{t-h,t}$ an optimal transport plan between $\mu_{t-h}$ and $\mu_t$, we get thus
\begin{align*}
L(\pi) & =  - \lim_{h\to 0}\int_h^1 \,dt \int_{E\times E} \frac{\pi(x,t)-\pi(y,t)}{h} \,d\gamma_{t-h,t}(x,y)\\
& =  \lim_{h\to 0}\int_h^1 \,dt \frac{1}{h}\int_{E\times E} \,d\gamma_{t-h,t}(x,y) \int_0^1 \langle d_x \pi(x+s(y-x),t), y-x\rangle\,ds\\
& = \lim_{h\to 0}\int_h^1 \,dt \frac{1}{h}\int_{E\times E} H(x,y)\|x-y\|_\infty\,d\gamma_{t-h,t}(x,y),
\end{align*}
where
\[
H(x,y) :=\left\{
\begin{array}{cc}
  \displaystyle\int_0^1 \left\langle d_x \pi(x+s(y-x),t), \frac{y-x}{\|y-x\|_\infty}\right\rangle\,ds, & x\neq y, \\
  \| d_x \pi(x,t)\|_1, & x=y,
\end{array}
\right.
\]
and $\langle \cdot, \cdot \rangle$ stands for the duality between
$\ell^1$ and $\ell^\infty$ and $\|\cdot\|_1$ stands for the norm in $\ell^1$.
Therefore, by H\"{o}lder inequality
\begin{align*}
L(\pi) & \leq
\lim_{h\to 0}\int_h^1  \frac{1}{h}W_p(\mu_{t-h},\mu_{t})
\left(\int_{E\times E} H^{p'}(x,y)\,d\gamma_{t-h,t}(x,y)\right)^{1/p'}dt.
\end{align*}
Clearly since for $\pi\in \A\otimes PC^1([0,1])$ the differential map $d_x\pi(\cdot)$ is continuous and bounded, then $H$
is upper semicontinuous and bounded. Since
$\gamma_{t-h,t}\rightharpoonup (\mathrm{Id}\otimes\mathrm{Id})_{\#} \mu_t$ in the narrow sense of measures
as $h\to 0$, we have
\[
\limsup_{h\to 0} \int_{E\times E} H^{p'}(x,y)\,d\gamma_{t-h,t}(x,y)\leq \int_{E} \| d_x\pi(x,t)\|_1^{p'}\,d\mu_t(x).
\]
Further,
\[
\limsup_{h\to 0} \frac{W_p(\mu_{t-h},\mu_{t})}{h}= |\dot{\mu}_t|_p \quad \text{for a.e.\ $t\in [0,1]$,}
\]
and we conclude that
\begin{equation}\label{eq_Ldpi1}
\begin{split}
L(\pi) & \leq
\int_0^1 |\dot{\mu}_t|_p
 \left(\int_{E} \| d_x\pi(x,t)\|_1^{p'}\,d\mu_t(x)\right)^{1/p'} \,dt \\ &  \le  \left(\int_0^1  |\dot{\mu}_t|_p^p \, dt \right)^{1/p} 
\|d_x \pi\|_{L^{p'}(E\times [0,1], \mu_t\otimes dt; \ell^1)}.
\end{split}
\end{equation}

In view of the above inequality~\eqref{eq_Ldpi1} we have that
$\tilde L(d_x\pi) :=L(\pi)$
correctly defines a continuous functional over $\{d_x\pi\colon \pi\in \A\otimes PC^1([0,1])\}$
which can be extended by continuity
to a linear functional on the closure of this set in $L^{p'}(E\times [0,1], \mu_t\otimes dt; \ell^1)$,
with its norm $\| \tilde L\|$ estimated as
\begin{equation}\label{eq_lis_normLup1}
\| \tilde L\|^p \leq \int_0^1  |\dot{\mu}_t|_p^p\, dt.
\end{equation}
By standard duality theory for Lebesgue spaces there is an
$\ell^\infty$-valued map $V$ representing $\tilde L$, and in particular one has
\begin{eqnarray}
\tilde L(d_x\pi) &=& \int_{E\times [0,1]} \langle d_x\pi, V \rangle\, d\mu_t dt \quad \text{for every $\pi \in \A\otimes PC^1([0,1])$,} \label{eq_reprLtild1}\\
\| \tilde L\|^p  &\geq & \int_{E\times [0,1]} \|V\|_\infty^{p} \, d\mu_t dt . \label{eq_reprLtild2}
\end{eqnarray}
In fact,~\eqref{eq_reprLtild1} is true because when $\pi \in \A\otimes PC^1([0,1])$, then
\[
\tilde L(d_x\pi)  = \int_{E\times [0,1]} \langle d_x\pi, V^n\rangle \,d\mu_t  dt
\]
for some $\ell^\infty$-valued map $V^n$ which may be assumed to have only first $n$ nonzero components,
and clearly for $n\leq m$ one has that the first $n$ components of $V^n$ and $V^m$ coincide. This also gives
\[
\| \tilde L\|^p  \geq \sup_n\int_{E\times [0,1]} \|V^n\|_\infty^{p} \, d\mu_t  dt
\]
showing~\eqref{eq_reprLtild2}.
We define then, for $\pi \in \A\otimes PC^1([0,1])$, with the customary abuse of the notation using the same letters for operators and functions,
\begin{equation}\label{eq_defLVt}
(V\pi)(x,t) := \langle d_x\pi(t,x), V_t(x)\rangle, \quad \text{ where $V_t(x):= V(x,t)$.}
\end{equation}

The map $V$ trivially satisfies the Leibniz rule~\eqref{eq_VtLeib1} and~\eqref{eq_VtLip1} with
$\overline{V}  := \| V\|_\infty$, 
since for $\pi \in \A\otimes PC^1([0,1])$  one has $(\lip_{E} \pi)(x,t)=\|d_x \pi (x,t)\|_1$.
Finally, we have
\[
L(\pi)=\tilde L(d_x\pi) = \int_0^1 \,dt \int_E (V_t\pi(t,\cdot))(x) \,d\mu_t(x),
\]
and recalling~\eqref{eq_lisdefL1}, we arrive at the equality
\begin{equation}\label{eq_lis_conteq1}
\int_0^1 \,dt \int_E \left(\partial_t \pi (x,t) + (V_t\pi(t,\cdot))(x) \right)\,d\mu_t(x) =
\int_E \pi(x,1)\,d\mu_1(x) - \int_E \pi(x,0)\,d\mu_0(x)
\end{equation}
for every $\pi\in \A\otimes PC^1([0,1])$,
and in particular, the continuity equation~\eqref{eq_fl2cont1} holds.
Since by Lemma~\ref{lm_mcurr_densc0lip} the algebra $\A$ satisfies~($\A_1$), then by
Proposition~\ref{prop_fl2T'curr1}
the functional $T'$  defined by~\eqref{eq_fl2defT'} is uniquely extendable to a normal  current
over $E\times [0,1]$, and hence Theorem~\ref{th_fl2T'repr2}(ii)
provides the existence of a measure $\eta$ over $C([0,1]; E)$ concentrated
over absolutely continuous curves for which $(e_t)_\# \eta = \mu_t$ for every $t \in [0,1]$
such that~\eqref{eq_fl2repreta1mod1} and~\eqref{eq_fl2repreta1} hold.
We further abuse the notation and write $V_t$ for the vector field extended to the whole $\Lip(E)$ (instead of $\tilde V_t$).

Let now  $\eta$ be an arbitrary such measure.
Since it induces a transport plan $(e_t\otimes e_{t+h})_{\#}\eta$ with marginals $\mu_t$ and $\mu_{t+h}$, by definition of Wasserstein distance we get
\begin{align*}
    W_p(\mu_t,\mu_{t+h})^p &\leq \int_{C([0,1];E)} d(\theta(t),\theta(t+h))^p\, d\eta(\theta)
    \leq \int_{C([0,1];E)} \, d\eta(\theta) \left(\int_t^{t+h}
    |\dot{\theta}|(s)\,ds\right)^p\\
    & \leq  h^p \int_{C([0,1];E)} \, d\eta(\theta) \left(\frac{1}{h}\int_t^{t+h}
    |\dot{\theta}|^p(s)\,ds\right)^p\mbox{ by Jensen inequality}.
\end{align*}
so that  for a.e.\ $t\in [0,1]$  one has
\begin{equation*}
\begin{aligned}
    |\dot{\mu}_t|_p^p
\leq \int_{C([0,1];E)}  |\dot{\theta}|(t)^p\, d\eta(\theta).
\end{aligned}
\end{equation*}
Let now $\mathcal{F}\subset \ell^1$ be a countable family of cylindrical elements of $\ell^1$ dense in the unit ball of $\ell^1$.
Since $\mathcal{F}\subset \A$, we get
\begin{align*}
|\dot{\theta}|(t) =\|\dot{\theta}_w(t)\|_\infty =\sup_{\pi\in \mathcal{F}} \langle \pi,\dot{\theta}_w(t) \rangle =
\sup_{\pi\in \mathcal{F}} \frac{d\,}{dt} \langle \pi, \theta(t)\rangle & = \sup_{\pi\in \mathcal{F}} (V_t\pi)(\theta(t))\quad \text{by~\eqref{eq_fl2repreta1}}\\
&\leq \overline V(\theta(t),t),
\end{align*}
 and therefore
\begin{equation}\label{eq_mudot_est1}
\begin{aligned}
    |\dot{\mu}_t|_p^p
  \leq  \int_{C([0,1];E)} |\overline V(\theta(t),t)|^p\, d\eta(\theta)
  = \int_E |\overline V|^p(x,t)\,d\mu_t(x)
\end{aligned}
\end{equation}
for a.e.\ $t\in [0,1]$.
On the other hand, with the help of~\eqref{eq_reprLtild2} we get
\begin{equation}\label{eq_mudot_est2}
\begin{aligned}
    \int_0^1\,dt \int_E \overline{V} (x,t)^p\,d\mu_t(x) =     \int_0^1\,dt \int_E \|V(x,t)\|_\infty^p\,d\mu_t(x) \leq
    \|\tilde L\|
     \le  \int_0^1 |\dot{\mu}_t|_p^p\,dt.
\end{aligned}
\end{equation}
Therefore, the above inequalities~\eqref{eq_mudot_est1} and~\eqref{eq_mudot_est2} are in fact equalities, which shows~\eqref{eq_lis_metrvel1}.
\end{proof}

\section{Flows of measures induced by a current}

In~\cite{PaoSte14-flow} the following definition has been introduced.
A finite positive Borel measure $\mu$ over a metric space $E$ is called \emph{transportable} by a
current $T\in \M_1(E)$, if there is
\begin{itemize}
\item[(a)] a
($\sigma$-finite) Borel measure $\bar\eta$ over $C(\R^+;E)$ concentrated over
absolutely continuous nonconstant curves representing $T$ without cancelation of mass
in the sense that
\begin{equation}\label{eq_flow1cancmass0}
\begin{aligned}
T(\omega)&=\int_{C(\R^+;E)} \ld\theta\rd (\omega)\,d\bar\eta(\theta), \quad\text{for all $\omega\in D^1(E)$},\\
\MM(T)&=\int_{C(\R^+;E)} \ell(\theta)\,d\bar\eta(\theta),
\end{aligned}
\end{equation}
where the curves (unlike the previous sections) are assumed to be parameterized over $\R^+$ rather than $[0,1]$,
and
\item[(b)]
a finite Borel measure $\sigma$ over $C(\R^+;E)\times\R^+$, such that
\begin{equation}\label{eq_flow1adm1}
\begin{aligned}
\beta_{0\#}\sigma &=\mu,\quad
q_{\#}\sigma \ll \eta,
\quad\text{where}\\
\beta_t &\colon (\theta,s)\in C(\R^+;E)\times\R^+\mapsto \theta(t+s)\in E,\\
q&\colon (\theta,s)\in C(\R^+;E)\times\R^+\mapsto \theta\in C(\R^+;E).
\end{aligned}
\end{equation}
\end{itemize}
The respective flow of the measure $\mu$ is then defined by $\mu_t:=\beta_{t\#}\mu$.
The idea behind this definition is as follows: the flow is seen as a movement of an ensemble of particles, each particle moving along some absolutely continuous curve, which is an ``integral curve'' to the vector field represented by the current $T$. The movement of this ensemble is assumed to be given by a measure $\sigma$ over $C(\R^+;E)\times\R^+$, so that, intuitively, $\sigma(\theta,s)$ stands for the number of particles which move along the curve $\theta\in C(\R^+;E)$ at time $s\in \R^+$. Suppose that we set our timer at $t:=0$ and observe the distribution of the particles. Each particle at this instant is moving along some curve
$\theta\in C(\R^+;E)$ for already some time $s\in \R^+$ and hence
is observed at the position $\beta_0(\theta,s):=\theta(s)\in E$, so that we observe the distribution of particles given by the measure $\mu:=\beta_{0\#}\sigma$.
After time $t\in \R^+$ is elapsed (i.e.\ when our timer shows $t$), we observe the measure $\mu_t:=\beta_{t\#}\sigma$, where
$\beta_t(\theta,s):=\theta(s+t)\in E$. The family $\{\mu_t\}_{t\in \R^+}$ may be viewed thus as a flow of the measure $\mu=\mu_0$.
The ``integral curves'' are exactly the curves charged by the measure $\eta$ representing $T$ as in~\eqref{eq_flow1cancmass0}
(the drawback of this notion is its sensibility to the equivalent choice of the distance in $E$; in fact, even for $E:=\R^n$
the different choices of the norm produce different representations of the same current and hence different flows).
It has been proven in~\cite{PaoSte14-flow} that under some natural conditions on the data, $\mu_t$  satisfies the continuity equation with the family of vector fields
related to the current $T$, and, as we prove in the following
statement, every measure $\eta$ representing this flow which is found by the superposition principle
(Theorem~\ref{thm:sp}) in fact represents without cancelation of mass the normal current which differs from the original one
just by a measure but not by a direction of the vector field.

\begin{corollary}\label{co_fl2Sgam2}
Let $T\in \M_1(E)$, where $E$ is
a Banach space with Radon-Nikodym property and with strictly convex norm $\|\cdot\|$,
$T= V\wedge\mu_T$ for a Borel map $V\colon E\to E$, $\|V(x)\|=1$ for $\mu_T$-a.e.\ $x\in E$,
in the sense that
\[
T(f\,d\pi)= \int_E \langle f(x) \langle V(x), d\pi(x)\rangle\, d\mu_T(x)\quad \text{for $f\,d\pi\in D^1(E), \pi\in Q^1(E)$},
\]
and
$\mu$ be a finite Borel measure
over $E$ transportable by $T$ through a finite Borel measure $\sigma$ over $C(\R^+;E)\times\R^+$ such that
\[
\int_{C(\R^+;E)} \ell(\theta)\,d(q_{\#}\sigma)(\theta) <+\infty,
\]
and $q_{\#}\sigma$ concentrated over curves  $\theta\in C(\R^+;E)$ arclength parameterized over
$[0, \ell(\theta))$ and constant
over $[\ell(\theta),+\infty)$,
the respective flow of $\mu$ being given by the family of measures $\{\mu_t\}$ with $\mu_0=\mu$.
Then the following assertions hold.
\begin{itemize}
\item[(i)] One has
\begin{equation}\label{eq_pde_transportable1}
  \partial_t\mu_t+ \mathrm{div}\, V_t \mu_t =0,\quad\text{ where $V_t(x)=V(x) (1-\varphi_t(x))$},
\end{equation}
in the weak sense,
with $\varphi\colon E\times [0,1]\to [0,1]$,   $\varphi_t(x):=\varphi(x,t)$ a Borel function.
\item[(ii)]
There is a $\gamma\in L^1(E,\mu_T)$
such that, if $V_t$ is as in~\eqref{eq_pde_transportable1},
then
every finite Borel measure $\eta$ over $C(\R^+;E)$ provided by Theorem~\ref{thm:radon-nikodym}
represents the normal current $T\res \gamma$.
\end{itemize}
\end{corollary}

\begin{remark}
Clearly, the statement of Corollary~\ref{co_fl2Sgam2} remains valid (with obvious changes in the ODE),
if $E$ is just a metric space isometrically imbedded in a Banach space with Radon-Nikodym property.
\end{remark}

\begin{proof}
Claim~(i) is part of~\cite[theorem~6.1(B)]{PaoSte14-flow}.
To prove claim~(ii), note that
by Corollary~\ref{co_currconteq0} $\eta$ represents the normal current defined by
\[
T_{\eta} (f\,d\pi):= \int_{\R^+} \,dt\int_E f\, V_t\pi\, d\mu_t,
\]
without cancelation of mass, where $(V_t\pi)(x) := \langle V(x)(1-\varphi_t(x)), d\pi(x)\rangle$ for $\pi\in Q^1(E)$
(in fact, $V_t/|V_t|$) is independent of $t$).
But from remark~6.3 in~\cite{PaoSte14-flow} one has that $|V_t|\mu_t\ll \mu_T$ for a.e.\ $t\in \R^+$, and hence
there is a $\gamma\in L^1(E,\mu_T)$ such that
$T_{\eta} =T\res\gamma$ as claimed in~(ii).
\end{proof}

\section*{Acknowledgments} The authors would like to thank E.~Paolini and L.~Ambrosio for helpful discussions.

\appendix

\section{Auxiliary results on Lipschitz functions}\label{sec:density}

We start with the following observation.
If $E$ is separable, and $x_0\in E$, then the set $\{u\in \Lip_1(\bar B_R(x_0))\colon \|u\|_\infty\leq k\}$ is a bounded subset of $\Lip_b(E)$, hence the weak$^*$ topology of
$\Lip_b(E)$ over it is metrizable.

The lemmata below on Lipschitz functions are by no means optimal, we only collect here what we need in the paper.

\subsection{Approximation by families of Lipschitz functions}

\begin{lemma}\label{lm_mcurr_densd1a}
Let $E$ be a separable metric space, $D\subset E$ be a countable dense set in $E$.
Then
\begin{itemize}
\item[(i)] the family
\[
\left\{\bigwedge_{i=1}^n (c_i+d(\cdot, y_i))\colon c_i\in \R, y_i\in D, i=1,\ldots, n, n\in \N  \right\}
\]
is pointwise dense in $\Lip_1(E)$,
\item[(ii)] the family
\[
\left\{\bigwedge_{i=1}^n (c_i+\lambda_i d(\cdot, y_i))\colon  c_i\in \R, y_i\in D, \lambda_i\in \Q\cap[0,1], i=1,\ldots, n, n\in \N  \right\}
\]
is Lip-dense in $\Lip_1(E)$.
\end{itemize}
\end{lemma}

\begin{proof}
Let $D:=\{y_k\}$ and consider an arbitrary $\pi\in \Lip_1(E)$. Setting
\[
\tilde\pi(x):=\inf_k (\pi(y_k)+d(x,y_k)), \qquad \bar\pi(x):=\inf_k (\pi(y_k)+\Lip\, \pi d(x,y_k)),
\]
we have that $\tilde\pi (x)=\bar\pi(x)=\pi(x)$ since $\tilde\pi (y_k)=\bar\pi(y_k)=\pi(y_k)$ and $\tilde\pi\in \Lip_1(E)$,
$\bar\pi\in \Lip_1(E)$ (it suffices that
$\pi$, $\tilde\pi$ and $\bar\pi$ be continuous).
Thus for such a $\pi$  denoting
\[
\tilde \pi^m (x) :=\bigwedge_{k=1}^m (\pi(y_k)+d(x,y_k)),\qquad \bar \pi^m (x) :=\bigwedge_{k=1}^m (\pi(y_k)+\lambda_k d(x,y_k)),
\]
where $\lambda_k\in \Q\cap [0,1]$, $\lambda_k\searrow \Lip\,\pi$,
one gets $\lim_m \tilde \pi^m(x)=\lim_m \bar\pi^m(x) = \pi(x)$ for every $x\in E$ with $\Lip\,\tilde \pi^m \leq 1$, $\Lip\,\tilde \pi^m \leq \Lip\,\pi\leq 1$,
hence by~\eqref{eq_lsclip1} one has $\lim_m \Lip\,\bar\pi^m=\Lip\,\pi$, concluding the proof.
%
\end{proof}

\begin{lemma}\label{lm_mcurr_densd1b}
Let
$X$ be a separable Banach space. Then for every $y\in X'$ and for $\pi\in \Lip_1(X')$, $\pi(x):=\|x-y\|^\prime$,
where $\|\cdot\|^\prime$ stands for the norm in $X'$,
there is a sequence
\[
\pi^m(x)  = \bigvee_{j=1}^m (c_{j}+\langle a_j, x \rangle) \in \Lip_1(X'),
\]
where $c_{j}\in \R$, $\{a_j\}$ is a dense sequence in the unit ball of $X$,
with $\pi^m(x)\to \pi(x)$ for every $x\in X'$.
\end{lemma}

\begin{proof}
For $\pi(x) =\|x-y\|^\prime$,
one has
$\pi(x)=\sup_j \langle a_j, x\rangle$,
hence denoting
\[
\pi^m(x)  := \bigvee_{j=1}^m \langle a_j, x-y  \rangle = \bigvee_{j=1}^m (c_{j}+\langle a_j, x \rangle)\in \Lip_1(X'),
\]
where $c_{j}=-\langle a_j, y\rangle\in \R$,
we get $\pi^m(x)\to \pi(x)$ for every $x\in X'$ as claimed.
\end{proof}

\begin{lemma}\label{lm_mcurr_densc0lip}
Let
$E$ be a closed set in a separable linear subspace of $\ell^\infty$.
Then for every $\pi\in \Lip_b(E)$ there is a sequence  $\{\pi_k\}\subset C^1(E)\cap \Lip_b(E)\cap \mathrm{Cyl}(E)$
such that $\lim\pi_k=\pi$ pointwise with
\[
\|\pi_{k}\|_\infty\leq  \|\pi\|_\infty, \quad \Lip\,\pi_{k}\leq \Lip\,\pi,\quad \text{and} \quad \limsup_\nu\lip\,\pi_k (x)\leq \lip\,\pi (x)
\]
 for every $x \in E$. The same assertion holds for $E\times \R$ instead of $E$.
\end{lemma}

\begin{proof}
We consider an arbitrary $\pi\in \Lip_b(E)$ extended to a Lipschitz function over $\ell^\infty$.
Set $\pi_k (x):= \pi(P_k(x))$,
where $P_n\colon E\to \ell^\infty$ are finite rank projections defined in Lemma~\ref{lm_muPnapprox1}.
Clearly, $\lim_k\pi_k =\pi$ in the pointwise sense, and
\begin{align*}
\Lip\,(\pi_k\res B_\rho(x)) & = \sup_{\{u,v\}\in B_\rho(x)}   \frac{|\pi_k(u)-\pi_k(v)|}{\|u-v\|}\\
 &\leq \sup_{\{u,v\}\in B_\rho(x)} \frac{|\pi(P_k u)-\pi(P_k v)|}{\|P_k (u-v)\|} \frac{\|P_k (u-v)\|}{\|u-v\|}\\
 &\leq \sup_{\{u,v\}\in B_\rho(x)} \frac{|\pi(P_k u)-\pi(P_k v)|}{\|P_k (u-v)\|} = \sup_{\{u,v\}\in P_k (B_\rho(x))} \frac{|\pi(u)-\pi(v)|}{\|u-v\|}\\
& \leq \sup_{\{u,v\}\in B_\rho(P_k(x))} \frac{|\pi(u)-\pi(v)|}{\|u-v\|}=\Lip\,(\pi\res B_\rho(P_k(x)),
\end{align*}
so that taking the limit in the above estimate as $\rho\to 0^+$, we get
$\lip\,\pi_k\leq \lip\,\pi\circ P_k$, hence
$\limsup_k\lip\,\pi_k\leq \lip\,\pi$ pointwise,
and letting $\rho\to \infty$, we get
$\Lip\,\pi_k\leq \Lip\,\pi$.

Continuously differentiable cylindrical approximations of $\pi$ are now obtained by using convolutions of $\pi_k$
with some smooth approximate
identity, which we choose to have compact support. Indeed, in this way each $\pi_k$ is a limit of a pointwise convergent
sequence of uniformly bounded smooth functions of finite number of variables
$\{ \pi_{k n}\}_n$ with $\Lip\, \pi_{kn}\leq \Lip\, \pi_k$ (hence also Lip-converging)
and $\|\pi_{kn}\|_\infty\leq \| \pi_k \|_\infty$.
Recalling that Lip-convergence is metrizable over the set of functions with uniformly bounded Lipschitz constants, we may extract a diagonal sequence
$\hat\pi_k :=\pi_{kn(k)}$ Lip-convergent to $\pi$. Since
$\pi_{kn}$ are obtained by convolutions of $\pi_k$ with a kernel with compact support, we
have for every $\rho>0$ and for some $\bar n(\rho)\in \N$ that
\[
\lip \,\pi_{kn} (x) \leq
\Lip\,(\pi_{kn}\res B_{\rho}(x)) \leq \Lip\,(\pi_k\res B_{2\rho}(x)) \quad \text{for every $n \ge \bar n(\rho)$.}
\]
Recalling that
\[
\Lip\,(\pi_k\res B_{2\rho}(x))\leq \Lip\,(\pi\res B_{2\rho}(P_k(x))),
\]
and that $B_{2\rho}(P_k(x)) \subset  B_{3\rho}(x)$
for all  $k\in \N$ sufficiently large (depending on $x \in E$ and on $\rho$), we have
the existence of some $\bar k(x,\rho)\in \N$ such that
\[
\lip\, \hat{\pi}_{k} (x) = \lip \,\pi_{kn(k)} (x) \le \Lip\,(\pi\res B_{3\rho}(x)) \quad \text{for every $k \ge \bar k(x,\rho)$},
\]
and hence $\limsup_k   \lip\, \hat{\pi}_{k} (x) \le  \Lip\,(\pi\res B_{3\rho}(x))$, and, taking the limit in the latter
relationship as $\rho\to 0^+$, we get
$\limsup_k   \lip\, \hat{\pi}_{k} (x)\leq \lip\, \pi (x)$ as claimed.

The last assertion (about $E\times \R$) is proven exactly in the same way but with the operators
$(P_n, \mbox{Id})\colon (x,t)\in E\times \R\mapsto (P_n x, t) \in \ell^\infty\times \R$ instead of $P_n$, thus
concluding the proof.
\end{proof}

\subsection{Tensorization results}

In this subsection, we assume by default that $E\times [0,1]$ be equipped with the distance $d_2$.
The main auxiliary assertion of this section is the following tensorization lemma.

\begin{lemma}\label{lm_approxtens1}
Let $E$ be a separable metric space, $\A \subset \Lip_b(E)$ be an algebra of functions, and suppose that for every
$\pi\in \Lip_b(E\times [0,1])$ there is a sequence
$\{\tilde\pi_m\}\subset \Lip_b(E\times [0,1])$
of uniformly bounded functions with uniformly bounded Lipschitz constants
satisfying $\tilde\pi_m(\cdot,t)\in \A$ for every $t\in [0,1]$, $m\in \N$, and $\lim_m \tilde\pi_m =\pi$ pointwise.
Suppose that
for some $C>0$
\begin{eqnarray}\label{eq_Liplocconv1a1}
\text{either}& \limsup_m \lip\, \tilde\pi_m(x,t) & \leq C\lip\, \pi(x,t)\quad\text{for all $(x,t)\in E\times [0,1]$},\\
\text{or}& \limsup_m \Lip\, \tilde\pi_m &\leq C\Lip\, \pi.
\label{eq_Liplocconv1a0}
\end{eqnarray}
Then
there is a sequence  $\{\pi_j\}\subset \A\otimes PC^1([0,1])$
having the same properties as $\{\tilde \pi_m\}$ (possibly with a different constant $C>0$ which depends only on the original one).
\end{lemma}

\begin{remark}\label{rm_approxtens1_metr}
The assertion of Lemma~\ref{lm_approxtens1} does not depend on the equivalent choice of the distance
in $E\times [0,1]$ (which only affects the constant $C>0$).
\end{remark}

\begin{proof}
In view of Remark~\ref{rm_approxtens1_metr} we assume $E\times [0,1]$ to be equipped with the distance $d_1$.
For  and $n \in\N$, and for each $k=0,\ldots, n$
we let $\tau_k := k/n \in [0,1]$.
Given a $\pi\in \Lip(E\times [0,1])$,
 and $\{\tilde\pi_m\}$ as in the statement being proven,
we approximate $\pi$ by linear interpolations with respect to the variable $t \in [0,1]$,
substituting $\pi(x,\tau_k)$ with $\tilde\pi_m(x,\tau_k)$,
introducing the
functions
\[
\pi_{nm} (x,s) := \sum_{k=0}^{n-1} \1_{[\tau_k,  \tau_{k+1} ) }(s) \left(  (k+1 - n s)
\tilde\pi_m(x, \tau_k)
+ (n s -k) 
\tilde\pi_m (x,  \tau_{k+1} )
\right).
\]
Clearly $\|\pi_{nm}\|_\infty \leq \|\tilde\pi_m\|_\infty$, so that all the functions $\pi_{nm}$ are uniformly bounded.
Notice  that
\[
\pi_{nm} (x,s) = (1 - n s )^+
\tilde\pi_m(x, \tau_0)
+ \sum_{k=1}^{n-1}  (1- |n s -k | )^+
\tilde\pi_m(x,  \tau_{k})
+ ( 1 + n s -n)^+
\tilde\pi_m(x,  \tau_{n}),
\]
and hence $\pi_{nm}\in \A
\otimes \Lip([0,1])$.

We estimate now $\Lip\, \pi_{nm}$.
Namely,
\begin{equation}\label{eq_convergence-as-algebra-dense0}
\begin{split}
\left|\pi_{nm} (x,s) - \pi_{nm}(y, s)\right| & \le |k+1 - n s| \cdot
| \tilde\pi_m(x, \tau_k) - \tilde\pi_m(y, \tau_k) |
\\
&\qquad\qquad
+ |n s -k| \cdot
| \tilde\pi_m(x,  \tau_{k+1})  - \tilde\pi_m(y,  \tau_{k+1}) |\\
& \le (k+1- n s + n s - k    ) \Lip\, \tilde\pi_m \,  d(x,y).
\\
&\leq  \Lip\, \tilde\pi_m\, d(x,y).
\end{split}
\end{equation}
On the other hand,
if $\{t,s\}\subset [\tau_k, \tau_{k+1})$, one has
\begin{equation}\label{eq_lipcnst1}
\begin{aligned}
|\pi_{nm} (y, s) - \pi_{nm}(y, t)| &\le  n |s-t|\cdot |\tilde\pi_m(y, \tau_k) - \tilde\pi_m(y,  \tau_{k+1} ) |\\
& \leq \Lip\,\tilde\pi_m n |s-t|\cdot |\tau_k-\tau_{k+1}  |= \Lip\,\tilde{\pi}_m |s-t|.
\end{aligned}
\end{equation}
For an arbitrary $n\in\N$ one has
\begin{equation}\label{eq_lipcnst2}
\begin{aligned}
| \pi_{nm} (y, s) - \pi_{nm}(y, t)| & \le |\pi_{nm} (y, s) - \pi_{nm}(y, \tau_{k+1})| +\\
&\qquad
\sum_{i={k+1}}^{h-1} |\pi_{nm} (y, \tau_{i}) - \pi_{nm}(y, \tau_{i+1}) | + \\
&\qquad |\pi_{nm} (y, \tau_{h}) - \pi_{nm}(y, t ) |
\le  \Lip\,\tilde\pi_m |s-t|.
\end{aligned}
\end{equation}
This together with~\eqref{eq_convergence-as-algebra-dense0} proves
\begin{equation}\label{eq_apprLip1}
\Lip\, \pi_{nm}\leq \Lip\, \tilde \pi_m,
\end{equation}
hence, in particular, all the Lipschitz constants $\Lip\,\pi_{nm}$ are uniformly bounded.

In view of pointwise convergence of $\tilde\pi_m$ and of uniform boundedness of their Lipschitz constants, for every compact $K\subset E$ and for every $j\in \N$ there is an $m(j,K)\in \N$ such that
\[
|\tilde\pi_m(x,t)-\pi(x,t)|\leq 1/j \quad\mbox{for all $(x,t)\in K\times [0,1]$},
\]
when $m\geq m(j,K)$, because convergence of $\tilde\pi_m$ is uniform over compact sets.
Since $E$ is separable, there is an increasing sequence $\{K_j\}$ of compact sets such that $E=\overline{\cup_j K_j}$. Set now
$\bar m(j):= m(j, K_j)$ and
let $n(j)\in \N$ satisfy $\lim_j n(j)=+\infty$.
Denote
$\pi_j:=\pi_{n(j)\bar m(j)}$.
For every $x \in \cup_j K_j$, $t \in [0,1]$, if $s \in [\tau_k, \tau_{k+1})$, since
$\Lip\,\pi_j\leq C$ for some $C>0$ independent of $j$, we get
\begin{equation}\label{eq_ptconvLip1}
\begin{aligned}
|\pi_j(x,s) - \pi(x,s)| & \le |\pi_j(x,s) - \pi_j(x,\tau_k)| + |\pi_j(x,\tau_k) - \pi(x,\tau_k)|
 + |\pi(x,\tau_k) - \pi(x,s)|
 \\
 & \leq C|s-\tau_k| + |\pi_{n(j) \bar m(j)}(x,\tau_k) - \pi(x,\tau_k)| +
 C |s-\tau_k|\\
 & = 2C|s-\tau_k| + |\tilde \pi_{\bar m(j)}(x,\tau_k) - \pi(x,\tau_k)|
 \le 2C |s - \tau_k| + 1/n(j) \\
 & \le 2C/n(j) + 1/n(j)
\end{aligned}
\end{equation}
for all $j\geq \bar{\jmath}$ with $\bar{\jmath}=\bar{\jmath}(x)\in \N$ is the smallest $j\in \N$ such that $x\in K_j$.
Thus, by density of $\cup_j K_j$ in $E$ and again by uniform boundedness of
$\{\Lip\,\pi_j\}$ we have $\lim_j\pi_j =\pi$ pointwise.
Now~\eqref{eq_Liplocconv1a0} with $\{\pi_j\}$ instead of $\{\tilde\pi_m\}$ follows from~\eqref{eq_apprLip1}.

Note that given a $\rho>0$ and letting $\tilde n(\rho):=\lfloor 1/2\rho\rfloor +1$ (so that
$1/\tilde n <2\rho)$, we have for every $m\in \N$, $n\geq \tilde n (\rho)$, $(x,t)\subset E\times [0,1]$
and $s\in B_\rho(t)$, $y\in B_\rho(x)$ the estimates
\begin{eqnarray*}
\Lip\, (\pi_{nm}\res (B_\rho(x)\times\{s\})) &\leq \Lip\, (\tilde \pi_m\res (B_\rho(x)\times B_{2\rho}(t)))
& \text{by~\eqref{eq_convergence-as-algebra-dense0},}\\
\Lip\, (\pi_{nm}\res (\{y\} \times B_\rho(t)) ) &\leq \Lip\, (\tilde \pi_m\res (B_\rho(x)\times B_{2\rho}(t)))
&\text{by~\eqref{eq_lipcnst1} and~\eqref{eq_lipcnst2}},
\end{eqnarray*}
which imply
\[
\Lip\, (\pi_{nm}\res (B_\rho(x)\times B_\rho(t))) \leq 2 \Lip\, (\tilde \pi_m\res (B_{2\rho}(x)\times B_{2\rho}(t))).
\]
Now, for every $j\in \N$ take a $\rho_j>0$ such that
\[
\Lip\, (\tilde \pi_{\bar m(j)}\res (B_{2\rho_j}(x)\times B_{2\rho_j}(t)))\leq 2\lip\, \tilde \pi_{\bar m(j)}(x,t),
\]
which gives
\[
\Lip\, (\pi_{n\bar m(j)}\res (B_{\rho_j}(x)\times B_{\rho_j}(t))) \leq 2\lip\, \tilde \pi_{\bar m(j)}(x,t)
\]
whenever $n\geq \tilde n(\rho_j)$. Taking now $n(j)\in \N$ with $n(j)\geq \tilde n(\rho_j)$ and $\lim_j n(j)=+\infty$,
and setting as above $\pi_j:= \pi_{n(j)\bar m(j)}$, it remains to verify~\eqref{eq_Liplocconv1a1} with $\{\pi_j\}$ instead of $\{\tilde\pi_m\}$. The latter holds in view of
\begin{align*}
\limsup_j \lip\, \pi_j(x,t) & = \limsup_j \lip\, \pi_{n(j)\bar m(j)}(x,t) \\
&\leq
\limsup_j 2 \Lip\, (\pi_{n(j)\bar m(j)}\res (B_{\rho_j}(x)\times B_{\rho_j}(t)))\\
& \leq 4\limsup_j\lip\, \tilde \pi_{\bar m(j)}(x,t) \leq 4 \limsup_m\lip\, \tilde \pi_m(x,t) \leq 4C \lip\pi(x,t)
\end{align*}
for all $(x,t)\in E\times [0,1]$ and~\eqref{eq_Liplocconv1a1}.
\end{proof}

\begin{remark}\label{rem_approxtens2}
Under the same conditions on $E$ and on the distance in $E\times [0,1]$ as in Lemma~\ref{lm_approxtens1}, if for every open
$U\subset E\times [0,1]$ and every
$\pi\in \Lip_b(E\times [0,1])$ with $\supp\, \pi\subset U$
there is a sequence
$\{\tilde\pi_m\}\subset \Lip_b(E\times [0,1])$ as in the statement of Lemma~\ref{lm_approxtens1}
satisfying
$\supp\, \tilde\pi_m\subset U$
instead of~\eqref{eq_Liplocconv1a1} and~\eqref{eq_Liplocconv1a1},
then still
there is a sequence  $\{\pi_j\}\subset \A\otimes PC^1([0,1])$
having the same properties as $\{\tilde \pi_m\}$. The proof of this fact is just identical to the first part
of the proof of Lemma~\ref{lm_approxtens1}, combined with the observation that
$\supp\, \tilde\pi_m\subset U$ implies $\supp\, \pi_{nm}\subset U$ for all sufficiently large $n\in \N$,
i.e.\ for all $n\geq \bar n(m,U)$. In the construction of $\pi_j$ one takes then $n(j)\geq \bar n(\bar m(j),U)$.
\end{remark}

\section{Auxililary lemmata on currents and derivations}

Throughout this section by default $E$ is assumed to be a complete metric space.
If 
$T\in \M_1(E)$, the we call a family $\mathcal{F}\subset \Lip_1(E)$
\emph{realizing the mass} of $T$, if $\mu_T = \bigvee\left\{\mu_{T\res d\pi}\colon \pi\in \mathcal{F}\right\}$.

\begin{lemma}\label{lm_massestV1}
Let 
$\mu$ be a finite positive Borel measure over $E$, and
$T\in \M_1(E)$ be given by
\[
T(f\,d\pi)=\int_E fV\pi\,d\mu,
\]
where $V\colon \Lip(E)\to L^1(\mu)$ is a linear operator, satisfying the Leibniz property and
\[
|(V\pi)(x)|\leq \bar V(x) (\lip\, \pi)(x)
\]
for all $\pi\in \Lip(E)$
and  $\mu$-a.e.\ $x\in E$ for some nonnegative
$\bar V\in L^1(E,\mu)$.
Then $\mu_T=|V|\mu$, and in particular,
\[
\MM(T)=\int_E |V|\,d\mu
\]
for some nonnegative function $|V| \in L^1(E,\mu)$ with $|V|(x)\leq \bar V(x)$ for $\mu$-a.e.\ $x\in E$.
Moreover,
\begin{itemize}
  \item[(i)] $|V|$ is the minimum function (up to $\mu$-a.e.\ equality) $\beta \in L^1(E,\mu)$ satisfying
\begin{equation}\label{eq_hatVlip1}
|(V\pi)(x)|\leq \beta(x) (\lip\, \pi)(x)
\end{equation}
for all $\pi\in \Lip(E)$;
  \item[(ii)] $|V\pi|\leq |V|$
for every $\pi\in \Lip_1(E)$;
   \item[(iii)] for every countable family $\mathcal{F}\subset \Lip_1(E)$ realizing the mass of $T$
one has
\[
|V|(x) =
\sup_{\pi\in \mathcal{F}} (V\pi)(x).
\]
\end{itemize}
\end{lemma}

\begin{proof}
Since
\begin{equation}\label{eq_hatVlip2}
|T(f\,d\pi)| \leq \int_E |f|\, \left|V\pi\right|\,d\mu \leq \Lip\,\pi \int_E |f|\, \bar V\,d\mu,
\end{equation}
then $\mu_T\leq \bar V\, d\mu \ll \mu$. We may define now $|V|$ as the Radon-Nikodym derivative
\[
|V| := \frac{d\mu_T}{d\mu},
\]
thus getting automatically $|V|(x)\leq \bar V(x)$ for $\mu$-a.e.\ $x\in E$.
If $ \beta $ satisfies~\eqref{eq_hatVlip1} then
using $ \beta $ instead of $\bar V$ in~\eqref{eq_hatVlip2}, we get
$|V|(x)\leq \beta (x)$ for $\mu$-a.e. $x$ proving~(i).
Further, from~\eqref{eq_hatVlip2} with $|V|$ instead of $\bar V$ and $\pi\in \Lip_1(E)$ we get
\[
\int_E f\, \left|V\pi\right|\,d\mu \leq \int_E f\, |V|\,d\mu
\]
for all nonnegative $f\in L^1(E,\mu)$, which proves~(ii).
Finally, since $V\pi$ is the Radon-Nikodym derivative of $\mu_{T\res d \pi}$ with respect to $\mu$, it follows that
$ \sup_{\pi\in \mathcal{F}} (V\pi)(x)$ is $\mu$-a.e.\ equal to the density of  $\bigvee\left\{\mu_{T\res d\pi}\colon \pi\in \mathcal{F}\right\}$ with respect to $\mu$, i.e.\ $|V|$, which proves~(iii).
\end{proof}

The statements below give several examples of families realizing the mass of every current.

\begin{lemma}\label{lm_mcurr_linseq1a}
Let $T\in \M_1(E)$. If $\mathcal{F}\subset \Lip_1(E)$ is such that
\begin{itemize}
\item[(A)]  either 
there is a family $\tilde{\mathcal{F}}\subset \Lip_1(E)$ realizing the mass of $T$ such that
for every $\pi\in \tilde{\mathcal{F}}$
there is a sequence of functions $\{\pi_\nu\}\subset \mathcal{F}$ satisfying $\lim_\nu\pi_\nu =\pi$ pointwise over 
$E$,
\item[(B)]
or the family
\[
\mathcal{F}':=\left\{\bigwedge_{i=1}^n (c_i+\pi_i)\colon \pi_i\in \mathcal{F}, i=1,\ldots, n, n\in \N, c_i\in \R \right\}
\]
realizes the mass of $T$.
\end{itemize}
Then 
$\mathcal{F}$ realizes the mass of $T$.
\end{lemma}

\begin{remark}\label{rm_Lip1_realmass1}
By Proposition~2.7 from~\cite{AmbrKirch00}
$\tilde{\mathcal{F}}:=\Lip_1(E)$ realizes the mass of every current $T$, and hence in particular, by the above
Lemma~\ref{lm_mcurr_linseq1a}(A) every family $\mathcal{F}\subset \Lip_1(E)$ dense with respect to pointwise convergence
in $\Lip_1(E)$ realizes the mass of every current.
\end{remark}

\begin{remark}\label{rm_Lip1_realmass2}
It is immediate form the proof that Lemma~\ref{lm_mcurr_linseq1a}(B) remains true, if in the definition of $\mathcal{F}'$
one changes $\wedge$ to $\vee$.
\end{remark}

\begin{proof}
It is enough to show that for every open set $U\subset E$ and for every $\varepsilon>0$
there is a sequence of bounded Borel functions
$\{f_i\}$ and a sequence $\{\pi_i\}\subset \mathcal{F}$, such
that
\begin{equation}\label{eq_mcurr_seq1b}
\sum_i T(f_i\,d\pi_i) \geq \mu_T(U)-\varepsilon,
\end{equation}
where $\sum_i f_i \leq \1_U$, $\pi_i\in \mathcal{F}$.

In case~(A) we know that~\eqref{eq_mcurr_seq1b} is true for $\pi_i\in \tilde{\mathcal{F}}$, and
using the continuity property of currents and the assumption on $\mathcal{F}$,
one finds for each $i \in \N$ an $m=m(i)$ such that
\begin{equation}\label{eq_mcurr_seq1c}
T(f_i\,d\pi_i^m) \geq T(f_i\,d\pi_i)-\varepsilon/2^i,
\end{equation}
where $\pi_i^m \in \mathcal{F}$. Thus from~\eqref{eq_mcurr_seq1b} and~\eqref{eq_mcurr_seq1c}
we get
$
\sum_i T(f_i\,d\pi_i^m) \geq \mu_T(U)-2\varepsilon,
$
showing the claim.

In case~(B) we have~\eqref{eq_mcurr_seq1b}
for $\pi_i\in \mathcal{F}'$ with
\[
\pi_i = \bigwedge_{j=1}^k (c_j^i+g_j^i), \qquad g_j^i\in \mathcal{F}, c_j^i\in \R, j=1,\ldots, k
\]
and $k=k(i)\in \N$.
In view of the chain rule and the locality property of currents one has
\[
T(f_i\,d\pi_i) =\sum_{j=1}^k T(\1_{A_j^i} f_i\,d(c_j^i+g_j^i))=\sum_{j=1}^k T(\1_{A_j^i} f_i\,dg_j^i),
\]
where
\begin{align*}
A_1^i &:=\{x\in E\colon c_1^i+g_1^i(x)\leq g_j^i(x), j=1,\ldots, n\},\\
A_l^i & :=\{x\in E\colon c_l^i+g_l^i(x)\leq g_j^i(x), j=1,\ldots, n\}\setminus\bigcup_{\nu=1}^{l-1} A_\nu^i, \qquad l=2,\ldots, k,
\end{align*}
and therefore
$
\sum_{i,j} T(\1_{A_j^i}f_i\,dg_j^i) \geq \mu_T(U)-2\varepsilon,
$
showing the claim.
\end{proof}

\begin{corollary}\label{co_mcurr_densd1c}
Let
$X$ be a separable Banach space and $T\in \M_1(X')$.
Then every family $\mathcal{F}$ of linear functionals over $X'$ from $X$ dense in the unit ball realizes the mass of $T$.
\end{corollary}

\begin{proof}
Let
$\{y_k\}\subset \supp T$ stand for a countable dense set in $\supp T$.
By Remark~\ref{rm_Lip1_realmass1} and
Lemma~\ref{lm_mcurr_densd1a}(i)
we have that the family
\[
\mathcal{F}':=\left\{\bigwedge_{i=1}^n (c_i+\pi_i)\colon \pi_i(x)=\|x-y_i\|^\prime, i=1,\ldots, n, n\in \N, c_i\in \R \right\},
\]
where $\|\cdot\|^\prime$ stands for the norm in $X'$,
realizes the mass of $T$, and hence by Lemma~\ref{lm_mcurr_linseq1a}(B)
the family
\[
\tilde{\mathcal{F}}:=\left\{\pi_i\colon \pi_i(x)=\|x-y_i\|^\prime, i\in \N\right\}
\]
also realizes the mass of $T$.
But
for each $\pi_k\in \tilde{\mathcal{F}}$,
$\pi_k (x) =\|x-y_k\|^\prime$,
one has by Lemma~\ref{lm_mcurr_densd1b} that $\hat\pi^m(x)\to \pi(x)$ for every $x\in X'$ for
\[
\hat\pi^m(x) = \bigvee_{j=1}^m (c_{kj}+\langle a_j, x \rangle)\in \Lip_1(E),
\]
where $c_{kj}\in \R$, $a_j\in \mathcal{F}$. Therefore, applying Lemma~\ref{lm_mcurr_linseq1a}(A), we get
that now the family
\[
\mathcal{F}'':=\left\{\bigvee_{j=1}^m (c_i+\langle a_j, x \rangle)\colon a_j\in \mathcal{F}, i=1,\ldots, n,\quad n\in \N, c_i\in \R\right\}
\]
realizes the mass of $T$, so that applying Lemma~\ref{lm_mcurr_linseq1a}(B) with $\mathcal{F}''$ instead of $\mathcal{F}'$
(and having in mind Remark~\ref{rm_Lip1_realmass2}), we get the desired claim.
\end{proof}

\section{Weak$^*$ integral and Jensen inequality}\label{sec_wkstInt}

Let $(\Omega,\Sigma,\mu)$ be a finite measure space (the measure $\mu$ defined over a $\sigma$-algebra $\Sigma$ is positive and finite), $X$ be a Banach space equipped with the norm $\|\cdot\|$.
Assume that the function $f\colon \Omega\to X'$ is weak$^*$ scalarly measurable with respect to $\Sigma$ in the sense that
so are real valued functions $\langle \psi, f(\cdot)\rangle$ for every $\psi\in X$ and the norm
$\|f(\cdot)\|\in L^1(\Omega,\mu)$.
Clearly, then $\langle \psi, f(\cdot)\rangle\in L^1(\Omega,\mu)$ for every $\psi\in X$.
The formula
\[
\langle \psi, l\rangle:=\int_\Omega \langle \psi, f(\omega)\rangle\,d\mu(\omega)
\]
for every $\psi \in X$
defines therefore a linear functional over $X$ which is bounded, because
\[
|\langle \psi, l\rangle |\leq \|\psi\| \int_\Omega \|f\|\,d\mu.
\]
In other words, $l\in X'$ and we will call it the weak$^*$ integral of $f$, writing
\[
l=:\int_\Omega f(\omega)\,d\mu(\omega).
\]
Note that
\begin{itemize}
  \item[(1)] for every positive $\hat\mu\leq\mu$ one has the existence of  $\int_\Omega f\,d\hat\mu$;
  \item[(2)] one has
  \begin{equation}\label{eq_ineqJens_norm1}
    \left\|\int_\Omega f\,d\mu\right\|    \leq \int_\Omega\left\| f\right\|d\mu ,
  \end{equation}
  because
\begin{align*}
\left\|\int_\Omega f\,d\mu\right\|    &= \sup_{\psi\in X, \|\psi\|=1}\left\langle\psi, \int_\Omega f\,d\mu\right\rangle
=\sup_{\psi\in X, \|\psi\|=1}\int_\Omega \left\langle\psi,  f\right\rangle\,d\mu \\
        & \leq \int_\Omega \sup_{\psi\in X, \|\psi\|=1} \left\langle\psi,  f\right\rangle\,d\mu =\int_\Omega \left\|f\right\|d\mu .
\end{align*}
\end{itemize}

The following version of the converse statement to the Jensen inequality is valid.

\begin{lemma}\label{lm_revJensen_wkst1}
Suppose that the norm $\|\cdot\|^\prime$ is strictly convex. If
\begin{equation}\label{eq_revJensen1a}
\left\|
\int_\Omega f(\omega)\,d\mu(\omega)
\right\|=
\int_\Omega \left\| f(\omega) \right\|\,d\mu(\omega),
\end{equation}
then
either $f=0$ or for every $\psi\in X$ one has
\begin{equation}\label{eq_revJensen1b}
\langle \psi, f(\omega)\rangle =
\frac{\int_\Omega \langle \psi, f(\omega')\rangle \,d\mu(\omega')}{\int_\Omega \|f(\omega')\| \,d\mu(\omega')} \left\| f(\omega) \right\|
\end{equation}
for $\mu$-a.e.\ (possibly depending on $\psi$)  $\omega\in \Omega$, i.e.\ there is a set $N_\psi\subset\Sigma$ such
$\mu(N_\psi)=0$ and the above equality is valid for all $\omega\in \Omega\setminus N_\psi$.
\end{lemma}

\begin{remark}\label{rem_bochn_conv1}
If under the conditions of the above lemma $f$ is Bochner integrable, then either $f=0$ or
\begin{equation}\label{eq_revJensen1c}
f(\omega)=
\frac{\int_\Omega f(\omega')\,d\mu(\omega')}{\int_\Omega \|f(\omega')\| \,d\mu(\omega')} \left\| f(\omega) \right\|
\end{equation}
for $\mu$-a.e.\ $\omega\in \Omega$ as follows from~\cite{PaoSte14-flow}[lemma~B.1].
\end{remark}

\begin{proof}
Assume that $f\neq 0$.  Note that for every positive $\mu_1\leq\mu$
and $\mu_2:=\mu-\mu_1\leq\mu$ one has
\begin{equation}\label{eq_revJensen2a}
\begin{aligned}
\left\|
\int_\Omega f\,d\mu
\right\| &= \left\|\int_\Omega f\,d\mu_1 + \int_\Omega f\,d\mu_2\right\|
\leq \left\|\int_\Omega f\,d\mu_1\right\| + \left\|\int_\Omega f\,d\mu_2\right\|
\\
& \leq \int_\Omega \|f\|\,d\mu_1 + \int_\Omega \|f\|\,d\mu_2
= \int_\Omega \left\| f \right\|\,d\mu =\left\|
\int_\Omega f\,d\mu
\right\|,
\end{aligned}
\end{equation}
so that all the above inequalities are in fact equalities.
Since $\|\cdot\|^\prime$ is strictly convex, this gives in particular
\begin{equation}\label{eq_revJensen2b}
\int_\Omega f\,d\mu_1 =\lambda \int_\Omega f\,d\mu_2
\end{equation}
for some $\lambda\geq 0$.
Further, from~\eqref{eq_revJensen2a}, recalling~\eqref{eq_ineqJens_norm1}, one gets
$\left\|
\int_\Omega f\,d\mu_i
\right\|=
\int_\Omega \left\| f \right\|\,d\mu_i
$, 
$i=1,2$,
and hence
\begin{equation}\label{eq_revJensen2c}
\int_\Omega \|f\|\,d\mu_1 =\lambda \int_\Omega \|f\|\,d\mu_2
\end{equation}
Writing
\begin{align*}
\int_\Omega f\,d\mu_2
=\int_\Omega f\,d\mu - \int_\Omega f\,d\mu_1,
\qquad
\int_\Omega \|f\|\,d\mu_2
=\int_\Omega \|f\|\,d\mu - \int_\Omega \|f\|\,d\mu_1,\\
\end{align*}
from~\eqref{eq_revJensen2b} and~\eqref{eq_revJensen2c} we get
\begin{align*}
(1+\lambda) \int_\Omega f\,d\mu_1
=\lambda \int_\Omega f\,d\mu,
\qquad
(1+\lambda) \int_\Omega \|f\|\,d\mu_1
=\lambda \int_\Omega \|f\|\,d\mu.
\end{align*}
Since $\lambda\neq -1$ (otherwise the above relationship would give $\int_\Omega \|f\|\,d\mu=0$, hence
$f=0$
contrary to the
assumption), then
\begin{align*}
\int_\Omega f\,d\mu_1  & =\lambda \int_\Omega f\,d\mu,
\qquad
\int_\Omega \|f\|\,d\mu_1
=\lambda \int_\Omega \|f\|\,d\mu
\end{align*}
for some (different) $\lambda\geq 0$. Taking an arbitrary $\psi\in X$ and letting $\mu_1:=\mu\res e$ be an arbitrary $e\subset \Sigma$,
we rewrite the above equalities as
\begin{align*}
\int_e \langle \psi, f\rangle\,d\mu
=\lambda \int_\Omega \langle \psi, f\rangle\,d\mu,
\qquad
\int_e \|f\|\,d\mu  
=\lambda \int_\Omega \|f\|\,d\mu
\end{align*}
(note that, of course, $\lambda$ depends on $e$), thus getting
\[
\int_e \langle \psi, f\rangle\, d\mu =
\frac{\int_e \langle \psi, f\rangle \,d\mu}{\int_\Omega \|f\| \,d\mu}
\int_\Omega \langle \psi, f\rangle\, d\mu,
\]
which implies the equality of measures
\[
 \langle \psi, f(\cdot)\rangle \mu =
\frac{\int_\Omega \langle \psi, f\rangle \,d\mu}{\int_\Omega \|f\| \,d\mu}
\|f(\cdot)\|\mu.
\]
The latter provides~\eqref{eq_revJensen1b}.
\end{proof}

\bibliographystyle{plain}
\bibliography{mathopt}

\end{document}